\documentclass[leqno,10pt]{article}
\usepackage{amssymb,amsmath}
\usepackage[all]{xy}

\newenvironment{proof}{{\bf Proof. }}{\par}{\bigskip}
\newtheorem{theo}{Theorem}[section]
\newtheorem{defi}[theo]{Definition}

\newtheorem{assum}[theo]{Assumption}

\newtheorem{lem}[theo]{Lemma}
\newtheorem{prop}[theo]{Proposition}

\newtheorem{rem}[theo]{Remark}
\newtheorem{coro}[theo]{Corollary}

\newtheorem{exam}[theo]{Example}

\newcommand{\ugot}{\ensuremath{\mathfrak{u}}}
\newcommand{\kgot}{\ensuremath{\mathfrak{k}}}
\newcommand{\ggot}{\ensuremath{\mathfrak{g}}}

\newcommand{\Acal}{\ensuremath{\mathcal{A}}}
\newcommand{\Ccal}{\ensuremath{\mathcal{C}}}
\newcommand{\Fcal}{\ensuremath{\mathcal{F}}}
\newcommand{\Ecal}{\ensuremath{\mathcal{E}}}

\newcommand{\Hcal}{\ensuremath{\mathcal{H}}}

\newcommand{\Kcal}{\ensuremath{\mathcal{K}}}
\newcommand{\Lcal}{\ensuremath{\mathcal{L}}}
\newcommand{\Pcal}{\ensuremath{\mathcal{P}}}

\newcommand{\Ucal}{\ensuremath{\mathcal{U}}}
\newcommand{\Vcal}{\ensuremath{\mathcal{V}}}

\newcommand{\Nbb}{\ensuremath{\mathbb{N}}}
\newcommand{\Cbb}{\ensuremath{\mathbb{C}}}
\newcommand{\Rbb}{\ensuremath{\mathbb{R}}}
\newcommand{\Zbb}{\ensuremath{\mathbb{Z}}}

\newcommand{\A}{\ensuremath{\mathbb{A}}}
\newcommand{\F}{\ensuremath{\mathbf{F}}}

\newcommand{\e}{\operatorname{e}}

\newcommand{\mdr}{\ensuremath{\hbox{\scriptsize \rm mean-dec-rap}}}

\newcommand{\KK}{\ensuremath{{\mathbf K}^0}}

\newcommand{\s}{\ensuremath{\rm S}}
\newcommand{\Cr}{\ensuremath{\hbox{\rm Cr}}}
\newcommand{\f}{\ensuremath{\mathcal{C}^{\infty}}}
\newcommand{\fgene}{\ensuremath{\mathcal{C}^{-\infty}}}
\newcommand{\croc}{\ensuremath{\hookrightarrow}}
\newcommand{\T}{\ensuremath{\hbox{\bf T}}}
\newcommand{\End}{\ensuremath{\hbox{\rm End}}}

\newcommand{\str}{\operatorname{Str}}
\newcommand{\herm}{\operatorname{Herm}}

\newcommand{\ch}{\operatorname{Ch}}
\newcommand{\chgf}{\operatorname{Ch^1_{\rm sup}}}
\newcommand{\chgs}{\operatorname{Ch^2_{\rm sup}}}
\newcommand{\chbv}{\operatorname{Ch_{\rm BV}}}

\newcommand{\chs}{\operatorname{Ch_{\rm sup}}}

\newcommand{\chg}{\operatorname{Ch_{\rm sup}}}

\newcommand{\chc}{\operatorname{Ch_{\rm c}}}

\newcommand{\chr}{\operatorname{Ch_{\rm rel}}}

\newcommand{\chrf}{\operatorname{Ch^1_{\rm rel}}}
\newcommand{\chrs}{\operatorname{Ch^2_{\rm rel}}}

\def \tur {{\rm Th_{{\ensuremath{\rm rel}}}}}

\newcommand{\supp}{\operatorname{\hbox{\rm \small Supp}}}
\newcommand{\p}{\operatorname{p}}

\newcommand{\pr}{\operatorname{P_{\rm rel}}}
\newcommand{\prf}{\operatorname{P^1_{\rm rel}}}
\newcommand{\prs}{\operatorname{P^2_{\rm rel}}}
\newcommand{\prFf}{\operatorname{P^{F_1}_{\rm rel}}}
\newcommand{\prFs}{\operatorname{P^{F_2}_{\rm rel}}}
\newcommand{\res}{\operatorname{\bf r}}
\renewcommand{\index}{\operatorname{index}}

\def \cst {{\rm cst}}
\def \sm {{\rm m}}
\def \V  {{\rm V}}
\def \K {\mathbf{k}}

\def \Par {{\rm Par}}

\def \id {{\rm Id}}

\title{Equivariant Chern characters with generalized coefficients}
\author{Paul-Emile Paradan, Mich\`ele Vergne}

\date{January 2008}

\begin{document}

\maketitle

 {\small
 \tableofcontents}

\section{Introduction}

These notes form the next episode in a series of articles
dedicated to a detailed proof of a cohomological  index formula
for transversally elliptic pseudo-differential operators and
applications. The complete  notes will be published as a
monograph. The first two chapters are already available as
\cite{pep-vergne1} and \cite{pep-vergne2}. We tried our best in
order that each chapter is relatively self-contained, at the expense of
repeating definitions. In this episode, we construct the relative
equivariant Chern character of a morphism of vector bundles, {\em
localized} by a $1$-form $\lambda$ and we prove a multiplicativity
property of this generalized Chern character.

Let us first give  motivations for the construction of  the
``localized Chern character" of a morphism  of vector bundles.
 Let $M$ be a compact manifold. The Atiyah-Singer
formula for the index of an {\em elliptic} pseudo-differential
operator $P$ on $M$ with {\em elliptic} symbol $\sigma$ on $\T^*M$
involves integration over the non compact manifold $\T^*M$ of the
Chern character $\chc(\sigma)$ of $\sigma$ multiplied by the Todd
class of $\T^*M$:
$$\index(P)=\int_{\mathbf{T}^*M}(2i\pi)^{-\dim M}
\chc(\sigma) {\rm Todd}(\T^*M).$$

Here $\sigma$, the symbol of $P$,  is a morphism of vector bundles
on $\T^*M$ invertible outside the zero section of $\T^*M$ and the
Chern character $\chc(\sigma)$ is supported on a small
neighborhood of $M$ embedded in $\T^*M$ as the zero section. It is
important that the representative of the Chern character
$\ch_c(\sigma)$ is compactly supported to perform integration.
%

Assume that a compact Lie group $K$ acts on $M$. Let $\kgot$ be the
Lie algebra of $K$. If $X\in \kgot$, we denote by
$VX$ the infinitesimal vector field generated by $X$.

 If the elliptic operator $P$ is  $K$-invariant, then
$\index(P)$ is a smooth function on $K$. The equivariant index of
$P$ can be expressed similarly as the integral of the equivariant
Chern character $\chc(\sigma)$ of $\sigma$ multiplied by the
equivariant Todd class of $\T^*M$: for $X\in \kgot$ small enough,
$$
\index(P)(\exp X)=\int_{\mathbf{T}^*M}(2i\pi)^{-\dim M}
\chc(\sigma)(X) {\rm Todd}(\T^*M)(X).
$$

Here $\chc(\sigma)(X)$ is a compactly supported closed
equivariant differential form, that is a differential form on
$\T^*M$ depending smoothly of $X\in \kgot$, and closed for the
equivariant differential $D$. The result of the integration
determines a smooth function on a neighborhood of $1$ in $K$ and
similar formulae can be given near any point of $K$.

The motivation  for this article is to extend the construction of
the compactly supported class $\ch_c(\sigma)$ to the case where
$P$ is not  necessarily  elliptic, but  still {\em transversally elliptic} relatively
to the action of  $K$. An equivariant pseudo-differential operator $P$ with symbol
$\sigma(x,\xi)$ on $\T^*M$ is called transversally elliptic,  if it
is elliptic in the directions transversal to $K$-orbits. More
precisely, let $\T^*_K M$ be the set of co-vectors that are
orthogonal to the $K$-orbits. Then $\sigma(x,\xi)$, when
restricted to $\T^*_K M$, is invertible when $\xi\neq 0$. In this
case, the operator $P$ has again an index which is a generalized
function on $K$. It is thus natural to look for a Chern
character with  coefficients generalized functions   as well.
Thus we will construct a closed compactly supported
equivariant form $\chc(\sigma,\omega)(X)$ with coefficients generalized 
functions on $\kgot$ associated to such a symbol. In the next chapter, we will prove
that again
$$\index(\sigma)(\exp X)=\int_{\mathbf{T}^*M}(2i\pi)^{-\dim M}
\chc(\sigma,\omega)(X) {\rm Todd}(\T^*M)(X).$$

As the differential form  $\chc(\sigma,\omega)(X)$ is compactly supported,  
the integration can be performed and
in this case the result determines  a generalized function on a
neighborhood of $1$ in $K$. Similar formulae can be given near any
point of $K$.

Here $\omega$ is the Liouville $1$-form on $\T^*M$ which  defines a
map $f_\omega:\T^*M\to \kgot^*$  by the formula 
$$
\langle f_\omega(n),X\rangle=\langle \omega(n),V_nX\rangle 
$$ 
and $\T_K^*M$ is precisely the set $f_\omega^{-1}(0)$. Our construction of the Chern
character  $\chc(\sigma,\omega)$ is based on  this fact.

\bigskip

Thus given a $K$-manifold $N$ and a real invariant one form
$\lambda$  on $N$, we consider the map $f_\lambda: N\to \kgot^*$
defined by $\langle f_\lambda ,X\rangle=\langle \lambda,VX\rangle$. Define the
``critical set" $C_\lambda$ by  $C_\lambda=f_{\lambda}^{-1}(0)$.
 In other words, the point
$n\in N$ is in $C_\lambda$ if the co-vector $\lambda(n)$ is
orthogonal to the vectors tangent to the orbit $K\cdot n$.
 We will associate to any equivariant morphism of
vector bundles $\sigma$ on $N$,  invertible outside a closed
invariant set $F$,  an equivariant (relative) Chern character
$\chr(\sigma,\lambda)$, with $C^{-\infty}$-coefficients, with
support on the set $C_\lambda\cap F$. In particular, if this set
is compact, our equivariant relative Chern character leads to a
compactly supported Chern character $\chc(\sigma,\lambda)$ which is supported in a neighborhood of
$C_\lambda\cap F$. Then
 we will prove a certain number of functorial properties of the Chern
character $\chr(\sigma,\lambda)$. One of the most important
property is its multiplicativity.

The main ideas of our construction come from two sources:

$\bullet$ We use the construction basically due to Quillen of the
\emph{equivariant relative Chern character} $\chr(\sigma)$ already
explained in  \cite{pep-vergne1},
\cite{pep-vergne2}.

$\bullet$ We use a localization argument on the ``critical" set
$f_\lambda^{-1}(0)$ originated in Witten \cite{Witten} and
systematized in Paradan \cite{pep1,pep2}.
Indeed,  the fundamental remark inspired by the ``non abelian localization
formula of Witten"  is that
$$1=0$$ in equivariant cohomology on the complement of the
critical set $C_\lambda$.

The idea to use the Liouville one-form $\omega$ to construct a Chern character for symbols of transversally 
elliptic operators was already present in the preceding construction of Berline-Vergne
\cite{B-V.inventiones.96.1}.   In this work, a Chern character on $\T^*M$ was constructed 
with gaussian look in transverse directions to the $K$-action  and oscillatory behavior in 
parallel directions of the $K$-action. Thus this Chern character  was integrable in the generalized sense.
Our new construction  gives directly a (relative) Chern character  with  compact 
support equal to the intersection of the support of the morphism $\sigma$ with the critical set $\T^*_KM$.

\medskip

Let us now explain in more details the content of this article.

In Section \ref{sec:chg-sigma}, we define the equivariant differential $D$,  
the equivariant relative cohomologies $\Hcal^\infty(\kgot,N,N\setminus
F)$ with $C^{\infty}$ coefficients as well  as the relative
cohomology $\Hcal^{-\infty}(\kgot,N,N\setminus F)$ with
$C^{-\infty}$ coefficients and  the product in relative cohomology
$\diamond : \Hcal^{-\infty}(\kgot,N,N\setminus
F_1)\times\Hcal^{\infty}(\kgot,N,N\setminus F_2) \to
\Hcal^{-\infty}(\kgot,N,N\setminus (F_1\cap F_2))$.

In Section \ref{sec:chg-sigma-gene}, we start by recalling
Quillen's construction of the relative Chern character that we
studied already in  \cite{pep-vergne1}, \cite{pep-vergne2}. Let us
consider  a $K$-equivariant morphism $\sigma:\Ecal^+\to \Ecal^-$
between two $K$-vector bundles over $N$: the symbol $\sigma$ is
invertible outside a (possibly non-compact) subset $F\subset N$.
Following an idea of Quillen \cite{Quillen85}, we defined the
\emph{equivariant relative Chern character} $\chr(\sigma)$ as the
class defined by a couple $(\alpha,\beta(\sigma))$ of equivariant
forms:

\begin{itemize}
\item $\alpha:=\ch(\Ecal^+)-\ch(\Ecal^-)$ is a closed equivariant form on $N$.

\item $\beta(\sigma)$ is an equivariant form on $N\setminus F$ which is constructed
with the help of the invariant super-connections $\A^\sigma(t)=\nabla + it (\sigma\oplus\sigma^*),\ t\in\Rbb$.
Here $\nabla$ is  an invariant connection on $\Ecal^+\oplus \Ecal^-$ preserving the grading.

\item we have  on $N\setminus F$, the equality of equivariant forms
\begin{equation}\label{intro:a=D(b)}
\alpha\vert_{N\setminus F}=D(\beta(\sigma)).
\end{equation}
\end{itemize}

In this construction, the equivariant forms $X\mapsto\alpha(X)$
and $X\mapsto\beta(\sigma)(X)$ have a {\bf smooth} dependance
relatively to the parameter $X\in\kgot$. 
Thus  $\chr(\sigma)$ is an element of the relative cohomology
group $\Hcal^{\infty}(\kgot,N,N\setminus F).$

In Subsection \ref{sec:chg-sigma-gene}, we deform the relative Chern
character of $\sigma$ by using as a further tool of deformation an
invariant real $1$-form $\lambda$ on $N$. Using the family of
invariant super-connections
$$
\A^{\sigma,\lambda}(t):=\nabla + it (\sigma\oplus\sigma^*)+ it\lambda,\ t\in\Rbb
$$
we construct an equivariant form $\beta(\sigma,\lambda)$ on $N\setminus (F\cap C_\lambda)$
such that
\begin{equation}\label{intro:a=D(b)-gene}
\alpha\vert_{N\setminus (F\cap C_\lambda)}=D(\beta(\sigma,\lambda))
\end{equation}
holds on $N\setminus (F\cap C_\lambda)$. There are two kinds of
difference between (\ref{intro:a=D(b)}) and
(\ref{intro:a=D(b)-gene}).
\begin{itemize}
\item the subset $N\setminus (F\cap C_\lambda)$ contains $N\setminus F$ so Equation
(\ref{intro:a=D(b)-gene}) which holds on
$N\setminus (F\cap C_\lambda)$  is in some sense  ``stronger'' than  Equation
(\ref{intro:a=D(b)}) which holds on $N\setminus F$.

\item the equivariant form $X\mapsto\beta(\sigma,\lambda)(X)$ has a $\fgene$ dependance
relatively to the parameter $X\in\kgot$.
\end{itemize}

\medskip

We define the relative Chern character of $\sigma$ deformed by the
$1$-form $\lambda$ as the class defined by the couple
$(\alpha,\beta(\sigma,\lambda))$ : we denote it by
$\chr(\sigma,\lambda)$. It is an element of the relative
cohomology group $\Hcal^{-\infty}(\kgot,N,N\setminus (F\cap
C_\lambda)).$

One case of interest is when $F\cap C_\lambda$ is a compact subset of $N$.
Using an invariant function $\chi$ on $N$, identically equal to $1$ in a neighborhood of
$F\cap C_\lambda$ and with compact support, the equivariant form
$$
\p(\chr(\sigma,\lambda)):=\chi\alpha +d\chi\beta(\sigma,\lambda)
$$
is equivariantly closed, with compact support on $N$, and has a
$\fgene$ dependance relatively to the parameter $X\in\kgot$: its
equivariant class is denoted $\chc(\sigma,\lambda)$.

An important case  of  $K$-equivariant bundle is the trivial bundle represented 
by the trivial symbol $\left[0\right]:N\times\Cbb\to N\times\{0\}$. In this case, we denote  the
relative class $\chr(\left[0\right],\lambda)$ by $\pr(\lambda)$. The class $\pr(\lambda)$ is 
defined by a couple $(1,\beta(\lambda))$ where $\beta(\lambda)$ is a generalized equivariant 
form on $N\setminus C_\lambda$ satisfying
$$
1=D(\beta(\lambda)).
$$
This equation $1=0$ on $N\setminus C_\lambda$, together with explicit description of 
$\beta(\lambda)$, is the principle   explaining   Witten ``non abelian localisation theorem".

In Subsection $\ref{sec:chg-sigma-gene}$ we study the functorial
properties of the classes $\chr(\sigma,\lambda)$ and
$\chc(\sigma,\lambda)$. We prove in particular that these classes
behave nicely under the  product. When $\sigma_1,\sigma_2$ are two
equivariant symbols on $N$,  we can take their  product
$\sigma_1\odot\sigma_2$ which is a symbol on $N$ such that
$\supp(\sigma_1\odot\sigma_2)=\supp(\sigma_1)\cap\supp(\sigma_2)$.
We prove then that
\begin{equation}\label{eq:intro-chr-produit-lambda}
\chr(\sigma_1,\lambda)\diamond\chr(\sigma_2)= \chr(\sigma_1\odot\sigma_2,\lambda).
\end{equation}
We have then a factorized  expression for the class
$\chr(\sigma,\lambda)$ by taking $\sigma_1=\left[0\right]$ in
(\ref{eq:intro-chr-produit-lambda}): we have
\begin{equation}\label{eq:intro-chr-lambda}
\chr(\sigma,\lambda)= \pr(\lambda)\diamond\chr(\sigma).
\end{equation}

The first class $\chr(\sigma)$ is supported on $\supp(\sigma)$
while the second  class $\pr(\lambda)$ (equivalent to $1$)
 is supported on $C_\lambda$.

When $K$ is a torus, and $\lambda$ the invariant one form associated to a generic vector 
field $VX$ via a metric on $\T M$, the set $C_\lambda$ coincide with the set $N^K$ of 
fixed points for the action of the torus $K$.  Then $\pr(\lambda)(X)$ can be represented 
as a differential form with coefficients  boundary values of rational functions of $X\in \kgot$.  
In this special case, Equation  (\ref{eq:intro-chr-lambda}) is strongly related to Segal's localization 
theorem on the fixed point set  for equivariant {\bf K}-theory.

In Subsection \ref{sec:product-transv-elliptic}, we study the
multiplicativity properties of our  Chern characters when a
product of groups acts on $N$.  If $\lambda,\mu$ are $1$-forms
invariant by $K_1\times K_2$, we define $C^1_{\lambda}$  to be the
critical set of $\lambda$ with respect to $K_1$ and $C^2_{\mu}$
the critical set of $\mu$ with respects to $K_2$. Let
$\sigma$ and $\tau$ be two $K_1\times K_2$-equivariant morphisms on $N$ which are invertible 
respectively on $N\setminus F_1$ and $N\setminus F_2$. Then the
relative Chern character $\chrf(\sigma,\lambda)$ is defined as a
class in $\Hcal^{-\infty,\infty}(\kgot_1\times \kgot_2,
N,N\setminus (C_{\lambda}^1\cap F_1))$ while
$\chrs(\tau,\mu)$ is defined as a class in
$\Hcal^{\infty,-\infty}(\kgot_1\times \kgot_2, N,N\setminus
(C_{\mu}^2\cap F_2))$. As suggested by the notation, an
element in  $\Hcal^{-\infty,\infty}(\kgot_1\times \kgot_2,
N,N\setminus F)$  can be represented as a couple
$(\alpha(X,Y),\beta(X,Y))$ of differential forms with smooth
dependence in $Y$, and a generalized function in $X$. Then we can
multiply  the classes $\chrf(\sigma,\lambda)$ and
$\chrs(\tau,\mu)$. One main theorem, which will be crucial for
the functoriality properties of the equivariant index, is Theorem
\ref{prop:chg-produit-groupe-sigma}
$$
\chrf(\sigma,\lambda)\diamond
\chrs(\tau,\mu)=\chr(\sigma\odot\tau,\lambda+\mu).
$$

In Section \ref{sec:transversally}, we consider the case where the
$K$-manifold $N$ is the cotangent bundle $\T^*M$ of a $K$-manifold
$M$. Let $\T^*_K M$ be the set of co-vectors that are orthogonal
to the $K$-orbits. If $\sigma$ is invertible outside a closed
invariant set $F\subset\T^*M$,  $\sigma$  is  a transversally
elliptic morphism if  $F\cap \T^*_K M$ is compact. We work here
with the Liouville one form $\omega$ on $\T^*M$. It is easy to see
that the set $C_\omega$ coincides with $\T^*_K M$. So, for a
transversally elliptic morphism $\sigma$ we define its Chern class
with compact support as the class $ \chc(\sigma,\omega).$

We finally compare our construction with the Berline-Vergne
construction.

\section{Equivariant cohomologies with $\fgene$ coefficients }
\label{sec:chg-sigma}

Let $N$ be a manifold, and let $\Acal(N)$ be the algebra of
differential forms on $N$. We denote by $\Acal_c(N)$ the subalgebra
of compactly supported differential forms. We will
consider on $\Acal(N)$ and $\Acal_c(N)$ the $\Zbb_2$-grading in even
or odd differential forms.

Let $K$ be a  compact Lie group with Lie algebra  $\kgot$.
We suppose that the manifold $N$ is provided with an action of $K$.
We denote $X\mapsto VX$ the corresponding morphism from $\kgot$ into
the Lie algebra of vectors fields on $N$: for $n\in N$,
$$
V_nX:=\frac{d}{d\epsilon} \exp(-\epsilon X)\cdot n|_{\epsilon=0}.
$$

Let $\Acal^{\infty}(\kgot, N)$ be the $\Zbb_2$-graded algebra of
equivariant smooth functions $\alpha: \kgot\to\Acal(N)$. Its
$\Zbb_2$-grading is the grading induced by the exterior degree. Let
$D= d-\iota(VX)$ be the equivariant differential:
$(D\alpha)(X)=d(\alpha(X))-\iota(VX)\alpha(X)$. Here the operator
$\iota(VX)$ is the contraction of a differential form by the vector field $VX$.  
Let $\Hcal^{\infty}(\kgot,N):= \mathrm{Ker} D/ \mathrm{Im} D$ be the
equivariant cohomology algebra with $C^{\infty}$-coefficients. It is a module
over the algebra $\f(\kgot)^K$ of $K$-invariant $C^{\infty}$-functions on $\kgot$.

The sub-algebra $\Acal^{\infty}_c(\kgot, N)\subset\Acal^{\infty}(\kgot, N)$ of
equivariant differential forms with compact support is defined as follows :
$\alpha\in \Acal^{\infty}_c(\kgot, N)$ if there exists a compact subset $\Kcal_\alpha\subset N$
such that the differential form $\alpha(X)\in \Acal(N)$ is supported on $\Kcal_\alpha$ for
 any $X\in\kgot$. We denote $\Hcal_c^{\infty}(\kgot, N)$ the corresponding algebra of
cohomology: it is a  $\Zbb_2$-graded algebra.

\medskip

Kumar and Vergne \cite{Kumar-Vergne} have defined generalized
equivariant cohomology spaces obtained by considering equivariant
differential forms with $\fgene$ coefficients. Let us recall the
definition.

Let $\Acal^{-\infty}(\kgot, N)$ be the space of \emph{generalized}
equivariant differential forms. An element $ \alpha\in
\Acal^{-\infty}(\kgot, N)$ is, by definition, a
$\fgene$-equivariant map $\alpha:\kgot\to \Acal(N)$. The value
taken by $\alpha$ on a smooth compactly supported density $Q(X)dX$
on $\kgot$ is denoted by $ \int_{\kgot}\alpha(X)Q(X)dX\in \,
\Acal(N)$. We have $\Acal^{\infty}(\kgot,
N)\subset\Acal^{-\infty}(\kgot, N)$ and we can extend the
differential $D$ to $\Acal^{-\infty}(\kgot, N)$
\cite{Kumar-Vergne}. We denote by $\Hcal^{-\infty}(\kgot,N)$ the
corresponding cohomology space.
Note that $\Acal^{-\infty}(\kgot, N)$ is a module over
$\Acal^{\infty}(\kgot, N)$ under the wedge product, hence the
cohomology space $\Hcal^{-\infty}(\kgot,N)$ is a module over
$\Hcal^{\infty}(\kgot, N)$.

The sub-space $\Acal^{-\infty}_c(\kgot, N)\subset\Acal^{-\infty}(\kgot, N)$ of
generalized equivariant differential forms with compact support is defined as follows :
$\alpha\in \Acal^{-\infty}_c(\kgot, N)$ if there exits a compact subset $\Kcal_\alpha\subset N$
such that the differential form $\int_{\kgot}\alpha(X)Q(X)dX\in \,
\Acal(N)$ is supported on $\Kcal_\alpha$ for
 any  compactly supported density $Q(X)dX$.
We denote $\Hcal_c^{-\infty}(\kgot, N)$ the corresponding space of
cohomology. The $\Zbb_2$-grading on $\Acal(N)$
induces a $\Zbb_2$-grading on the cohomology spaces
$\Hcal^{-\infty}(\kgot,N)$ and $\Hcal^{-\infty}_c(\kgot,N)$.

Let us stress here that a generalized equivariant form
$\alpha(X,n,dn)$ on $N$ is {\bf smooth} with respect to the
variable $n\in N$. Thus we can restrict generalized forms to
$K$-equivariant submanifolds of $N$. However, in general, if $G$
is a subgroup of $K$, a $K$-equivariant generalized form on $N$ do
not restrict to a $G$-equivariant generalized form.

 More generally, if $g:M\to N$ is
a $K$-equivariant map from the $K$-manifold $M$ to the
$K$-manifold $N$, then we obtain a map
$g^*:\Acal^{-\infty}(\kgot,N)\to \Acal^{-\infty}(\kgot,M)$, which
induces a map $g^*$ in cohomology. When $U$ is an open invariant
subset of $N$, we denote by $\alpha\to \alpha|_U$ the restriction
of $\alpha\in \Acal^{-\infty}(\kgot,N)$ to $U$.

There is a natural map $\Hcal^{\infty}(\kgot,N)\to \Hcal^{-\infty}(\kgot,N)$.
This map is not injective in general.
Let us give a simple example.
Let $U(1)$ acting on  $N=\Rbb^2\setminus\{0\}$ by rotations. 
Let $\kgot\sim \Rbb$ be the Lie algebra of $U(1)$. The vector $X\in \kgot$ produces the infinitesimal 
vector field $VX=X(y\partial_x-x\partial_y)$ on $\Rbb^2$. 
Denote by $\frac{1}{X}$ any generalized function of $X\in \kgot$ such that $X(\frac{1}{X})=1$. 
Let $\lambda=\frac{x dy-ydx}{x^2+y^2}$. Then 
$$
D(\frac{1}{X}\lambda)= (d-\iota(VX))(\frac{1}{X}\lambda)=1.
$$
Thus the image of $1$ is exact in $\Hcal^{-\infty}(\kgot,N)$, so that the image of 
$\Hcal^{\infty}(\kgot,N)$ is $0$ in $\Hcal^{-\infty}(\kgot,N)$.

\subsection{Examples of generalized equivariant forms}
\label{section:example-gene-forms}

In this article, equivariant forms with generalized coefficients
appear in the following situation. Let $t\mapsto \eta_t(X)$ be a
smooth map from $\Rbb$ into $\Acal^{\infty}(\kgot, N)$. For any
$t\geq 0$, the integral
$$
\beta_t(X)=\int_0^t\eta_s(X)ds
$$
defines an element of $\Acal^{\infty}(\kgot, N)$. One may ask if
the ``limit" of $\beta_t(X)$ when $t$ goes to infinity exists.

\medskip

Let $X_1,\ldots,X_{\dim K}$ be a base of $\kgot$. For any
$\nu:=(\nu_1,\ldots,\nu_{\dim K})\in\Nbb^{\dim K}$, we denote
$\frac{\partial}{\partial X^\nu}$ the differential operator
$\prod_i\left(\frac{\partial}{\partial X_i}\right)^{\nu_i}$ of
degree $|\nu|:=\sum_i\nu_i$.

\begin{defi}\label{def:norme-Kcal}
For a compact subset $\Kcal$ of $\kgot$ and $r\in\Nbb$, we denote
$\|- \|_{\Kcal,r}$ the semi-norm on $\f(\kgot)$ defined by
$\|Q\|_{\Kcal,r}=\sup_{X\in\Kcal,|\nu|\leq r}
\Big|\frac{\partial}{\partial X^\nu}Q(X)\Big|.$
\end{defi}

We make the following assumption on $\eta_t(X)$. For every compact
subset $\Kcal\times \Kcal'\subset \kgot\times N$ and for any
integer $r\in \Nbb$, there exists $\cst> 0$ and $r'\in\Nbb$ such
that  the following estimate
\begin{equation}\label{eq:hyp-eta}
\Big|\!\Big|\int_{\kgot}\eta_t(X)Q(X)dX\Big|\!\Big|(n)\leq\cst
\frac{\|Q\|_{\Kcal,r'}}{(1+t)^r},\quad  n\in \Kcal',\quad t\geq 0,
\end{equation}
holds for every function $Q\in\f(\kgot)$ supported in $\Kcal$.
Here the norm $\|-\|$ on the differential forms on $N$ is defined
via the choice of a Riemannian metric on $N$.

Under estimates (\ref{eq:hyp-eta}), we can define the
equivariant form $\beta \in\Acal^{-\infty}(\kgot, N)$ as the limit
of the equivariant forms $\beta_t\in \Acal^{\infty}(\kgot, N)$
when $t$ goes to infinity. More precisely, for every
$Q\in\f_c(\kgot)$,  we have
\begin{equation}\label{eq:def-beta}
\int_{\kgot}\beta(X)Q(X)dX:=\int_0^\infty
\left(\int_{\kgot}\eta_t(X)Q(X)dX\right) dt.
\end{equation}

In fact, in order to insure that  the right hand side of
(\ref{eq:def-beta}) defines a \emph{smooth} form on $N$, we need
the following strongest version of the estimate
(\ref{eq:hyp-eta}) : we have
\begin{equation}\label{eq:hyp-eta-2}
\Big|\!\Big|
D(\partial)\cdot\int_{\kgot}\eta_t(X)Q(X)dX\Big|\!\Big|(n)
\leq\cst \frac{\|Q\|_{\Kcal,r'}}{(1+t)^r},\quad  n\in \Kcal',\quad t\geq 0,
\end{equation}
for any differential operator $D(\partial)$ acting on $\Acal(N)$.
Under (\ref{eq:hyp-eta-2}), the generalized
equivariant form $\beta(X):=\int_0^{\infty}\eta_t(X) dt$ satisfies
$$
D(\beta) (X):=\int_0^{\infty}D(\eta_t)(X) dt.
$$

\medskip

Let us consider the following basic case which appears in
\cite{pep1,pep2}. Let $f:N\to \kgot^*$ be an equivariant map, and
let $\gamma_t(X)$ be an equivariant form on $N$ which depends
polynomially on both variables $t$ and $X$. We consider the family
$$
\eta_t(X):=\gamma_t(X)\, \e^{i\, t\, \langle f, X\rangle},\quad t\in\Rbb.
$$
Then, for any $Q\in\f_c(\kgot)$, we have
$$
\int_{\kgot}\gamma_t(X)\, \e^{i\, t\, \langle f, X\rangle}Q(X)dX:=\widehat{\gamma_tQ }( t \, f)
$$
where\ $\widehat{{\ }}$\  is the Fourier transform. Since the  Fourier transform of a
compactly supported function
is rapidly decreasing, one sees that the estimates (\ref{eq:hyp-eta})
and (\ref{eq:hyp-eta-2}) holds
in this case on the open subset $\{f\neq 0\}$ : the integral
$$
\int_0^\infty \gamma_t(X)\, \e^{i\, t\, \langle f, X\rangle}dt
$$
defines an equivariant form with generalized coefficients on $\{f\neq 0\}\subset N$.

\subsection{Relative equivariant cohomology : the $\fgene$ case }
\label{section:coho-support}

Let $F$ be a closed $K$-invariant subset of $N$. We have a
restriction operation $r:\alpha\mapsto \alpha|_{N\setminus F}$
from $\Acal^{-\infty}(\kgot, N)$ into $\Acal^{-\infty}(\kgot,
N\setminus F)$. To an equivariant cohomology class on $N$
vanishing on $N\setminus F$, we associate a relative equivariant
cohomology class. Let us explain the construction : see
\cite{Bott-Tu,pep-vergne1} for the non-equivariant case, and
\cite{pep-vergne2} for the equivariant case with $\f$
coefficients. Consider the complex
$\Acal^{-\infty}(\kgot,N,N\setminus F)$ with
$$
\Acal^{-\infty}(\kgot,N,N\setminus
F):=\Acal^{-\infty}(\kgot,N)\oplus
\Acal^{-\infty}(\kgot,N\setminus F)
$$
and differential $D_{\rm rel}\left(\alpha,\beta\right)=
\left(D\alpha,\alpha|_{N\setminus F}- D\beta \right)$. Let
$\Acal^{\infty}(\kgot,N,N\setminus F)$ be the sub-complex of
$\Acal^{-\infty}(\kgot,N,N\setminus F)$ formed by  couples of
equivariant forms with smooth coefficients : this sub-complex is
stable under $D$.

\begin{defi}\label{relcoh}
The cohomology of the complexes
$(\Acal^{-\infty}(\kgot,N,N\setminus F),D_{\rm rel})$  and
$(\Acal^{\infty}(\kgot,N,N\setminus F),D_{\rm rel})$
are the relative equivariant cohomology spaces \break
$\Hcal^{-\infty}(\kgot,N,N\setminus F)$ and
$\Hcal^{\infty}(\kgot,N,N\setminus F)$. \end{defi}

The class defined by a $D_{\rm rel}$-closed element
$(\alpha,\beta)\in\Acal^{-\infty}(\kgot,N,N\setminus F)$ will be
denoted $[\alpha,\beta]$.


The complex $\Acal^{-\infty}(\kgot,N,N\setminus F)$ is
$\Zbb_2$-graded : for $\epsilon\in\Zbb_2$, we take \break
$\left[\Acal^{-\infty}(\kgot,N,N\setminus
F)\right]^\epsilon=\left[\Acal^{-\infty}(\kgot,N)\right]^{\epsilon}\oplus
\left[\Acal^{-\infty}(\kgot,N\setminus F)\right]^{\epsilon+1}$.
Since $D_{\rm rel}$ sends
$\left[\Acal^{-\infty}(\kgot,N,N\setminus F)\right]^{\epsilon}$
into $\left[\Acal^{-\infty}(\kgot,N,N\setminus
F)\right]^{\epsilon+1}$, the $\Zbb_2$-grading descends to the
relative cohomology spaces $\Hcal^{-\infty}(\kgot,N,N\setminus
F)$.

\medskip

We review the basic facts concerning the relative cohomology groups.
We consider now the following maps.
\begin{itemize}
\item The projection $j:\Acal^{-\infty}(\kgot,N,N\setminus F)\to\Acal^{-\infty}(\kgot,N)$
is the degree $0$ map defined by $j(\alpha,\beta)=\alpha$.

\item The inclusion  $i:\Acal^{-\infty}(\kgot,N\setminus F)\to\Acal^{-\infty}(\kgot,N,N\setminus~F~)$
is the degree $+1$ map defined by $i(\beta)=(0,\beta)$.

\item The restriction $r:\Acal^{-\infty}(\kgot,N)\to\Acal^{-\infty}(\kgot,N\setminus~ F)$
is the degree $0$ map defined by $r(\alpha)=\alpha|_{N\setminus~F}$.
\end{itemize}

It is easy to see that $i,j,r$ induce maps in cohomology that we still denote
by $i,j,r$.

\begin{lem}\label{lem:basic-relative-coho}
$\bullet$  We have an exact triangle
$$
\xymatrix{
    & \Hcal^{-\infty}(\kgot,N,N\setminus~F~)\ar@{<-}_{i}[ld]\ar@{->}^{j}[rd] &     \\
\Hcal^{-\infty}(\kgot,N\setminus F)   &           &  \Hcal^{-\infty}(\kgot,N).\ar@{->}^{r}[ll]
}
$$

$\bullet$ If $F\subset F'$ are closed $K$-invariant subsets of $N$,
the restriction map $(\alpha,\beta)\mapsto
(\alpha,\beta|_{N\setminus F'})$ induces a map
\begin{equation}\label{eq:res-F-gene}
\res_{F',F}:\Hcal^{-\infty}(\kgot,N,N\setminus F)\to
\Hcal^{-\infty}(\kgot,N,N\setminus F').
\end{equation}

$\bullet$ The inclusion $\Acal^{\infty}(\kgot,N,N\setminus F) \croc
\Acal^{-\infty}(\kgot,N,N\setminus F)$ induces a map \break
$\Hcal^{\infty}(\kgot,N,N\setminus F)\to
\Hcal^{-\infty}(\kgot,N,N\setminus F)$.
\end{lem}


\subsection{Product in relative equivariant cohomology}
\label{section:produit-relatif-gene}

Let $F_1$ and $F_2$ be two closed $K$-invariant subsets of $N$. In
\cite{pep-vergne2}, we have define a product
$\diamond : \Hcal^\infty(\kgot,N,N\setminus F_1)\times\Hcal^\infty(\kgot,N,N\setminus F_2) \to
\Hcal^\infty(\kgot,N,N\setminus (F_1\cap F_2))$. Let us check that this product is still defined when
one equivariant form has generalized coefficients.


Let $U_1:=N\setminus F_1$, $U_2:=N\setminus F_2$ so that
$U:=N\setminus (F_1\cap F_2)=U_1\cup U_2$. Let
$\Phi:=(\Phi_1,\Phi_2)$ be a partition of unity subordinate to the
covering $U_1\cup U_2$ of $U$. By averaging by $K$, we may suppose
that the functions $\Phi_k$ are invariant.

Since $\Phi_k\in\f(U)^K$ is supported in $U_k$, the product
$\gamma\mapsto \Phi_k\gamma$ defines  maps
$\Acal^{\infty}(\kgot,U_k) \to
\Acal^{\infty}(\kgot,U)$ and $\Acal^{-\infty}(\kgot,U_k) \to
\Acal^{-\infty}(\kgot,U)$. Since
$d\Phi_1=-d\Phi_2\in \Acal(U)^K$ is supported in $U_1\cap U_2$,
the product $\gamma\mapsto
d\Phi_1\wedge\gamma$ defines a map
$\Acal^{-\infty}(\kgot,U_1\cap U_2) \to
\Acal^{-\infty}(\kgot,U)$.

With the help of $\Phi$, we define a bilinear map
$\diamond_\Phi:\Acal^{-\infty}(\kgot,N,N\setminus F_1)\times
\Acal^\infty(\kgot,N,N\setminus F_2) \to
\Acal^{-\infty}(\kgot,N,N\setminus (F_1\cap F_2))$ as follows.

\begin{defi}\label{prodgene}
For an equivariant form $a_1:=(\alpha_1,\beta_1)\in
\Acal^{-\infty}(\kgot,N,N\setminus F_1)$ with \emph{generalized
coefficients} and an equivariant form
 $a_2:=(\alpha_2,\beta_2)\in \Acal^{\infty}(\kgot,N,N\setminus F_2)$ with
\emph{smooth coefficients}, we define
$$
a_1 \diamond_\Phi a_2
:=(\alpha_1\wedge \alpha_2,\beta(a_1,a_2))\quad {\rm with}
$$
$$
\beta(a_1,a_2)= \Phi_1\beta_1 \wedge
\alpha_2+(-1)^{|a_1|}\alpha_1\wedge
\Phi_2\beta_2-(-1)^{|a_1|}d\Phi_1\wedge \beta_1\wedge \beta_2.
$$
\end{defi}

\medskip

Remark that $\Phi_1\beta_1 \wedge \alpha_2$,
$\alpha_1\wedge \Phi_2\beta_2$ and $d\Phi_1\wedge \beta_1\wedge
\beta_2$ are well defined equivariant forms with generalized coefficients on $U_1\cup U_2$.
So $a_1 \diamond_\Phi a_2\in \Acal^{-\infty}(\kgot,N,N\setminus (F_1\cap F_2))$. Note also
that $a_1 \diamond_\Phi a_2\in \Acal^{\infty}(\kgot,N,N\setminus (F_1\cap F_2))$, if $\alpha_1,\beta_1$
have smooth coefficients.

A small computation shows that
$D_{\rm rel}(a_1 \diamond_\Phi a_2)$ is equal to $(D_{\rm rel}a_1) \diamond_\Phi
a_2+(-1)^{|a_1|}a_1 \diamond_\Phi (D_{\rm rel}a_2)$.  Thus $\diamond_\Phi$ defines
bilinear maps

\begin{equation}\label{eq:produit-relatif-gene}
\Hcal^{-\infty}(\kgot,N,N\setminus F_1)
\times\Hcal^{\infty}(\kgot,N,N\setminus F_2)
\stackrel{\diamond_\Phi}{\longrightarrow}
\Hcal^{-\infty}(\kgot,N,N\setminus (F_1\cap F_2)),
\end{equation}
and
\begin{equation}\label{eq:produit-relatif}
\Hcal^{\infty}(\kgot,N,N\setminus F_1)
\times\Hcal^{\infty}(\kgot,N,N\setminus F_2)
\stackrel{\diamond_\Phi}{\longrightarrow}
\Hcal^{\infty}(\kgot,N,N\setminus (F_1\cap F_2)).
\end{equation}

Let us see that this product do not depend of the choice of the
partition of unity. If we have another partition
$\Phi'=(\Phi_1',\Phi_2')$, then $\Phi_1-\Phi'_1=-(\Phi_2-\Phi'_2)$.
It is immediate to verify that, if $D_{\rm rel}(a_1)=0$ and $D_{\rm
rel}(a_2)=0$, one has $a_1 \diamond_\Phi a_2 - a_1 \diamond_{\Phi'}a_2=
D_{\rm rel}\left(0,(-1)^{|a_1|} (\Phi_1-\Phi'_1)\beta_1\wedge \beta_2\right)$.

So the products (\ref{eq:produit-relatif-gene}) and (\ref{eq:produit-relatif}) will be
denoted by $\diamond$.

\begin{lem}\label{ass}
$\bullet$ The relative product is compatible with restrictions:
if $F_1\subset F'_1$ and $F_2\subset F_2'$ are closed invariant
subsets of $N$, then the diagram
\begin{equation}\label{eq:fonctoriel-prod-relatif}
\xymatrix@C=6mm{ \Hcal^{-\infty}(\kgot,N,N\setminus F_1)\ar[d]^{\res_1}
\!\!\!\!\! & \times \quad  \Hcal^{\infty}(\kgot,N,N\setminus F_2)\ar[d]^{\res_2}\ar[r]^-\lozenge &
\   \Hcal^{-\infty}(\kgot,N,N\setminus (F_1\cap F_2)) \ar[d]^{\res_{12}}\\
\Hcal^{-\infty}(\kgot,N,N\setminus F_1') \!\!\!\!\! & \times \quad
\Hcal^{\infty}(\kgot,N,N\setminus F_2') \ar[r]^-\lozenge & \
\Hcal^{-\infty}(\kgot,N,N\setminus (F_1'\cap F_2'))}
\end{equation}

is commutative. Here the $\res_i$ are the restrictions maps defined in (\ref{eq:res-F-gene}).

$\bullet$ The relative product is associative. More precisely, let $F_1,F_2,F_3$ be three
closed invariant subsets of $N$ and take $F=F_1\cap F_2\cap F_3$. For any relative classes
$a_1\in \Hcal^{-\infty}(\kgot,N,N\setminus F_1)$ and $a_i\in\Hcal^{\infty}(\kgot,N,N\setminus F_i)$
for $i=2,3$, we have $(a_1\diamond a_2)\diamond a_3= a_1\diamond (a_2\diamond a_3)$ in
$\Hcal^{-\infty}(\kgot,N,N\setminus F)$.

\end{lem}

\begin{proof}
The proof is identical to the one done in Section 3.3 of \cite{pep-vergne2}.
\end{proof}

\subsection{Inverse limit of equivariant cohomology with support}
\label{sec:cohomologie-support}

Let $U$ be an invariant open subset of $N$.

\begin{defi}\label{def:A_U}
A generalized equivariant form $\alpha$ on $N$ belongs to $\Acal_U^{-\infty}(\kgot,N)$ if there exists a closed
invariant subset $\Ccal_\alpha\subset U$ such that the differential form $\int_{\kgot}\alpha(X)Q(X)dX$ is supported
in $\Ccal_\alpha$, for any compactly supported density $Q(X)dX$ on $\kgot$.
\end{defi}

Note that the vector space $\Acal_U^{-\infty}(\kgot,N)$ is naturally a module over $\Acal^{\infty}(\kgot,N)$.
An element of $\Acal_U^{-\infty}(\kgot,N)$ will be called an equivariant form with support in $U$.
Let $\Acal_U^\infty(\kgot,N)$ be the intersection of $\Acal^{-\infty}_U(\kgot,N)$ with
$\Acal^{\infty}(\kgot,N)$ : $\alpha\in \Acal_U^\infty(\kgot,N)$ if there exist a closed set
$\Ccal_\alpha\subset U$ such that $\alpha(X)|_n=0$ for all $X\in\kgot$ and all
$n\in U\setminus \Ccal_\alpha$.

The spaces $\Acal_U^\infty(\kgot,N)$ and $\Acal_U^{-\infty}(\kgot,N)$ are stable
under the differential $D$, and we denote $\Hcal_U^\infty(\kgot,N)$ and
$\Hcal_U^{-\infty}(\kgot,N)$ the corresponding cohomology spaces: both are modules over
$\Hcal^{\infty}(\kgot,N)$.

Let $U,V$ be two invariants open subsets of $N$. The wedge product
gives a natural bilinear map
\begin{equation}\label{eq:prod-U-V}
\Hcal^{-\infty}_U(\kgot,N)\times\Hcal^{\infty}_V(\kgot,N)
\stackrel{\wedge}{\longrightarrow}\Hcal^{-\infty}_{U\cap V}(\kgot,N)
\end{equation}
of $\Hcal^{\infty}(\kgot,N)$-modules.

\medskip

Let $F$ be a closed $K$-invariant subset of $N$. We consider the set
$\Fcal_F$ of all open invariant neighborhoods $U$ of $F$ which is
ordered by the relation $U\leq V$ if and only if $V\subset U$. If $U\leq V$, we have then the
inclusion maps $\Acal_V^\infty(\kgot,N)\hookrightarrow
\Acal_U^\infty(\kgot,N)$ and $\Acal_V^{-\infty}(\kgot,N)\hookrightarrow
\Acal_U^{-\infty}(\kgot,N)$ which gives rise to the maps
$\Hcal_V^\infty(\kgot,N)\to \Hcal_U^\infty(\kgot,N)$ and $\Hcal_V^{-\infty}(\kgot,N)\to
\Hcal_U^{-\infty}(\kgot,N)$ both denoted $f_{U,V}$.

\begin{defi}\label{defiinductive}
 $\bullet$ We denote by   $\Hcal_F^\infty(\kgot,N)$ the inverse limit of the inverse system
 $(\Hcal_{U}^\infty(\kgot,N),f_{U,V};U,V\in\Fcal_F)$.

$\bullet$ We denote by $\Hcal^{-\infty}_F(\kgot,N)$ be the inverse
limit of the inverse system \break
$(\Hcal^{-\infty}_U(\kgot,N),f_{U,V}; U,V\in \Fcal_F)$.
\end{defi}

We will call $\Hcal_F^\infty(\kgot,N)$ and $\Hcal_F^{-\infty}(\kgot,N)$ the equivariant cohomology of $N$
supported on $F$ (with smooth or generalized coefficients) : both are module over $\Hcal^{\infty}(\kgot,N)$.

Let us give the following basic properties of the equivariant
cohomology spaces with support.

\begin{lem}\label{lem:basic-equi-coho-support}

$\bullet$ $\Hcal_F^{-\infty}(\kgot,N)=\{0\}$ if $F=\emptyset$.

$\bullet$ There is a natural map $\Hcal^{\infty}_F(\kgot,N)\to
\Hcal^{-\infty}_F(\kgot, N)$.

$\bullet$ There is a natural map $\Hcal^{-\infty}_F(\kgot,N)\to
\Hcal^{-\infty}(\kgot, N)$. If $F$ is compact, this map factors
through $\Hcal^{-\infty}_F(\kgot,N)\to \Hcal_c^{-\infty}(\kgot,N)$.

$\bullet$ If $F\subset F'$ are closed $K$-invariant subsets, there
is a restriction morphism
\begin{equation}\label{eq:res-local}
\res^{F',F}:\Hcal^{-\infty}_F(\kgot,N)\to
\Hcal^{-\infty}_{F'}(\kgot,N).
\end{equation}

$\bullet$ If $F_1$ and $F_2$ are two closed $K$-invariant subsets of
$N$, the wedge product of forms defines a natural product
\begin{equation}\label{eq:produit-equi-gene}
    \Hcal^{-\infty}_{F_1}(\kgot,N)\times \Hcal_{F_2}^{\infty}(\kgot,N)
    \stackrel{\wedge}{\longrightarrow}
    \Hcal_{F_1\cap F_2}^{-\infty}(\kgot,N).
\end{equation}

$\bullet$ If $F_1\subset F_1'$ and $F_2\subset F_2'$  are closed
$K$-invariant subsets, then the diagram
\begin{equation}\label{eq.2.reduction-gene}
\xymatrix@C=1cm{ \Hcal^{-\infty}_{F_1}(\kgot,N)\ar[d]^{\res^1} \!&\!\times
\!&\! \Hcal^{\infty}_{F_2}(\kgot,N)\ar[d]^{\res^2}\ar[r]^-{\wedge} &
\quad \Hcal_{F_1\cap F_2}^{-\infty}(\kgot,N) \ar[d]^{\res^{12}}\\
\Hcal_{F_1'}^{-\infty}(\kgot,N)   \!&\!\times \! &\!
\Hcal_{F_2'}^{\infty}(\kgot,N) \ar[r]^-{\wedge}   &
\quad \Hcal_{F_1'\cap F_2'}^{-\infty}(\kgot,N)
  }
\end{equation}
is commutative. Here the $\res^i$ are the restriction morphisms defined in (\ref{eq:res-local}).

\end{lem}

\begin{proof}
The proof of these properties are left to the reader. Note that the
product (\ref{eq:produit-equi-gene}) follows from (\ref{eq:prod-U-V}).
\end{proof}

\medskip

\subsection{Morphism $\p_F:\Hcal^{-\infty}(\kgot,N,N\setminus
F)\to\Hcal_F^{-\infty}(\kgot,N)$}\label{section:morphism-pF}

Now we define a natural map from
$\Hcal^{-\infty}(\kgot,N,N\setminus F)$ into
$\Hcal_F^{-\infty}(\kgot,N)$.

Let  $\beta\in \Acal^{-\infty}(\kgot,N\setminus F)$. If $\chi$ is
a $K$-invariant function on $N$ which is identically $1$ on a
neighborhood of $F$, note that $d\chi\beta$ defines an equivariant
form on $N$, since $d\chi$ is equal to $0$ in a neighborhood of $F$.

\begin{prop}\label{alphau}
For any open invariant neighborhood $U$ of $F$, we choose
$\chi\in\f(N)^K$ with support in $U$ and equal to $1$ in a
neighborhood of $F$.

$\bullet$ The map
\begin{equation}\label{eq:p-U-chi}
\p^{\chi}\left(\alpha,\beta\right)=\chi\alpha + d\chi\beta
\end{equation}
defines a homomorphism of complexes
$\p^{\chi}:\Acal^{-\infty}(\kgot,N,N\setminus F)\to
\Acal_U^{-\infty}(\kgot,N).$

In consequence, let $\alpha\in \Acal^{-\infty}(\kgot,N)$ be a
\emph{closed} equivariant form and $\beta\in
\Acal^{-\infty}(\kgot,N\setminus F)$ such that
$\alpha|_{N\setminus F}=D\beta$, then $\p^{\chi}(\alpha,\beta)$ is
a \emph{closed} equivariant form supported in $U$.

$\bullet$ The  cohomology class of  $\p^{\chi}(\alpha,\beta)$ in
$\Hcal_U^{-\infty}(\kgot,N)$ does not depend of $\chi$. We denote
this class by $\p_U(\alpha,\beta)\in \Hcal_U^{-\infty}(\kgot,N)$.

$\bullet$ For any neighborhoods $V\subset U$ of $F$, we have
$f_{U,V}\circ \p_V=\p_U$.
\end{prop}
\begin{proof}
The proof is similar to the proof of Proposition 2.3 in
\cite{pep-vergne1}. We repeat the main arguments. The equation
$\p^{\chi}\circ \,D_{\rm rel}= D \circ \p^{\chi}$
is immediate to check. In particular $\p^{\chi}(\alpha,\beta)$ is
closed, if $D_{\rm rel}\left(\alpha,\beta\right)=0$. For two different choices $\chi$ and
$\chi'$, we have $\p^{\chi}(\alpha,\beta)-\p^{\chi'}(\alpha,\beta)=
D\left((\chi-\chi')\beta\right)$. Since $\chi-\chi'=0$ in a neighborhood of $F$, the equivariant form
$(\chi-\chi')\beta$ is well defined on $N$ and with support in $U$. This proves the second  point.
Finally, the last point is immediate, since
$\p_U(\alpha,\beta)=\p_V(\alpha,\beta)=\p^\chi(\alpha,\beta)$ for
$\chi\in\f(N)^K$ with support in $V\subset U$.
\end{proof}\bigskip

\begin{defi}\label{def-alpha-beta}
Let $\alpha\in \Acal^{-\infty}(\kgot,N)$ be a \emph{closed}
equivariant form and \break $\beta\in
\Acal^{-\infty}(\kgot,N\setminus F)$ such that
$\alpha|_{N\setminus F}=D\beta$. We denote by $\p_F(\alpha,\beta)\in
\Hcal_F^{-\infty}(\kgot,N)$ the element defined by the sequence
$\p_U(\alpha,\beta)\in \Hcal_U^{-\infty}(\kgot,N),\, U\in\Fcal_F$.
We have then a morphism
\begin{equation}\label{eq:map-p}
\p_F:\Hcal^{-\infty}(\kgot,N,N\setminus F)\to
\Hcal_F^{-\infty}(\kgot,N).
\end{equation}

We will say that the element $\p_U(\alpha,\beta)\in
\Hcal_U^{-\infty}(\kgot,N)$  {\em is the $U$-component} of
$p_F(\alpha,\beta)$.

\end{defi}

\bigskip

In \cite{pep-vergne2}, we made the same construction for the
equivariant forms with smooth coefficients: for any closed
invariant subset $F$ we have a morphism
\begin{equation}\label{eq:map-p-smooth}
\p_F:\Hcal^{\infty}(\kgot,N,N\setminus F)\to
\Hcal_F^{\infty}(\kgot,N).
\end{equation}

The following proposition summarizes the functorial properties of
$\p$.

\begin{prop}

$\bullet$ If $F\subset F'$ are closed invariant subsets of $N$, then
the diagram
\begin{equation}\label{eq:fonctoriel-p-restriction}
\xymatrix@C=1cm{ \Hcal^{-\infty}(\kgot,N,N\setminus F)
\ar[d]^{\res_1}\ar[r]^-{\p_F} &
\Hcal_F^{-\infty}(\kgot,N) \ar[d]^{\res^1}\\
\Hcal^{-\infty}(\kgot,N,N\setminus F') \ar[r]^-{\p_{F'}} & \quad
\Hcal_{F'}^{-\infty}(\kgot,N) }
\end{equation}
is commutative. Here $\res_1$ and $\res^1$ are the restriction morphisms (see (\ref{eq:res-F-gene}) and
(\ref{eq:res-local})).

$\bullet$ If $F_1,F_2$ are closed invariant subsets of $N$, then the
diagram 
\begin{equation}\label{eq:fonctoriel-p-produit}
\xymatrix@C=8mm{ \Hcal^{-\infty}(\kgot,N,N\setminus F_1)
\ar[d]^{\p_{F_1}} \!\!\!\!\!\!\!\!\!\!\!\! & \times \quad
\Hcal^{\infty}(\kgot,N,N\setminus F_2)\ \
\ar[d]^{\p_{F_2}}\ar[r]^-{\lozenge} &
\ \  \Hcal^{-\infty}(\kgot,N,N\setminus (F_1\cap F_2)) \ar[d]^{\p_{F_1\cap F_2}}\\
\Hcal_{F_1}^{-\infty}(\kgot,N)  \!\!\!\!\!   & \times \quad \quad
\Hcal_{F_2}^{\infty}(\kgot,N)\quad  \ar[r]^-{\wedge} & \quad
\Hcal_{F_1\cap F_2}^{-\infty}(\kgot,N)}
\end{equation}
is commutative.

\end{prop}

\begin{proof}
The proof is entirely similar to the proof of Proposition 3.16 in
\cite{pep-vergne2} where one considers the case of smooth coefficients.

\end{proof}

\bigskip

If we take $F'=N$ in (\ref{eq:fonctoriel-p-restriction}), we see
that the  map $\p_F:\Hcal^{-\infty}(\kgot,N,N\setminus F)\to
\Hcal_F^{-\infty}(\kgot,N)$ factors the natural map
$\Hcal^{-\infty}(\kgot,N,N\setminus F)\to \Hcal^{-\infty}(\kgot,N)$.

\bigskip

If $F$ is compact, we can choose a  function $\chi$ with compact
support and identically equal to $1$ in a neighborhood of $F$.

\begin{defi}\label{def-alpha-beta-compact}
Let  $F$ be a {\em compact} $K$-invariant subset of $N$. Choose
$\chi\in\f(N)^K$ with compact support and equal to $1$ in a
neighborhood of $F$. Let $\alpha\in \Acal^{-\infty}(\kgot,N)$ be a
\emph{closed} equivariant form and  $\beta\in
\Acal^{-\infty}(\kgot,N\setminus F)$ such that
$\alpha|_{N\setminus F}=D\beta$. We denote by
$\p_c(\alpha,\beta)\in \Hcal_c^{-\infty}(\kgot,N)$ the class of
$p^{\chi}(\alpha,\beta)=\chi\alpha+d\chi \beta$ in
$\Hcal_c^{-\infty}(\kgot,N)$. We have then a morphism
\begin{equation}\label{eq:map-p-compact}
\p_c:\Hcal^{-\infty}(\kgot,N,N\setminus F)\to
\Hcal_c^{-\infty}(\kgot,N).
\end{equation}
\end{defi}

\section{The relative Chern character: the $\fgene$ case}
\label{sec:chg-sigma-gene}

Let $N$ be a manifold equipped with an action of a compact Lie
group $K$. Let $\Ecal=\Ecal^+\oplus\Ecal^-$ be an equivariant
$\Zbb_2$-graded complex vector bundle on $N$. We recall the
construction of the equivariant Chern character of $\Ecal$ that
uses Quillen's notion of super-connection (see \cite{B-G-V}).

We denote by $\Acal(N,\End(\Ecal))$ the algebra of $\End(\Ecal)$-valued
differential forms on $N$. Taking in account the $\Zbb_2$-grading of
$\End(\Ecal)$, the algebra $\Acal(N,\End(\Ecal))$ is a
$\Zbb_2$-graded algebra. The super-trace on $\End(\Ecal)$ extends to
a map $\str:\Acal(N,\End(\Ecal))\to \Acal(N)$.

Let $\A$ be a $K$-invariant super-connection  on $\Ecal$, and
$\F=\A^2$ its curvature, an element of $\Acal(N,\End(\Ecal))^+$.
Recall that, for $X\in \kgot$, the \emph{moment} of $\A$ is the
equivariant map $\mu^{\A}:\kgot \longrightarrow \Acal(N,\End(\Ecal))^+$
defined by the relation $\mu^{\A}(X)=\Lcal(X)-[\iota(VX),\A]$. We
define the \emph{equivariant curvature} of $\A$ by
\begin{equation}\label{eq:equi-curvature}
    \F(X)=\A^2+\mu^{\A}(X),\quad X\in\kgot.
\end{equation}

We usually denote simply by $\F$ the equivariant curvature, keeping
in mind that in the equivariant case, $\F$ is a function from
$\kgot$ to $\Acal(N,\End(\Ecal))^+$.

\begin{defi}\label{defi:chernequibete}
  The equivariant Chern character of  $(\Ecal,\A)$
  is the equivariant differential form on
  $N$ defined by
   $\ch(\A)=\str(\e^{\F})$
   (e.g. $\ch(\A)(X)=\str(\e^{\F(X)}))$.
\end{defi}
The form $\ch(\A)$ is equivariantly closed. We will use the
following transgression formulas (see \cite{B-G-V}, chapter 7,
\cite{pep-vergne1}).

\begin{prop}\label{transequi}
$\bullet$ Let $\A_t$, for $t\in \Rbb$, be  a one-parameter family
of $K$-invariant super-connections on $\Ecal$, and let
$\frac{d}{dt}\A_t\in \Acal(N,\End(\Ecal))^-$. Let $\F_t$ be the
equivariant curvature of $\A_t$. Then one has
\begin{equation}\label{transgression-equi}
\frac{d}{dt}\ch(\A_t)=D\left(\str\Big((\frac{d}{dt}\A_t)
\e^{\F_t}\Big)\right).
\end{equation}

$\bullet$ Let $\A(s,t)$ be a two-parameter family of $K$-invariant
super-connections. Here $s,t \in \Rbb$. We denote by $\F(s,t)$ the
equivariant curvature of $\A(s,t)$. Then:
\begin{eqnarray*}
\lefteqn{\frac{d}{ds}\str\Big((\frac{d}{dt}\A(s,t))\,
\e^{\F(s,t)}\Big)-
\frac{d}{dt}\str\Big((\frac{d}{ds}\A(s,t))\e^{\F(s,t)}\Big)} \nonumber\\
& &= D\left(\int_{0}^1\str\Big((\frac{d}{ds}\A(s,t)) \e^{u
\F(s,t)}(\frac{d}{dt}\A(s,t))\e^{(1-u)\F(s,t)}\Big)du\right).
\end{eqnarray*}
\end{prop}

In particular, the cohomology class defined by $\ch(\A)$  in
$\Hcal^\infty(\kgot,N)$ is independent of the choice of the
invariant super-connection $\A$ on $\Ecal$. By definition, this is the
equivariant Chern character $\ch(\Ecal)$ of $\Ecal$. By choosing
$\A=\nabla^+\oplus\nabla^-$ where $\nabla^{\pm}$ are connections on
$\Ecal^{\pm}$, this class is just $\ch(\Ecal^+)-\ch(\Ecal^-)$.
However, different choices of $\A$ define very different looking
representatives of $\ch(\Ecal)$.



%

\subsection{The relative Chern character of a morphism}
\label{subsec:chern-character}

Let $\Ecal=\Ecal^+\oplus \Ecal^-$ be an equivariant $\Zbb_2$-graded
complex vector bundle on $N$ and $\sigma: \Ecal^+ \to \Ecal^-$ be a
smooth morphism which commutes with the action of $K$. At each point
$n\in N$, $\sigma(n): \Ecal^+_n\to \Ecal^-_n$ is a linear map.
 The \emph{support} of $\sigma$ is the $K$-invariant closed subset
$$
\supp(\sigma)=\{n\in N\mid \sigma(n)\ {\rm is\ not \
invertible}\}.
$$

Let us recall the construction carried in \cite{pep-vergne2} of
the relative cohomology class $\ch_{\rm rel}(\sigma)$ in
$\Hcal^\infty(\kgot,N,N\setminus \supp(\sigma))$. The definition
will involve several choices. We choose invariant Hermitianstructures on
$\Ecal^{\pm}$ and  an invariant super-connection $\A$ on $\Ecal$ {\bf without $0$
exterior degree term}.




Introduce the odd Hermitian endomorphism of $\Ecal$ defined by
\begin{equation}\label{eq:v-sigma}
v_\sigma=
\left(\begin{array}{cc}
0 & \sigma^*\\
\sigma & 0\\
\end{array}\right).
\end{equation}
Then $ v_\sigma^2=\left(\begin{array}{cc}
\sigma^*\sigma& 0\\
0 &\sigma\sigma^*\\
\end{array}\right)
$ is a non negative even Hermitian  endomorphism of $\Ecal$ which
is positive definite on $N\setminus\supp(\sigma)$.




Consider the family of invariant super-connections
$\A^{\sigma}(t)=\A+i t\, v_\sigma,\ t\in\Rbb$ on $\Ecal$. The
equivariant curvature of $\A^{\sigma}(t)$ is thus the map
\begin{equation}\label{eq:F-A-sigma}
    \F(\sigma,\A,t)(X)=-t^2 v_\sigma^2+it
[\A,v_\sigma]+\A^2+\mu^\A(X).
\end{equation}
%

Consider the equivariant closed form
$\ch(\sigma,\A,t)(X):=\str\left(\e^{\F(\sigma,\A,t)(X)}\right)$ with
the  \emph{transgression form }
\begin{equation}\label{eq:trangression-form}
    \eta(\sigma,\A,t)(X):=-
\str\left(iv_\sigma \e^{\F(\sigma,\A,t)(X)}\right).
\end{equation}

In \cite{pep-vergne2}, we prove the following basic fact.

\begin{prop}\label{estimates}
The differential forms $\ch( \sigma,\A,t)(X)$ and
$\eta(\sigma,\A,t)(X)$ (and all their partial derivatives) tends to $0$ exponentially
fast when $t\to\infty$ uniformly on compact subsets of $(N\setminus
\supp(\sigma))\times\kgot$.
\end{prop}

\medskip

As $iv_\sigma=\frac{d}{dt}\A^\sigma(t)$, we have
$\frac{d}{dt}\ch(\sigma,\A,t)=-D(\eta(\sigma,\A,t))$. After
integration, it gives the following equality of equivariant
differential forms on $N$
\begin{equation}\label{eq:transgression-integral}
    \ch(\A)-\ch(\sigma,\A,t)=D\left(\int_0^t\eta(\sigma,\A,s)ds\right),
\end{equation}
since $\ch(\A)=\ch(\sigma,\A,0)$. Proposition \ref{estimates}
allows us to take the limit $t\to\infty$ in
(\ref{eq:transgression-integral}) on the open subset $N\setminus
\supp(\sigma)$. We get the following important lemma (see
\cite{Quillen85,pep-vergne1} for the non-equivariant case).

\begin{lem}\label{lem:quillen-equi}
We can define on $N\setminus \supp(\sigma)$ the equivariant
differential form with smooth coefficients
\begin{equation}
  \label{eq:beta}
  \beta(\sigma,\A)(X)=\int_{0}^\infty\eta(\sigma,\A,t)(X)dt,\quad
  X\in\kgot.
\end{equation}
We have $\ch(\A)|_{N\setminus
\supp(\sigma)}=D\left(\beta(\sigma,\A)\right)$.
\end{lem}

\medskip

We are in the situation of Subsection \ref{section:coho-support}.
The closed equivariant form $\ch(\A)$ on $N$ and the equivariant
form $\beta(\sigma,\A)$ on $N\setminus\supp(\sigma)$ define an even
relative cohomology class $[\ch(\A),\beta(\sigma,\A)]$ in
$\Hcal^\infty(\kgot,N,N\setminus \supp(\sigma))$.  We have the following

\begin{prop}[\cite{pep-vergne2}]\label{inds}

$\bullet$
 The class
 $[\ch(\A),\beta(\sigma,\A)]\in\Hcal^\infty(\kgot,N,N\setminus \supp(\sigma))$
 does not depend of the choice of $\A$, nor on the invariant Hermitian structure on 
$\Ecal$. We denote it by $\chr(\sigma)$.

 $\bullet$
 Let $F$ be an invariant closed subset of $N$. For $s\in [0,1]$,
 let $\sigma_s:\Ecal^+\to \Ecal^-$ be a
family of equivariant smooth morphisms such that
$\supp(\sigma_s)\subset F$. Then all classes $\chr(\sigma_s)$
coincide in $\Hcal^\infty(\kgot,N,N\setminus F)$.
\end{prop}

\subsection{The relative Chern character deformed by a one-form}
\label{subsec:deflambda}

If we allow equivariant cohomology with $C^{-\infty}$
coefficients, we can restrict further the support of the
equivariant Chern character of a bundle by modifying the term of
exterior degree $1$ of the super-connection. We use an idea
originally due to Witten \cite{Witten} and systematized in Paradan
\cite{pep1,pep2}, see also \cite{VergneICM}. The idea is to use as
further tool of deformation a  $K$-invariant {\bf real-valued}
one-form $\lambda$ on $N$, or equivalently a $K$-invariant vector
field on $N$. This tool was also used earlier in {\bf K}-theory by
Atiyah-Segal and Atiyah-Singer for deforming the zero section of $\T^*M$ (see
\cite{Atiyah-Segal68,Atiyah74}).

\medskip

We will need some definitions. Let $N$ be a $K$-manifold.
 Let $\lambda$ be a $K$-invariant {\bf real-valued} one-form on $N$.
 At each point $n\in N$,
$\lambda(n)\in \T^*_nN$.

\begin{defi}\label{defi:flambda}

 The one-form $\lambda$
 defines an equivariant map
\begin{equation}\label{eq:f-lambda}
    f_\lambda:N\to \kgot^*
\end{equation}
by $\langle f_\lambda(n),X\rangle=\langle
\lambda(n),V_n X\rangle$.

\end{defi}

\begin{defi}\label{critical}
$\bullet$ We define the invariant closed subset of $N$:
$$
C_{\lambda}=\{f_\lambda=0\}.
$$

$\bullet$ For any $K$-equivariant smooth morphism $\sigma:\Ecal^+\to \Ecal^-$ on
$N$, we define the invariant closed subset
$$
C_{\lambda,\sigma}= C_{\lambda}\cap \supp(\sigma).
$$
\end{defi}


\bigskip

In Section \ref{subsec:chern-character}, we have associated to a
$K$-equivariant smooth morphism $\sigma:\Ecal^+\to \Ecal^-$ the
relative class $\chr(\sigma)\in
\Hcal^{\infty}(\kgot,N,N\setminus\supp(\sigma))$. Here we consider
the cohomology space $\Hcal^{-\infty}(\kgot,N,N\setminus C_{\lambda,\sigma})$. We have the diagram
$$
\xymatrix{
 \Hcal^{-\infty}(\kgot,N,N\setminus C_{\lambda,\sigma} )\ar@{->}^{\res}[rrd] &  &     \\
\Hcal^{\infty}(\kgot,N,N\setminus \supp(\sigma)) &  &  \Hcal^{-\infty}(\kgot,N,N\setminus\supp(\sigma)).
\ar@{<-}^{e}[ll]
}
$$
where $\res$ is the restriction morphism, and $e$ is the morphism of extension of coefficients.

The goal of this section is to construct a class
$\chr(\sigma,\lambda)\in$ \break $\Hcal^{-\infty}(\kgot,N,N\setminus
C_{\lambda,\sigma})$  which is equal to $\chr(\sigma)$
in $\Hcal^{-\infty}(\kgot,N,N\setminus\supp(\sigma))$.

We choose $K$-invariant Hermitian structures on $\Ecal^{\pm}$ and  a
$K$-invariant super-connection $\A$ on $\Ecal$ without $0$
exterior degree term. Now we will modify $\A$ by introducing a $0$
exterior degree term and {\bf we will also  modify} its term of
exterior degree $1$.

Introduce the odd Hermitian endomorphism of $\Ecal$ defined by
$$
\left(\begin{array}{cc}
\lambda & \sigma^*\\
\sigma & \lambda\\
\end{array}\right)
=\lambda \,{\rm Id}_\Ecal+ v_\sigma.
$$

To simplify notations, we may write  $\lambda$ instead of $\lambda
\,{\rm Id}_\Ecal$. We consider the family of invariant super-connections
$$
\A^{\sigma,\lambda}(t)=\A+i t\,(\lambda +v_\sigma), \quad t\in \Rbb.
$$
The equivariant curvature of $\A^{\sigma,\lambda}(t)$ is
$\F(\sigma,\lambda,\A,t)= \F(\sigma,\A,t) +it D\lambda$, where
$\F(\sigma,\A,t)$ is the equivariant curvature of $\A^\sigma(t)$.
More explicitly,
$$
\F(\sigma,\lambda,\A,t)(X)= -t^2v_\sigma^2-it \langle
f_\lambda,X\rangle  +\mu^{\A}(X)+
 it[\A,v_\sigma]+\A^2+it d\lambda.
 $$

 In particular,  the term of $0$ exterior degree of
$\F(\sigma,\lambda,\A,t)(X)$ is the section of $\End(\Ecal)$ given
by $ -t^2 v_\sigma^2 -it \langle f_\lambda,X\rangle
+\mu^{\A}_{[0]}(X)$. We are interested by the equivariant
differential  form
$$
\ch(\sigma,\lambda,\A,
t):=\ch(\A^{\sigma,\lambda}(t))=\str\left(\e^{\F(\sigma,\lambda,\A,t)}\right).
$$

Then $\ch(\sigma,\lambda,\A,t)=\e^{it D\lambda} \ch(\sigma,\A,t)$ and
the transgression forms are:
\begin{eqnarray*}
   \eta(\sigma,\lambda,\A,t)(X)&=& -
   \str\left(i(v_\sigma+\lambda)
\e^{\F(\sigma,\lambda,\A,t)(X)}\right) \\
     &=& - \e^{it D\lambda(X)} \str\left(i(v_\sigma+\lambda )
\e^{\F(\sigma,\A,t)(X)}\right).
  \end{eqnarray*}

We repeat the argument of Section \ref{subsec:chern-character}.
The relation $\frac{d}{dt}\ch(\sigma,\lambda,\A,t)=$ \break
$-D(\eta(\sigma,\lambda,\A,t))$ gives after integration the
following equality of equivariant differential forms on $N$
\begin{equation}\label{eq:transgression-integral-gene}
    \ch(\A)-\ch(\sigma,\lambda,\A,t)=D\left(\int_0^t\eta(\sigma,\lambda,\A,s)ds\right).
\end{equation}

Now, we will show that we can take the limit of (\ref{eq:transgression-integral-gene})
when $t$ goes to $\infty$ on the open subset $N\setminus C_{\lambda,\sigma}$.

\medskip



In the following proposition,  $h^\sigma(n)\geq 0$ denotes the
smallest eigenvalue of $v_\sigma^2(n)$. We choose a metric on the
tangent bundle to $N$. Thus we obtain a norm $\|-\|$ on $\wedge
\T_n^*N\otimes \End(\Ecal_n)$ which varies smoothly with $n\in N$.

\begin{prop}\label{estimatesgen}
Let $\Kcal_1\times\Kcal_2$ be a compact subset of $N\times\kgot$.
Let $r$ be any positive integer. There exists a constant $\cst$
(depending of $\Kcal_1,\Kcal_2$, and
$r$) such that for any smooth function $Q$ on $\kgot$ supported in
$\Kcal_2$ we have
\begin{equation}\label{estNgen}
\Big|\!\Big|\int_{\kgot}\e^{\F(\sigma,\lambda,\A,t)(X)}
Q(X)dX\Big|\!\Big|(n) \leq \cst \ \frac{(1+t)^{\dim
N}}{(1+\|tf_\lambda(n)\|^2)^r} \ \|Q\|_{\Kcal_2,2r}\,
\e^{-h_\sigma(n)t^2}
\end{equation}
for all $t\geq 0$ and $n\in \Kcal_1$.
\end{prop}

\begin{proof}
We use the first estimate of Proposition
\ref{prop-estimation-generale-J-Q} of the Appendix to estimate the
integral
$$
\int_{\kgot}\e^{-it\langle f_\lambda,X\rangle}\e^{-t^2R(n)+S(n,X)+T(t,n)}
Q(X)dX
$$
with $R(n)= v_\sigma^2(n)$, $S(n,X)=\mu^{\A}(n)(X)$,
$T(t,n)=it[\A,v_\sigma](n) +itd\lambda(n)+\A^2(n)$. We obtain the
estimate (\ref{estNgen}) on $\Kcal_1$.
\end{proof}

\bigskip

The estimate (\ref{estNgen}) on the open subset
$$
N\setminus C_{\lambda,\sigma}=\{n\in N\ \vert\  h_\sigma(n)>0\ {\rm or}\ \| f_\lambda(n)\| >0\}.
$$
gives the following

\begin{coro}\label{coro:behaviour-gene}
$\bullet$ For a function $Q\in\f(\kgot)$ with compact support, the
element of $\Acal(N,\End(\Ecal))^+$ defined by
$\mathcal{I}_Q(t):=\int_{\kgot}\e^{\F(\sigma,\lambda,\A,t)(X)}
Q(X)dX$ tends rapidly to $0$ when $t\to\infty$, when restricted to
the open subset  $N\setminus C_{\lambda,\sigma}$.

\medskip

\noindent $\bullet$ The integral $\int_0^\infty\mathcal{I}_Q(t)dt$ defines a
smooth form on $N\setminus C_{\lambda,\sigma}$ with values in  $\End(\Ecal)$.

\medskip

\noindent $\bullet$ The equivariant Chern form $\ch(\sigma,\lambda,\A,t)$, when restricted
to $N\setminus C_{\lambda,\sigma}$, tends to $0$ as $t$ goes to $\infty$ in the space
$\Acal^{-\infty}(\kgot,N\setminus C_{\lambda,\sigma})$.

\medskip

\noindent $\bullet$ The family of smooth equivariant
forms, ${\rm J}_T:=\int_0^T\eta(\sigma,\lambda,\A,t)dt$, when
restricted to $N\setminus C_{\lambda,\sigma}$, admits a limit in
$\Acal^{-\infty}(\kgot,N\setminus C_{\lambda,\sigma})$ as
$T$ goes to infinity.

\end{coro}

\begin{proof}
We consider the estimates  (\ref{estNgen}) when the compact subset
$\Kcal_1$  is included in  $N\setminus C_{\lambda,\sigma}$. We can choose $c>0$ such that either
$h_\sigma(n)\geq c$ or $\|f_\lambda(n)\|^2\geq c$, for $n\in
\Kcal_1$. Then (\ref{estNgen}) gives for $t\geq 0$ and
$n\in\Kcal_1$ :
\begin{equation}\label{eq:estNgen-simple}
\Big|\!\Big|\ \mathcal{I}_Q(t)\Big|\!\Big|(n) \leq \cst
\|Q\|_{\Kcal_2,2r} (1+t)^{\dim N}\sup\Big(\frac{1}{(1+t^2c)^r},
\e^{-ct^2}\Big).
\end{equation}
Since $r$ can be chosen large enough,  (\ref{eq:estNgen-simple})
proves the first point: the integral
$\int_0^\infty\mathcal{I}_Q(t)dt$ converge on
$N\setminus C_{\lambda,\sigma}$. We have to check that
it defines a \emph{smooth} form with values in $\End(\Ecal)$. If
$D(\partial_n)$ is any differential operator acting on
$\Acal(N,\End(\Ecal))$, we have to show that, outside
$ C_{\lambda,\sigma}$,  the element of
$\Acal(N,\End(\Ecal))$ defined by
$\mathcal{I}^D_Q(t):=\int_{\kgot}D(\partial_n)\cdot
\e^{\F(\sigma,\lambda,\A,t)(X)} Q(X)dX$ tends rapidly to $0$ when
$t\to\infty$.  This fact follows from the estimate
(\ref{eq:estimate:D-fourier}) of Proposition
\ref{prop-estimation-generale-J-Q}. Then $\int_0^\infty\mathcal{I}_Q(t)dt$
is smooth and we have $D(\partial_n)\cdot
\int_0^\infty\mathcal{I}_Q(t)dt= \int_0^\infty\mathcal{I}^D_Q(t)$.

Since we have the relations $\int_\kgot \ch(\sigma,\lambda,\A,t)(X)Q(X)dX=
\str\left(\mathcal{I}_Q(t)\right)$ and $\int_\kgot{\rm J}_T(X)Q(X)dX=-i
\str\left((v_\sigma+\lambda)\int_0^T\mathcal{I}_Q(t)dt\right)$,
the last points follow from the first one.
\end{proof}

\medskip

\begin{rem}\label{rem:estimate-general} The estimate
(\ref{estNgen}) still holds when $Q$ is a smooth map from $\kgot$
into $\Acal(N)$ (or $\Acal(N,\End(\Ecal))$). See Remark
\ref{rem:estimate-Q-general}.
\end{rem}

\bigskip

We can then define on $N\setminus C_{\lambda,\sigma}$ the
equivariant differential odd form with $C^{-\infty}$ coefficients
\begin{equation}
  \label{eq:beta-gene}
  \beta(\sigma,\lambda,\A)=\int_0^\infty\eta(\sigma,\lambda,\A,t)dt.
\end{equation}
If we take the limit of (\ref{eq:transgression-integral-gene})
when $t$ goes to $\infty$ on the open subset $N\setminus C_{\lambda,\sigma}$,
we get
\begin{equation}\label{loca}
 \ch(\A)|_{N\setminus  C_{\lambda,\sigma}}=D\left(\beta(\sigma,\lambda,\A)\right)\quad {\rm in}\quad
\Acal^{-\infty}(\kgot,N\setminus  C_{\lambda,\sigma}).
\end{equation}

\begin{theo}\label{indsgen}
$\bullet$ The class
$$
\left[\ch(\A),\beta(\sigma,\lambda,\A)\right]\in 
\Hcal^{-\infty}(\kgot,N,N\setminus   C_{\lambda,\sigma})
$$
does not depend of the choice of $\A$, nor on the invariant Hermitian 
structure on $\Ecal$. We denote it by  $\chr(\sigma,\lambda)$.

\medskip

\noindent $\bullet$ Let $F$ be a closed $K$-invariant subset of $N$.
For $s\in [0,1]$,
 let  $\sigma_s:\Ecal^+\to \Ecal^-$ be a family of smooth $K$-equivariant morphisms
 and $\lambda_s$
a family of $K$-invariant one-forms such that $ C_{\lambda_s,\sigma_s}\subset F$.  Then all classes
$\chr(\sigma_s,\lambda_s)$ coincide in
$\Hcal^{-\infty}(\kgot,N,N\setminus F)$.
\end{theo}

\begin{proof}
Let us prove the first point. Let $\A_s, s\in [0,1]$, be a smooth
one-parameter family of invariant super-connections on $\Ecal$ without $0$
exterior degree terms. Let $\A(s,t)=\A_s+it (v_\sigma+\lambda)$.
Thus $\frac{d}{ds}\A(s,t)=\frac{d}{ds}\A_s$ and
$\frac{d}{dt}\A(s,t)=i(v_\sigma+\lambda)$. Let $\F(s,t)$ be the
equivariant curvature of $\A(s,t)$. We have
\begin{equation}\label{z0}
\frac{d}{ds}\ch(\A_s)=D(\gamma_s),\ \ {\rm with} \ \
\gamma_s=\str\Big((\frac{d}{ds}\A_s)\e^{\F(s,0)}\Big).
\end{equation}
We have
$\eta(\sigma,\lambda,\A_s,t)=-\str(i(v_\sigma+\lambda)\e^{\F(s,t)})$.
We apply the double transgression formula of Proposition
\ref{transequi}, and we obtain
\begin{equation}\label{z1}
\frac{d}{ds}\eta(\sigma,\lambda,\A_s,t)=-\frac{d}{dt}
\str\Big((\frac{d}{ds}\A_s)\e^{\F(s,t)}\Big)-D(\nu(s,t))
\end{equation}
 with
$\nu(s,t)(X)=\int_{0}^1 i\str\left((\frac{d}{ds}\A_s)
 \e^{u \F(s,t)(X)}
(v_\sigma+\lambda)\e^{(1-u)\F(s,t)(X)}\right)du$.

Let  $Q(X)$ be a smooth and compactly supported function on $\kgot$.
We consider the element of $\Acal(N,\End(\Ecal))$ defined by
$$
I_Q(u,s,t)=\int_{\kgot}i(\frac{d}{ds}\A_s)\e^{u \F(s,t)(X)}
(v_\sigma+\lambda)\e^{(1-u)\F(s,t)(X)} Q(X)dX,
$$
where $u,s\in [0,1]$ and $t\geq 0 $. Now
$$
\F(s,t)(X)= -it \langle f_\lambda,X\rangle  -t^2 v_\sigma^2
+\mu^{\A_s}(X)+\A_s^2+t[\A_s,v_\sigma]+it d\lambda.
$$

If we write $R=v_\sigma^2$, $S(X)=\mu^{\A_s}(X)$ and
$T(t)=\A_s^2+t[\A_s,v_\sigma]+it d\lambda$, our integral $I_Q(u,s,t)$
is equal to
$$\int_{\kgot} i \e^{-i t\langle f_\lambda,X\rangle}(\frac{d}{ds}\A_s)
\e^{u(-t^2 R+S(X)+T(t))}(v_\sigma+\lambda) \e^{(1-u)(-t^2 R+S(X)+T(t))} Q(X)dX.
$$


We apply Proposition \ref{suffitgendouble} of the Appendix. Let
$\Kcal_1\times\Kcal_2$ be a compact subset of $N\times\kgot$. Let
$r$ be any integer. There exists a constant $\cst>0$, such that:
for any $Q\in\f(\kgot)$ which is supported in $\Kcal_2$, we have
$$
\Big|\!\Big|I_Q(u,s,t)\Big|\!\Big|(n)
\leq
 \cst \, \|Q\|_{\Kcal_2,2r}\,
\frac{(1+t)^{\dim N}}{(1+t^2|\!|f_\lambda(n)|\!|^2)^r} \e^{-t^2h_\sigma(n)}
$$
for all $n\in \Kcal_1$, $t\geq 0$, and  $(u,s)\in [0,1]^2$.

If the compact subset $\Kcal_1$  is included in  $N\setminus C_{\lambda,\sigma}$,
we can choose $c>0$ such that either
$h_\sigma(n)\geq c$ or $\|f_\lambda(n)\|^2\geq c$, for $n\in
\Kcal_1$. Then we have
$$
\Big|\!\Big|I_Q(u,s,t)\Big|\!\Big|(n)\leq \cst \, \|Q\|_{\Kcal_2,2r}\, (1+t)^{\dim
N}\sup\Big(\frac{1}{(1+t^2c)^r}, \e^{-ct^2}\Big),
$$
for $n\in \Kcal_1$, $t\geq 0$, and $(u,s)\in [0,1]^2$.

Since $r$ can be chosen large enough, we have proved that
$I_Q(u,s,t)\in \Acal(N,\End(\Ecal))$, when restricted to the open
subset $N\setminus C_{\lambda,\sigma}$, is rapidly
decreasing in $t$ (uniformly in $(u,s)\in [0,1]^2$). Thanks to
Proposition \ref{suffitgendouble} of the Appendix, the same holds
for any partial derivative $D(\partial_n)I_Q(u,s,t)$. Since
$$
\int_\kgot \nu(s,t)(X)Q(X)dX=\str\Big(\int_0^1 I_Q(u,s,t)du\Big),
$$
the integral $\epsilon_s=\int_{0}^{\infty}\nu(s,t)dt$ defines for
any $s\in [0,1]$ a generalized equivariant differential form on
$N\setminus  C_{\lambda,\sigma}$.

So,  on the open subset
$N\setminus  C_{\lambda,\sigma}$, we can integrate
(\ref{z1})  in $t$ from $0$ to $\infty$:  we get
\begin{equation}\label{z3}
\frac{d}{ds}\beta(\sigma,\lambda,\A_s)=\gamma_s-D(\epsilon_s).
\end{equation}
If we put together (\ref{z0}) and (\ref{z3}), we obtain
\begin{eqnarray*}
\frac{d}{ds}\left(\ch(\A_s),\beta(\sigma,\lambda,\A_s)\right)&=&
\left(D(\gamma_s),\gamma_s - D(\epsilon_s)\right)\\
&=&D_{\rm rel}(\gamma_s,\epsilon_s).
\end{eqnarray*}
We have proved that the class
$\left[\ch(\A),\beta(\sigma,\lambda,\A)\right]\in
 \Hcal^{-\infty}(\kgot,N,N\setminus  C_{\lambda,\sigma})$
 does not depend of $s$.

We now prove the second point. We consider the invariant super-connection
$\A(s,t)=it(v_{\sigma_s}+\lambda_s)+\A$. Thus
$\frac{d}{ds}\A(s,t)=it\frac{d}{ds}(v_{\sigma_s}+\lambda_s)$ and
$\frac{d}{dt}\A(s,t)=i(v_{\sigma_s}+\lambda_s)$. Let $\F(s,t)$ be
the curvature of $\A(s,t)$. Let
$\eta(\sigma_s,\A,t)=-\str((\frac{d}{dt}\A(s,t))\e^{\F(s,t)})$. By
the double transgression formula, 
\begin{equation}\label{z4}
\frac{d}{ds}\eta(\sigma_s,\lambda_s,\A,t)
=-\frac{d}{dt}\str\Big(it\Big(\frac{d}{ds}(v_{\sigma_s}+\lambda_s)\Big)\e^{\F(s,t)}\Big)
-D(\nu(s,t))
\end{equation}
where the equivariant form $\nu(s,t)(X)$ is given by
$$
\nu(s,t)(X)= \int_{0}^1\str\left((it\frac{d}{ds}(v_{\sigma_s}+\lambda_s))
\e^{u \F(s,t)(X)}
(iv_{\sigma_s}+i\lambda_s)\e^{(1-u)\F(s,t)(X)}\right)du.
$$
We use again Proposition \ref{suffitgendouble} of the Appendix,
and we see that for any test function $Q(X)$,  the integral
$\int_\kgot\nu(s,t)(X)Q(X)dX$ is rapidly decreasing in $t$ on
$N\setminus F$. Then the integral
$\epsilon_s(X)=\int_{0}^{\infty}\nu(s,t)(X)dt$ defines for any
$s\in [0,1]$ a generalized equivariant differential form on
$N\setminus F$.

So,  on the open subset
$N\setminus F$, we can integrate
(\ref{z4})  in $t$ from $0$ to $\infty$. This gives
the relation $\frac{d}{ds}\beta(\sigma_s,\lambda_s,\A)=-D(\epsilon_s)$ and then
$$
\frac{d}{ds}\left(\ch(\A),\beta(\sigma_s,\lambda_s,\A)\right)=D_{\rm
rel}\left(0,\epsilon_s\right).
$$
The class of
$\left(\ch(\A),\beta(\sigma_s,\lambda_s,\A)\right)$ does not depend of $s$.

 With a similar proof, we see that it does not depend
 on the choice of invariant Hermitianstructure on $\Ecal$.

 \end{proof}

 \bigskip

\bigskip

In particular, we obtain the following corollary.

\begin{coro}\label{coro:ind}
$\bullet$ The classes $\chr(\sigma,\lambda)\in
 \Hcal^{-\infty}(\kgot,N,N\setminus  C_{\lambda,\sigma})$
 and $\chr(\sigma)\in
 \Hcal^{\infty}(\kgot,N,N\setminus \supp(\sigma))$ are equal in
 $\Hcal^{-\infty}(\kgot,N,N\setminus \supp(\sigma))$.

\medskip

\noindent $\bullet$  Let $\sigma$ be a $K$-invariant morphism. Let
$\lambda_0$ and $\lambda_1$ be two $K$-invariant one-forms such
that $\lambda_0(n)=\lambda_1(n)$ for any $n\in \supp(\sigma)$.
Then $ C_{\lambda_0,\sigma}= C_{\lambda_1,\sigma}=F$ and
$$
\chr(\sigma,\lambda_0)=\chr(\sigma,\lambda_1)\ \mathrm{in} \
\Hcal^{-\infty}(\kgot, N,N\setminus F).
$$

\noindent $\bullet$   Let $\lambda$ be a $K$-invariant one-form. Let
$\sigma_0:\Ecal^+\to \Ecal^-$ and $\sigma_1:\Ecal^+\to \Ecal^-$ be
two $K$-invariant morphisms such that  $\sigma_0(n)=\sigma_1(n)$ for
any $n\in C_\lambda$. Then $ C_{\lambda,\sigma_0}= C_{\lambda,\sigma_1}=F$ and
$$
\chr(\sigma_0,\lambda)=\chr(\sigma_1,\lambda) \ \mathrm{in} \
\Hcal^{-\infty}(\kgot, N,N\setminus F).
$$
\end{coro}

\begin{proof}
Indeed, for the first point, we consider the family
$\lambda_s=s\lambda$. It is obvious that $\chr(\sigma,0)=
\chr(\sigma)$ in $\Hcal^{-\infty}(\kgot,N,N\setminus
\supp(\sigma))$. For the second point, we consider the family
$\lambda_s=s\lambda_0+(1-s)\lambda_1$. For the third point, we
consider the family $\sigma_s=s\sigma_0+(1-s)\sigma_1$, and we
employ Proposition \ref{indsgen}.
\end{proof}

\subsection{The trivial bundle  and the ``non abelian localization theorem"}\label{subsec:trivial}

A particularly important case is the zero morphism $[0]$ between the
vector bundles $\Ecal^+=N\times \Cbb$ and  $\Ecal^-=N\times \{0\}$ :
$\Ecal^+$ is equipped with the connection $d$, then the invariant
real one-form $\lambda$ allows us to deform $d$ in $d+it \lambda$.

Then $c(t,\lambda)=\e^{i t D\lambda}$ is the corresponding Chern
character, with transgression form $\eta(t,\lambda)=-i\lambda \e^{it
D\lambda}$. Outside $C_{\lambda,[0]}=C_{\lambda}$, we can define the generalized
equivariant form
$\beta(\lambda)=-i\lambda\int_0^{\infty}\e^{itD\lambda}dt$.
The following formula, a particular case of Formula (\ref{loca}), 
is the principle of  the  Witten localization formula \cite{Witten}.

\begin{theo}(Non abelian localization theorem)
 We
have
$$
1=D(\beta(\lambda))
$$
outside $C_{\lambda}$.
 \end{theo}
   Morally, we have
$\beta(\lambda)=\frac{\lambda}{D\lambda}$, so that
$D(\beta(\lambda))=\frac{D\lambda}{D\lambda}=1$.

\begin{defi}\label{def:ch-O-lambda}
The class defined by $(1,\beta(\lambda))$ in
$\Hcal^{-\infty}(\kgot,N,N\setminus C_\lambda)$ is denoted
$$
\pr(\lambda).
$$
\end{defi}

Let us rewrite Theorem \ref{indsgen} in this particular case.

\begin{theo}\label{Prel-beta}
 Let $F$ be a closed $K$-invariant subset of $N$.
For $s\in [0,1]$, let  $\lambda_s$ be
a family of $K$-invariant one-forms such that $ C_{\lambda_s}\subset F$.  Then all classes
$\pr(\lambda_s)$ coincide in
$\Hcal^{-\infty}(\kgot,N,N\setminus F)$.
\end{theo}

Let us give some very simple examples.

$\bullet$
 Let $N:=\Rbb^2$ with coordinates $(x,y)$. The circle group
$S^1$ acts by rotations. We identify its Lie algebra ${\rm Lie}(S^1)$ with $\Rbb$.
The element  $X\in {\rm Lie}(S^1)$ produces the vector field $VX=X(y\partial_x-x\partial_y)$.
Let $\lambda=xdy-ydx$. Then $C_\lambda=\{(0,0)\}$. We have $D\lambda(X)=2 dx\wedge dy+ X(x^2+y^2)$. Thus
\[\begin{array}{ll}
\beta(\lambda)(X) &=-i\lambda\int_0^{\infty}\e^{it(X (x^2+y^2)+2 dx\wedge dy)}dt\\
 &=-i\frac{xdy-ydx}{x^2+y^2}\int_0^{\infty}\e^{itX}dt.\end{array}\]

The generalized function $X\mapsto -i\int_0^{\infty}\e^{itX}dt$ is equal to the boundary value, 
denoted by $\frac{1}{X+i 0}$, of the function $1/z$. We obtain
$$
\pr(\lambda)=\left[1,  \frac{1}{X+i0}\frac{xdy-ydx}{x^2+y^2}\right]
$$
in $\Hcal^{-\infty} ({\rm Lie}(S^1), \Rbb^2,\Rbb^2\setminus\{(0,0)\}).$

\bigskip

$\bullet$  Let $N:=\T^*S^1=S^1\times \Rbb$. The circle group $S^1$ acts
freely by rotations on $S^1$. If $(e^{i\theta},\xi)$ is a point of
$\T^*S^1$ with $\xi\in \Rbb$, the Liouville $1$-form is $\lambda:=-\xi
d\theta$. The element  $X\in {\rm Lie}(S^1)$ produces the vector field 
$VX=-X\partial_{\theta}$. The critical set  $C_\lambda$ is $S^1$ embedded in $\T^*S^1$ as the zero section.
 We have $D\lambda(X)=d\theta d\xi - X\xi$. Thus
\[\begin{array}{ll}
\beta(\lambda)(X) &=-i\lambda\int_0^{\infty}\e^{it(-X \xi+ d\theta\wedge d\xi)}dt\\
 &=i\xi{d\theta}\int_0^{\infty}\e^{-it\xi X}dt\end{array}\]

We obtain $\pr(\lambda)=[1,  \beta(\lambda)]$
in $\Hcal^{-\infty} ({\rm Lie}(S^1), \T^*S^1,\T^*S^1\setminus S^1)$
with
\begin{eqnarray*}\label{parphi}
\beta(\lambda)(X)&=&\frac{1}{X-i 0} d\theta \hspace{1cm} {\rm if}\,\, \xi>0,\\
\beta(\lambda)(X)&=&\frac{1}{X+i 0} d\theta  \hspace{1cm} {\rm if}\,\, \xi<0.
\end{eqnarray*}

\subsection{Tensor product }\label{tensor}

Let $\Ecal_1, \Ecal_2$ be two  equivariant $\Zbb_2$-graded vector
bundles on $N$. The space $\Ecal_1\otimes \Ecal_2$ is a
$\Zbb_2$-graded vector bundle with even part $\Ecal_1^+\otimes
\Ecal_2^+\oplus \Ecal_1^-\otimes \Ecal_2^-$ and odd part
$\Ecal_1^-\otimes \Ecal_2^+\oplus \Ecal_1^+\otimes \Ecal_2^-$. The
super-algebra $\Acal^*(N,\End(\Ecal_1\otimes \Ecal_2))$ can be
identified with
$\Acal^*(N,\End(\Ecal_1))\otimes\Acal^*(N,\End(\Ecal_2))$ where
the tensor is taken in the sense of super-algebras.

\medskip

Let $\sigma_1:\Ecal_1^+\to \Ecal_1^-$ and $\sigma_2:\Ecal_2^+\to
\Ecal_2^-$ be two smooth equivariant morphisms. With the help of
invariant Hermitian structures, we define the morphism
$$
\sigma_1\odot \sigma_2: \left(\Ecal_1\otimes \Ecal_2\right)^+
\longrightarrow \left(\Ecal_1\otimes \Ecal_2\right)^-
$$
by $\sigma_1\odot \sigma_2:= \sigma_1\otimes {\rm
Id}_{\Ecal_2^+}+{\rm Id}_{\Ecal_1^+}\otimes \sigma_2+{\rm
Id}_{\Ecal_1^-}\otimes\sigma_2^*+ \sigma_1^*\otimes {\rm
Id}_{\Ecal_2^-}$.

\medskip

Let $v_{\sigma_1}, v_{\sigma_2}$  and $v_{\sigma_1\odot \sigma_2}$
be the odd Hermitian endomorphisms associated to $\sigma_1,
\sigma_2$ and  $\sigma_1\odot \sigma_2$ (see (\ref{eq:v-sigma})).
Since $v_{\sigma_1\odot \sigma_{2}}^2= v_{\sigma_1}^2\otimes {\rm
Id}_{\Ecal_2}+{\rm Id}_{\Ecal_1}\otimes v_{\sigma_2}^2$, it
follows that $\supp(\sigma_1\odot \sigma_2)= \supp(\sigma_1)\cap
\supp(\sigma_2)$.

\medskip

We proved in \cite{pep-vergne2} that the relative Chern character
is multiplicative : the equality 
$\chr(\sigma_1\odot\sigma_2)=\chr(\sigma_1)\diamond\chr(\sigma_2)$
holds in $\Hcal^{\infty}(\kgot,N,N\setminus\supp(\sigma_1\odot\sigma_2))$.
This property admits the following generalization.

\begin{theo}\label{theo:chrel-produit-gene}{\rm \bf (The relative Chern character is
multiplicative)}
Let $\sigma_1,\sigma_2$ be two equivariant morphisms on $N$. Let
$\lambda$ be an invariant one form on $N$. The relative equivariant cohomology classes
\begin{itemize}
  \item $\chr(\sigma_1,\lambda)\in
  \Hcal^{-\infty}(\kgot,N,N\setminus C_{\lambda,\sigma_1})$,
  \item $\chr(\sigma_2)\in
  \Hcal^\infty(\kgot,N,N\setminus\supp(\sigma_2))$,
  \item
  $\chr(\sigma_1\odot\sigma_2,\lambda)\in
  \Hcal^{-\infty}(\kgot,N,N\setminus C_{\lambda,\sigma_1\odot\sigma_2})$
\end{itemize}
satisfy the following equality
$$
\chr(\sigma_1\odot\sigma_2,\lambda)=\chr(\sigma_1,\lambda)\diamond\chr(\sigma_2)
$$
in
$\Hcal^{-\infty}(\kgot,N,N\setminus C_{\lambda,\sigma_1\odot\sigma_2})$.
Here $\diamond$ is the product of relative classes
(see (\ref{eq:produit-relatif-gene})).
\end{theo}

In Subsection \ref{subsec:trivial}, we considered the zero morphism
$[0]:N\times \Cbb\to N\times \{0\}$. Since for any morphism $\sigma$ we
have $[0]\odot\sigma=\sigma$,  we get the following

\begin{coro}\label{coro:chr-lambda}
For any invariant one form $\lambda$, we have
$$
\chr(\sigma,\lambda)=\pr(\lambda)\diamond\chr(\sigma)
$$
in $\Hcal^{-\infty}(\kgot,N,N\setminus C_{\lambda,\sigma})$.
\end{coro}

The remaining part of this section is devoted to the proof of
Theorem \ref{theo:chrel-produit-gene}.

\medskip

For $k=1,2$, we choose  invariant super-connections $\A_k$, without
$0$ exterior degree terms on the $\Zbb_2$-graded vector bundles
$\Ecal_k$. We consider the closed equivariant forms
$$
c_1(t):=\ch(\sigma_1,\lambda,\A_1,t),\quad
c_2(t):=\ch(\sigma_2,\A_2,t)
$$
and the transgression forms
$$
\eta_1(t):=\eta(\sigma_1,\lambda,\A_1,t),\quad
\eta_2(t):=\eta(\sigma_2,\A_2,t)
$$
so that $\frac{d}{dt}(c_k(t))=- D(\eta_k(t))$.

Let $\beta_1=\int_0^{\infty}\eta_1(t)dt$ : it is an equivariant
form on $U_1:=N\setminus C_{\lambda,\sigma_1}$ with
generalized coefficients. Let $\beta_2=\int_0^{\infty}\eta_2(t)dt$ :
it is an equivariant form on $U_2:=N\setminus\supp(\sigma_2)$ with
smooth coefficients.  The representatives of
$\chr(\sigma_1,\lambda)$ and $\chr(\sigma_2)$ are respectively
$(c_1(0),\beta_1)$, $(c_2(0),\beta_2)$.

\medskip

For the symbol $\sigma_1\odot\sigma_2$, we consider $\A(t)=\A+it
(\lambda+ v_{\sigma_1\odot \sigma_{2}})$ where $\A=\A_1\otimes
\id_{\Ecal_2}+\id_{\Ecal_1}\otimes \A_2$. Then
$\ch(\A)=c_1(0)c_2(0)$. Furthermore, it is easy to see that the
transgression form  for the family $\A(t)$ is
$$
\eta(t)=\eta_1(t)c_2(t)+c_1(t)\eta_2(t).
$$
Let
$\beta_{12}=\int_{0}^{\infty}\eta(t)dt$ : it is an equivariant form
on
\begin{eqnarray*}
U &:=&N\setminus(\supp(\sigma_1)\cap \supp(\sigma_2)\cap C_\lambda)\\
&=& U_1\bigcup U_2
\end{eqnarray*}
with generalized coefficients. A representative of
$\chr(\sigma_1\odot\sigma_2,\lambda)$ is \break
$(c_1(0)c_2(0),\beta_{12})$.

We need the following lemma.

\begin{lem}\label{lem:I-Phi-defined}
$\bullet$ The integral
$$
  {\rm I}_2 :=\int\!\!\!\int_{0\leq s\leq  t} \eta_1(s) \wedge \eta_2(t)ds \,dt
$$
defines an equivariant form with smooth coefficients on $U_2$.

$\bullet$ The integral
$$
  {\rm I}_1 :=\int\!\!\!\int_{0\leq t\leq  s} \eta_1(s) \wedge \eta_2(t)ds \,dt
$$
defines\footnote{The integral ${\rm I}_1$ is the limit when $T\to
\infty$ of the family $(\int\!\!\!\int_{0\leq t\leq  s\leq T} \eta_1(s) \wedge
\eta_2(t)dsdt)_{T>0}$ of equivariant forms with smooth
coefficients.} an equivariant form with  generalized coefficients on
$U_1$.

$\bullet$ We have the relations  $D{\rm I}_1=\beta_{12}-\beta_1c_2(0)$  on $U_1$
and $D{\rm I}_2=-\beta_{12}+\beta_2c_1(0)$ on $U_2$.
\end{lem}


\begin{proof}
Let $\Kcal_2$ be a compact subset of $U_2$. Let $h_2 > 0$ such
that $h_{\sigma_2}(n)\geq h_2$ for $n\in\Kcal_2$. Let $\Kcal$ be a
compact subset of $\kgot$. From Proposition \ref{estimates}, we
know that there exists constants $\cst$ and $\cst'$ (depending of
$\Kcal_2$, $\Kcal$) such that: for $(X,n)\in \Kcal\times \Kcal_2$
we have
\begin{equation}\label{eq:eta-1-U2}
\Big|\!\Big| \eta_2(t)(X)\Big|\!\Big|(n)\leq \cst\, (1+t)^{\dim N}
\e^{-h_2t^2},\quad \mathrm{for\ all\ }t\geq 0,
\end{equation}
and
\begin{equation}\label{eq:eta-2-U2}
\Big|\!\Big| \eta_1(s)(X)\Big|\!\Big|(n)\leq \cst' \, (1+s)^{\dim
N},\quad \mathrm{for\ all\ }s\geq 0.
\end{equation}
Then, when $0\leq s \leq  t$, we have, on $\Kcal_2$: $ \| \eta_1(s)
\wedge \eta_2(t)\|\leq \cst"(1+t)^{2\dim N}\, \e^{-h_2t^2} $. So the
integral ${\rm I}_2$ is absolutely convergent on $0\leq s\leq t$.
Since similar majoration holds for the partial derivative of
$\eta_k$ (relatively to the variables $n\in N$ and $X\in\kgot$), the
integral ${\rm I}_2$ defines a smooth map from $\kgot$ into
$\Acal(U_2)$.

\medskip

Let us prove the second point. Let $\Kcal$ be a compact subset of
$\kgot$. For any test function $Q(X)$ on $\kgot$ supported in
$\Kcal$, let us estimate the form
$\gamma(s,t,Q):=\int_{\kgot}\eta_1(s)(X)\eta_2(t)(X)Q(X)dX$ on
$0\leq t\leq s$ and on a compact subset $\Kcal_1$ of $U_1$. We have
$$
\gamma(s,t,Q)=\int_\kgot \e^{-is \langle f_\lambda,X\rangle }
\Upsilon(s,t,X)Q(X)dX
$$
where $ \Upsilon(s,t,X)= \e^{is d\lambda} \str\left(-i
(v_{\sigma_1}+\lambda)
\e^{\F(\sigma_1,\A_1,s)(X)}\right)\wedge \eta_2(t)(X)$. Let $r$ be a
positive integer. If we use the estimates of Proposition \ref{prop-estimation-generale-J-Q}
(see also Remark \ref{rem:estimate-Q-general}),  we get
$$
\Big|\!\Big| \gamma(s,t,Q)\Big|\!\Big|(n)\leq \cst \,\Big|\!\Big|\eta_2(t)Q\Big|\!\Big|_{\Kcal,2r} (n)
\frac{(1+s)^{\dim N}}{(1+s^2\|f_\lambda(n)\|^2)^{r}}\e^{-h_1(n)s^2},
$$
for all $t,s\geq 0$ and $n\in\Kcal_1$. Here $\cst$ is a constant depending of
$r,\Kcal_1,\Kcal$, and $h_1(n)\geq 0$ is the smallest eigenvalue of
$v_{\sigma_1}^2(n)$.

The term  $\|\eta_2(t)Q\|_{\Kcal,2r}(n)$ is smaller than
$\|Q\|_{\Kcal,2r}\|\eta_2(t)\|_{\Kcal,2r}(n)$. If we use the
second point of Proposition \ref{prop-estimation-generale} of the
Appendix,  we see that
$$
\Big|\!\Big|\eta_2(t)\Big|\!\Big|_{\Kcal,2r}(n)\leq
\cst' (1+t)^{\dim N},\quad {\rm for\ all}\ n\in \Kcal_1,\ t\geq 0.
$$
Finally, for $0\leq t\leq s$ and $n\in \Kcal_1$, we have:
\begin{equation}\label{eq:gamma-U1}
\Big|\!\Big| \gamma(s,t,Q)\Big|\!\Big|(n)\leq \cst'' \Big|\!\Big|Q\Big|\!\Big|_{\Kcal,2r}
\frac{(1+s)^{2\dim N}}{(1+s^2\|f_\lambda(n)\|^2)^{r}}\e^{-h_1(n)s^2}.
\end{equation}

If the compact subset $\Kcal_1$  is included in  $N\setminus
 C_{\lambda,\sigma_1}$, we can choose $c>0$ such that
either $h_1(n)\geq c$ or $\|f_\lambda(n)\|^2\geq c$, for $n\in
\Kcal_1$. Then we have
$$
\Big|\!\Big| \gamma(s,t,Q)\Big|\!\Big|(n)\leq \cst'' \,
\|Q\|_{\Kcal,2r}\, (1+s)^{2\dim N}\sup\Big(\frac{1}{(1+s^2c)^r},
\e^{-cs^2}\Big),
$$
for $n\in \Kcal_1$ and $0\leq t\leq s$.

Since $r$ can be chosen large enough, we have proved that the
integral of the differential forms $\gamma(s,t,Q)$ on  $0\leq
t\leq s$ is absolutely convergent. Since similar majoration holds
for the partial derivative of $\gamma(s,t,Q)$  (relatively to the
variables $n\in N$ and $X\in\kgot$). the integral ${\rm I}_1(X)$
defines a $\fgene$-map from $\kgot$ to $\Acal(U_1)$ by the
relation $\int_\kgot{\rm I}_1(X)Q(X)dX:=\int\!\!\!\int_{0\leq t\leq  s}\gamma(s,t,Q) dsdt$.

For the last point we compute
\begin{eqnarray*}
D({\rm I}_1)&=&D\left(\int\!\!\!\int_{0\leq t \leq  s}\eta_1(s)\eta_2(t)ds \,dt\right)\\
 &=& \int\!\!\!\int_{0\leq t \leq  s}\Big(D\eta_1(s)\eta_2(t)-\eta_1(s)D\eta_2(t)\Big) ds \,dt.
\end{eqnarray*}
Now we use $D(\eta_j(s))=-\frac{d}{ds}c_j(s)$, so that we obtain
\begin{eqnarray*}
D({\rm I}_1 )&=& \int\!\!\!\int_{0\leq t \leq  s}
\left((-\frac{d}{ds}c_1(s))\eta_2(t)+\eta_1(s)(\frac{d}{dt}
c_2(t))\right) ds\,dt \\
  &=&
\Big(\int_{0}^{\infty}
c_1(t)\eta_2(t)dt+\int_0^{\infty}\eta_1(s)c_2(s)ds\Big)- c_2(0)\beta_1\\
&=&
\beta_{12}- c_2(0)\beta_1.
\end{eqnarray*}
 Similarly, we
compute $D({\rm I}_2)=-\beta_{12}+c_1(0)\beta_2$.
\end{proof}

\bigskip

Let $\Phi_1+\Phi_2= {\rm 1}_{U}$ be
a partition of unity subordinate to the decomposition $U=U_1\cup
U_2$ : the functions $\Phi_k$ are supposed $K$-invariant.
We consider ${\rm I}_\Phi:=\Phi_1 {\rm I}_1 -\Phi_2 {\rm I}_2$ which
is an equivariant form with generalized coefficients on $U$.
We now prove that
\begin{equation}\label{prodrel}
\Big(c_1(0),\beta_{1}\Big)\diamond_\Phi \Big(c_2(0),\beta_{2}\Big)-
\Big(c_1(0)c_2(0),\beta_{12}\Big)=D_{\rm rel}\Big(0,{\rm I}_{\Phi}\Big).
\end{equation}
Indeed the product $(c_1(0),\beta_{1})\diamond_\Phi
(c_2(0),\beta_{2})$ is equal to \break $\Big(c_1(0)c_2(0),\Phi_1 \beta_1
c_2(0)+c_1(0)\Phi_2\beta_2-d\Phi_1 \beta_1\beta_2\Big)$,
so that the first member of Equality (\ref{prodrel}) is
$\Big(0,\Phi_1 \beta_1 c_2(0)+c_1(0)\Phi_2\beta_2-d\Phi_1
\beta_1\beta_2-\beta_{12}\Big)$.
Thus we need to check that
\begin{equation}\label{eq:fondamentale}
 -D({\rm I}_{\Phi})=\Phi_1 \beta_1
c_2(0)+c_1(0)\Phi_2\beta_2-d\Phi_1 \beta_1\beta_2-\beta_{12}.
\end{equation}
Using the last point of Lemma \ref{lem:I-Phi-defined}, we have
\begin{eqnarray*}
-D({\rm I}_{\Phi})&=& d\Phi_2{\rm I}_2-d\Phi_1{\rm I}_1 + \Phi_2 D{\rm I}_2- \Phi_1D{\rm I}_1\\
&=&-d\Phi_1({\rm I}_2+{\rm I}_1)+ \Phi_2 (-\beta_{12}+c_1(0)\beta_2)- \Phi_1(\beta_{12}-c_2(0)\beta_1)\\
&=&-d\Phi_1\beta_1\beta_2 -\beta_{12}+\Phi_2 c_1(0)\beta_2+ \Phi_1c_2(0)\beta_1.
\end{eqnarray*}
which was the equation to prove. Here we have used that $\Phi_1+\Phi_2={\rm 1}_{U}$, hence $d\Phi_2=-d\Phi_1$.


\subsection{The Chern character deformed by a one form}\label{sec:ch-good-lambda}

Let $\sigma:\Ecal^+\to\Ecal^-$ be an equivariant morphism on $N$,
and $\lambda$ be an invariant one form on $N$. Following  Section
\ref{section:morphism-pF}, we consider the image of the
relative class $\chr(\sigma,\lambda)$ through the map
$$
\Hcal^{-\infty}(\kgot,N,N\setminus  C_{\lambda,\sigma})\to
\Hcal_{C_{\lambda,\sigma}}^{-\infty}(\kgot,N).
$$
The following theorem
summarizes the construction of the image.

\begin{theo}\label{prop:ch-good-lambda}

$\bullet$ For any invariant neighborhood $U$ of $C_{\lambda,\sigma}$,
take $\chi\in\f(N)^K$ which is equal to 1 in a
neighborhood of $C_{\lambda,\sigma}$ and with support
contained in $U$. Then
\begin{equation}\label{sigmaAlambda}
 c(\sigma,\lambda,\A,\chi)=\chi\, \ch(\A) + d\chi
\,\beta(\sigma,\lambda,\A)
\end{equation}
is an equivariant closed differential form with generalized
coefficients, supported in $U$. Its cohomology class
$c_U(\sigma,\lambda)\in \Hcal^{-\infty}_U(\kgot,N)$ does not depend
of the choice of $\A,\chi$ and the invariant Hermitian structures on
$\Ecal^{\pm}$. Furthermore, the inverse family $c_U(\sigma,\lambda)$
when $U$ runs over the neighborhoods of $C_{\lambda,\sigma}$ defines a class
$$
\chs(\sigma,\lambda)\in \Hcal^{-\infty}_{C_{\lambda,\sigma}}(\kgot,N).
$$

\noindent $\bullet$ The image of $\chs (\sigma,\lambda)$ in
$\Hcal^{-\infty}_{\hbox{\rm \tiny Supp}(\sigma)}(\kgot,N)$ is equal to $\chs(\sigma)$.

\medskip

\noindent $\bullet$  Let $F$ be a closed $K$-invariant subset of $N$. For $s\in
[0,1]$, let  $\sigma_s:\Ecal^+\to \Ecal^-$ be a family of smooth $K$-equivariant morphisms
and $\lambda_s$ a family of $K$-invariant one-forms such that
$C_{\lambda_s,\sigma_s}\subset F$.  Then all classes
$\chs(\sigma_s,\lambda_s)$ coincide in $\Hcal^{-\infty}_F(\kgot,N)$.
\end{theo}

\begin{defi}\label{def:ch-c-sigma-lambda}
When $C_{\lambda,\sigma}$ is a \emph{compact} subset of
$N$, we define
$$
\chc(\sigma,\lambda)\in \Hcal^{-\infty}_{c}(\kgot,N)
$$
as the image of $\chs(\sigma,\lambda)\in
\Hcal^{-\infty}_{C_{\lambda,\sigma}}(\kgot,N)$ in
$\Hcal^{-\infty}_{c}(\kgot,N)$.  A representative
of  $\chc(\sigma,\lambda)$ is given be the equivariant form
$c(\sigma,\lambda,\A,\chi)$, with $\chi$ compactly supported.

\end{defi}

When $\sigma$ is \emph{elliptic}, we have already a class
$\ch_{c}(\sigma)\in \Hcal^{\infty}_{c}(\kgot,N)$ with compact
support. If we use the second point of Theorem
\ref{prop:ch-good-lambda}, one sees that
$$
\ch_{c}(\sigma)=\ch_{c}(\sigma,\lambda)\quad \mathrm{in}\quad
\Hcal^{-\infty}_{c}(\kgot,N),
$$
for any invariant one-form $\lambda$. So the class with compact
support $\ch_{c}(\sigma,\lambda)$ will be of interest when $\sigma$
is \emph{not elliptic}, but $C_{\lambda,\sigma}$ is
\emph{compact}.

\medskip

In Subsection \ref{subsec:trivial}, we considered the zero morphism
$[0]:N\times \Cbb\to N\times \{0\}$ and
 its relative Chern character $\pr(\lambda)\in
\Hcal^{-\infty}(\kgot,N,N\setminus C_\lambda)$.  The associated
generalized equivariant class in $\Hcal^{-\infty}_{C_\lambda}(\kgot,N)$
will be denoted $\Par(\lambda)$ : this class was defined in \cite{pep1,pep2}.
We repeat Theorem \ref{prop:ch-good-lambda}
for this special case.

Recall that $\beta(\lambda)(X):= -i\lambda\int_0^\infty \e^{itD(\lambda)(X)}dt$ is an
equivariant form with generalized coefficients on $N\setminus C_\lambda$.

\begin{theo}\label{prop:equi-lambda}\cite{pep1}
$\bullet$ Let $\chi\in\f(N)$ be a $K$-invariant function  which is
equal to 1 in a neighborhood of $C_\lambda$ and with support
contained in $U$. The equivariant differential form
$\Par(\lambda,\chi)=\chi + d\chi \,\beta(\lambda)$
is an {\em equivariantly closed}  differential form with
$C^{-\infty}$ coefficients, supported in $U$. Its cohomology class
$\Par_U(\lambda)\in \Hcal^{-\infty}_U(\kgot,N)$ does not depend of the
choice of $\chi$. Furthermore, the inverse family $\Par_U(\lambda)$
when $U$ runs over the neighborhoods of $C_{\lambda}$ defines a
class
\begin{equation}\label{eq:par-lambda}
\Par(\lambda)\in \Hcal^{-\infty}_{C_{\lambda}}(\kgot,N).
\end{equation}

\noindent $\bullet$  The image of this class in $\Hcal^{-\infty}(\kgot,N)$ coincides with $1$.

\medskip

\noindent $\bullet$  Let $F$ be a closed $K$-invariant subset of $N$. For $s\in
[0,1]$, let $\lambda_s$ be a family of $K$-invariant one-forms such
that $ C_{\lambda_s}\subset F$. Then all classes $\Par(\lambda_s)$
coincide in $\Hcal^{-\infty}_F(\kgot,N)$.

\end{theo}

\medskip

We proved in Theorem \ref{theo:chrel-produit-gene} (see also
Corollary \ref{coro:chr-lambda}) that
$\chr(\sigma_1\odot\sigma_2,\lambda)$ is equal to the product
$\chr(\sigma_1,\lambda)\diamond\chr(\sigma)$ in
$\Hcal^{-\infty}(\kgot,N,N\setminus C_{\lambda,\sigma_1\odot\sigma_2})$. If we use the
commutativity of the diagram (\ref{eq:fonctoriel-p-produit})  for the closed invariant subsets
$F_1:=C_{\lambda,\sigma_1}$, $F_2:=\supp(\sigma_2)$ and $F_1\cap F_2= C_{\lambda,\sigma_1\odot\sigma_2}$,
we get

\begin{theo}\label{theo:produit-chg-lambda}
We have the following relation in $\Hcal^{-\infty}_{C_{\lambda,\sigma_1\odot\sigma_2}}(\kgot,N)$ :
\begin{equation}\label{eq:chg-produit}
\chs(\sigma_1\odot\sigma_2,\lambda)=\chs(\sigma_1,\lambda)\wedge
\chs(\sigma_2).
\end{equation}
In particular, if $\sigma_1=[0]$, we have
$$
\chg(\sigma,\lambda)=\Par(\lambda)\wedge \chg(\sigma)
\quad {\rm in}\quad \Hcal^{-\infty}_{C_{\lambda,\sigma}}(\kgot,N).
$$
\end{theo}

\subsection{Product of groups}\label{sec:product-transv-elliptic}

Let $K_1,K_2$ be two compact Lie groups, with Lie algebras
$\kgot_1,\kgot_2$, and $N$ a $K_1\times K_2$ manifold. We wish to
multiply two elements $\alpha_1(X,Y)$ and $\alpha_2(X,Y)$ of
$\Acal^{-\infty}(\kgot_1\times\kgot_2,N)$. The product will be
well defined if $\alpha_1(X,Y)$ depends smoothly on $Y$, while
$\alpha_2(X,Y)$ depends smoothly on $X$. We introduce thus
$\Acal^{-\infty,\infty}(\kgot_1\times \kgot_2,N)$ as the space of
generalized equivariant forms $\alpha(X,Y)$ depending smoothly on
$Y$: for any compactly supported function $Q\in\f(\kgot_1)$, the
integral $\alpha_Q(Y):=\int_{\kgot_1}\alpha(X,Y)Q(X)dX$ converges
and depends smoothly of $Y\in\kgot_2$. We denote by
$\Hcal^{-\infty,\infty}(\kgot_1\times \kgot_2,N)$ the
corresponding cohomology space of equivariant cohomology classes
$\alpha(X,Y)$ depending smoothly on $Y\in \kgot_2$. Similarly we
define the space $\Hcal_c^{-\infty,\infty}(\kgot_1\times
\kgot_2,N)$ of equivariant cohomology  classes with compact
support depending smoothly on $Y\in \kgot_2$. If $F$ is a closed
$K_1\times K_2$ invariant subset of $N$, then we define similarly
$\Hcal^{-\infty,\infty}(\kgot_1\times \kgot_2,N, N\setminus F)$ as
well as $\Hcal^{-\infty,\infty}_F(\kgot_1\times \kgot_2,N)$. If
$F_1,F_2$ are two  closed $K_1\times K_2$-invariant subsets of
$N$, the product
\begin{equation}\label{eq:diamond-XY}
\xymatrix{
  \Hcal^{-\infty,\infty}(\kgot_1\times \kgot_2,N,N\setminus F_1)\!\!\!\! & \!\! \times \!\!& \!\!\!\!
  \Hcal^{\infty,-\infty}(\kgot_1\times \kgot_2,N,N\setminus F_2)   \\
\ar@{->}^{\diamond}[r] &  &\Hcal^{-\infty}(\kgot_1\times \kgot_2,N,N\setminus (F_1\cap F_2))
}
\end{equation}
is well defined by the same formula as in Definition
\ref{prodgene}.

Similarly the wedge product
\begin{equation}\label{eq:produit-wedge-gene-two}
\Hcal^{-\infty,\infty}_{F_1}(\kgot_1\times \kgot_2,N)
\times \Hcal^{\infty,-\infty}_{F_2}(\kgot_1\times \kgot_2,N)
\stackrel{\wedge}{\longrightarrow} \Hcal^{-\infty}_{F_1\cap
F_2}(\kgot_1\times \kgot_2,N)
\end{equation}
is well defined and the map $\p_F$  (Section
\ref{section:morphism-pF}) is compatible with the products.

\medskip

\subsubsection{The case of $1$-forms}\label{subsec:1-form}

Consider a $K_1\times K_2$-manifold $N$ and a $K_1\times
K_2$-invariant one form on $N$ denoted $\lambda$. We write the map
$f_\lambda$  from $N$ into
 $\kgot_1^*\times\kgot_2^*$ as $f_\lambda:=(f_\lambda^1,f_\lambda^2)$.
We have
$$
C_\lambda= C_\lambda^{1}\cap C_\lambda^{2}
$$
where $C_\lambda^{i}:=\{f_\lambda^i=0\}$.

Consider on $N\setminus C_\lambda$, the equivariant form with generalized coefficients
$$
\beta(\lambda)(X,Y)=-i\lambda
\int_0^{\infty}\e^{itD\lambda(X,Y)}dt\quad (X,Y)\in\kgot_1\times\kgot_2.
$$

\begin{lem}\label{lem:beta-smooth}
The restriction of  $\beta(\lambda)$ to the open subset
$N\setminus C_\lambda^1\subset N\setminus C_\lambda$ defines a generalized
function of $(X,Y)\in\kgot_1\times\kgot_2$ which depends smoothly of $Y$. In
other words, $\beta(\lambda)\vert_{N\setminus C_\lambda^1}$ belongs to
$\Acal^{-\infty,\infty}(\kgot_1\times \kgot_2,N\setminus C_\lambda^1)$.

\end{lem}

\begin{proof} Let us check that, for <any compactly supported
function $Q\in\f(\kgot_1)$, the integral
$\mathrm{c}_Q(Y):=\int_{\kgot_1}\beta(\lambda)(X,Y)Q(X)dX$ converges in
\break $\Acal(N\setminus C_\lambda^1)$ and depends smoothly of $Y\in\kgot_2$.

Consider the form on $N$ defined by
$\mathrm{I}_Q(t,Y):=\int_{\kgot_1}\e^{it D\lambda(X,Y)}Q(X)dX$. At
a point $n\in N$, $\mathrm{I}_Q(t,Y)(n)$ is equal to
$$
\e^{-it\langle f_{\lambda}^{2}(n),Y\rangle}\e^{itd\lambda(n)}
\int_{\kgot_1}\e^{-it\langle f_{\lambda}^{1}(n),X\rangle}Q(X)dX=\\
\e^{-it\langle f_{\lambda}^{2}(n),Y\rangle}
\e^{itd\lambda(n)}\widehat{Q}(t\, f_{\lambda}^{1}(n)).
$$

The term $\e^{itd\lambda(n)}$ is a polynomial in $t$ of degree
bounded by $\dim N$. Since the Fourier transform $\widehat{Q}$ is
rapidly decreasing, we have, for any integer $r$,  the estimate
\begin{equation}\label{eq:estimate-beta-smooth-1}
\Big|\!\Big|\mathrm{I}_Q(t,Y)(n) \Big|\!\Big|\leq \cst \
\frac{(1+t)^{\dim N}}{(1+t^2\|f_{\lambda}^{1}(n)\|^2)^r},
\end{equation}
for all $t\geq 0$ and $(n,Y)$ in a compact subset of $N\times\kgot_2$.

Let $\Kcal$ be a compact subset  of $N\setminus C_\lambda^1$ : one can choose $c>0$ such that
$\|f_{\lambda}^{1}\|>c$ on $\Kcal$. For any integer $q$,  we have then the estimate
\begin{equation}\label{eq:estimate-beta-smooth-2}
\Big|\!\Big|\mathrm{I}_Q(t,Y)(n) \Big|\!\Big|\leq  \ \frac{\cst}{(1+(ct)^2)^q},
\end{equation}
for all $t\geq 0$, $n\in \Kcal$, and $Y$ in a compact subset of
$\kgot_2$. Thus the integral $\int_0^\infty\mathrm{I}_Q(t,Y)dt$ is
absolutely convergent. Since the estimates
(\ref{eq:estimate-beta-smooth-1}) and
(\ref{eq:estimate-beta-smooth-2}) hold for any derivative
$D(\partial_{n,Y})\mathrm{I}_Q(t,Y)$ in the variable $(n,Y)$,  the
integral $\mathrm{c}_Q(Y)=-i\lambda\int_0^\infty\mathrm{I}_Q(t,Y)dt$
defines a smooth map from $\kgot_2$ into $\Acal(N\setminus C_\lambda^1)$.
\end{proof}

\medskip

\begin{defi}\label{def:prf}
We define
$\prf(\lambda)\in\Hcal^{-\infty,\infty}(\kgot_1\times\kgot_2,N,N\setminus
C_\lambda^1)$  to be  the relative class
$$
\prf(\lambda)=\left[1,\beta(\lambda)|_{N\setminus C_\lambda^{1}}\right].
$$
\end{defi}

Assume now that we have two $K_1\times K_2$-invariant  one forms
$\lambda$ and $\mu$ on $N$. We write
$f_\lambda:=(f_\lambda^1,f_\lambda^2)$ and $f_\mu:=
(f_\mu^1,f_\mu^2).$ Let $C_\lambda^1:=\{f_\lambda^1= 0\}$, and
$C_\mu^2:=\{f_\mu^2= 0\}$. Then the form
$\prf(\lambda)(X,Y)\in\Hcal^{-\infty,\infty}(\kgot_1\times\kgot_2,N,N\setminus
C_\lambda^1)$ depends smoothly on $Y\in  \kgot_2$, while the form
$\prs(\mu)(X,Y)\in\Hcal^{\infty,-\infty}(\kgot_1\times\kgot_2,N,N\setminus
C_\mu^2)$ depends smoothly of $X\in \kgot_1$, and  one can form
the product of the relative classes:
$$
\prf(\lambda)\diamond\prs(\mu)\in\Hcal^{-\infty}(\kgot_1\times\kgot_2,N,N\setminus
(C_\lambda^1\cap C_\mu^2)).
$$

We consider the invariant one form $\lambda+\mu$ and the
associated map \break
$(f_\lambda^1+f_{\mu}^1,f_\lambda^2+f_\mu^2):N\longrightarrow
\kgot_1^*\times\kgot_2^*$ which vanishes on
$$
C_{\lambda+\mu}:=\left\{f_\lambda^1+f_\mu^1=0\right\}\bigcap
\left\{f_\lambda^2+f_\mu^2=0\right\}.
$$

Let $\pr(\lambda+\mu)\in\Hcal^{-\infty}(\kgot_1\times\kgot_2,N,N\setminus  C_{\lambda+\mu})$
be the relative class associated to $\lambda+\mu$.

\medskip

We take some invariant norms on $\kgot^*_1,\kgot^*_2$, and we
consider the following functions on $N$ : $\|f_{\lambda}^{1} \|$, $\|f_{\lambda}^{2}\|$, $\|f_{\mu}^1\|$,
$\|f_{\mu}^2\|$. In order to compare $\prf(\lambda)\diamond\prs(\mu)$ with
$\pr(\lambda+\mu)$, we introduce the following

\begin{defi}\label{conetwo}
We define the closed invariant subset
\begin{equation}\label{eq:C-lambda-mu}
\Ccal(\lambda,\mu):= \left\{\|f_{\lambda}^{1} \|\leq
\|f_\mu^1\|\right\} \bigcap \left\{\|f_\mu^2\|\leq \|f_\lambda^2\|
\right\}.
\end{equation}

\end{defi}

 Clearly the set $\Ccal(\lambda,\mu)$ contains $C_\lambda^1\cap
C_\lambda^2$ as well as the set $\Ccal_{\lambda+\mu}$. Thus the
following restriction operations are well defined:

\begin{itemize}
\item $\res:\Hcal^{-\infty}(\kgot_1\times\kgot_2,N,N\setminus
(C_\lambda^1\cap C_\mu^2))\to
\Hcal^{-\infty}(\kgot_1\times\kgot_2,N,N\setminus
\Ccal(\lambda,\mu))$, \item
$\res':\Hcal^{-\infty}(\kgot_1\times\kgot_2,N,N\setminus
C_{\lambda+\mu})\to
\Hcal^{-\infty}(\kgot_1\times\kgot_2,N,N\setminus
\Ccal(\lambda,\mu)).$
\end{itemize}

The aim of this section is to prove the following theorem.
\begin{theo} \label{theo:ch-rel-produit-groupe}
We have the following equality
\begin{equation}\label{eq:ch-rel-produit-groupe}
\res\Big(\prf(\lambda)\diamond\prs(\mu)\Big)=
\res'\Big(\pr(\lambda+\mu)\Big)
\end{equation}
 in $\Hcal^{-\infty}(\kgot_1\times\kgot_2,N,N\setminus \Ccal(\lambda,\mu))$.
\end{theo}

\begin{proof}
Consider the closed invariant sets $F_1=\left\{\|f_{\lambda}^{1} \|\leq
\|f_\mu^1\|\right\}$, $F_2:=\left\{\|f_\mu^2\|\leq \|f_\lambda^2\|
\right\}$. Then $C_\lambda^1\subset F_1$ and $C_\mu^2\subset F_2$.
The form $\prf(\lambda)\in
\Hcal^{-\infty,\infty}(\kgot_1\times\kgot_2,N,N\setminus
C_\lambda^1)$ restricts to $\prFf(\lambda)\in
\Hcal^{-\infty,\infty}(\kgot_1\times\kgot_2,N,N\setminus F_1)$
while $\prs(\mu)$ restricts to $\prFs(\mu)\in
\Hcal^{\infty,-\infty}(\kgot_1\times\kgot_2,N,N\setminus F_2)$.
Using the diagram (\ref{eq:fonctoriel-prod-relatif}), we see
that $\res\left(\prf(\lambda)\diamond\prs(\mu)\right)=\prFf(\lambda)
\diamond \prFs(\mu)$.

We thus need to compare the forms $\prFf(\lambda) \diamond
\prFs(\mu)$ and $\pr(\lambda+\mu)$. We work with the invariant
open subsets $U_i:= N\setminus F_i$ and $U=U_1\cup U_2=N\setminus\Ccal(\lambda,\mu)$.
 The relative classes $\prFf(\lambda)$, $\prFs(\mu)$ and $\res'\left(\pr(\lambda+\mu)\right)$
are represented by the couples $(1,\beta_1(\lambda))$,
$(1,\beta_2(\mu))$ and $(1,\beta(\lambda+\mu)\vert_U)$ :

$\bullet$ the form $\beta(\lambda)(X,Y):=-i\lambda \int_0^\infty
\e^{it D\lambda(X,Y)}dt$ defines on $N\setminus C_{\lambda}^1$ an
equivariant form with generalized coefficients depending smoothly
on $Y$. We denote by $\beta_1(\lambda)(X,Y)$ the restriction of
this form on $U_1$.

$\bullet$ the form $\beta(\mu)(X,Y):=-i\mu \int_0^\infty \e^{it
D\mu(X,Y)}dt$ defines on $N\setminus C_{\mu}^2$ an equivariant form
with generalized coefficients depending smoothly on $X$. We denote
by $\beta_2(\mu)(X,Y)$ the restriction of this form on $U_2$.

$\bullet$ the form $\beta(\lambda+\mu)(X,Y):=-i(\lambda+\mu)
\int_0^\infty \e^{it D(\lambda+\mu)(X,Y)}dt$ defines on $N\setminus
C_{\lambda+\mu}$ an equivariant form with generalized
coefficients. We denote
by $\beta(\lambda+\mu)\vert_U(X,Y)$ the restriction of this form on $U$.

\bigskip

Let $\Phi_1+\Phi_2=1$ be a partition of unity on
$U=U_1\cup U_2$ : the function $\Phi_i$ are supposed
$K_1\times K_2$-invariant. We want to prove that
$$
(1,\beta_1(\lambda))\diamond_\Phi(1,\beta_2(\mu))-
(1,\beta(\lambda+\mu)\vert_U)
$$
is $D_{\rm rel}$-exact.

Let $\eta_1(s)=-i\lambda \e^{is D\lambda(X,Y)}$ and $\eta_2(s)=
-i\mu \e^{is D\mu(X,Y)}$. We consider the family of smooth
equivariant forms on $N$
\begin{eqnarray*}
\gamma_{(s,t)}(X,Y)&:=&\eta_1(s)\wedge \eta_2(t)\\
&=& -\lambda\mu \e^{is d\lambda + it d\mu} \e^{-i\langle
f_{(s,t)},(X,Y)\rangle}
\end{eqnarray*}
where $f_{(s,t)}: N\to \kgot_1^*\times \kgot_2^*$ is equal to
$(sf_\lambda^1+tf_\mu^1,sf_\lambda^2+ t f_\mu^2)$.

\begin{lem}\label{lem:I-produit-group}
$\bullet$  The integral ${\rm I}_1 :=\int\!\!\!\int_{0\leq t\leq  s} \gamma_{(s,t)}ds \,dt$
defines a $K_1\times K_2$-equivariant form with generalized coefficients on
$U_1$.

$\bullet$  The integral ${\rm I}_2 :=\int\!\!\!\int_{0\leq s\leq  t} \gamma_{(s,t)}ds \,dt$
defines a $K_1\times K_2$-equivariant form with generalized coefficients on
$U_2$.

$\bullet$ We have $D({\rm I}_1)=
\beta(\lambda+\mu)\vert_{U_1}-\beta_1(\lambda)$ on $U_1$ and $D({\rm I}_2)=$ \break
$-\beta(\lambda+\mu)\vert_{U_2} + \beta_2(\mu)$ on $U_2$.
\end{lem}

\begin{proof}
For any compactly supported function $Q(X,Y)\in\f(\kgot_1\times\kgot_2)$
we consider the integral ${\rm J}_Q(s,t):=
\int_{\kgot_1\times\kgot_2}\gamma_{(s,t)}(X,Y) Q(X,Y)dXdY$.
At a point $n\in N$, $\mathrm{J}_Q(s,t)(n)$ is equal to
\begin{eqnarray*}
-\lambda\mu \e^{is d\lambda(n) + it d\mu(n)}
\int_{\kgot_1\times\kgot_2}\e^{-i\langle f_{(s,t)}(n),(X,Y)\rangle}Q(X,Y)dXdY=\\
-\lambda\mu  \e^{is d\lambda(n) + it d\mu(n)}
\widehat{Q}(f_{(s,t)}(n)).
\end{eqnarray*}
Since $\widehat{Q}$ is rapidly decreasing, and $\e^{is d\lambda(n) + it d\mu(n)}$ is
a polynomial in the variable
$(s,t)$, we have, for any integer $r$,  the estimate
$$
\Big|\!\Big|\mathrm{J}_Q(s,t)(n) \Big|\!\Big|\leq
\frac{P(s,t)}{(1+\|f_{(s,t)}(n)\|^2)^r},
$$
for all $s,t\geq 0$ and $n$ in a compact subset of $N$.  Here
$P(s,t)$ is a polynomial function and $\|f_{(s,t)}(n)\|^2= \|s
f_{\lambda}^{1}(n)+t
f_\mu^1(n)\|^2+\|sf_{\lambda}^{2}(n)+tf_{\mu}^{2}(n)\|^2$.

\medskip

Let us prove the first point. We work on a compact subset $\Kcal$
of $U_1:=\{\|f_{\mu}^{1} \|<\|f_{\lambda}^1 \|\}$. Let $0\leq
r< 1$ and $\epsilon>0$  such that on $\Kcal$ we have :
$\|f_{\mu}^1\|\leq r \|f_{\lambda}^1\|$ and $\|f_{\lambda}^1
\|\geq \epsilon$. We use then that
$$
\|f_{(s,t)}(n)\|^2\geq(1-r)^2 s^2\|f_{\lambda}^1(n)\|^2\geq c s^2,
$$
for $n\in\Kcal$ and $0\leq t\leq s$ (we take
$c=(1-r)^2\epsilon^2$). Finally, for $n\in\Kcal$ and $0\leq t\leq
s$, we get  the estimate of the form
$$
\Big|\!\Big|\mathrm{J}_Q(s,t)(n) \Big|\!\Big|\leq \frac{\cst}{(1+c
s^2)^q}
$$
where $q$ can be taken as large as we want. For any differential operator
$D(\partial)$ acting on $\Acal(N)$, we can prove by the same arguments that
$$
\Big|\!\Big|D(\partial)\cdot\mathrm{J}_Q(s,t)(n) \Big|\!\Big|\leq \frac{\cst}{(1+c s^2)^q},
$$
for $n\in\Kcal$ and $0\leq t\leq s$

This proves that ${\rm I}_1$ defines a $K_1\times K_2$-equivariant form with generalized
coefficients on $U_1$  through the relation:
$\int_{\kgot_1\times\kgot_2}{\rm I}_1(X,Y)Q(X,Y)dXdY=$ \break $\int\!\!\!\int_{0\leq
t\leq  s}\mathrm{J}_Q(s,t)dsdt$.

The proof of the second point is the same, exchanging $\lambda$
and $\mu$.

The last point follows from the computation done in the proof of
Lemma \ref{lem:I-Phi-defined}. We repeat the argument. We have on $U_1$
\begin{eqnarray*}
D({\rm I}_1) &=&\int\!\!\!\int_{0\leq t\leq  s} D(\eta_1(s)\eta_2(t))ds \,dt\\
&=&\int\!\!\!\int_{0\leq t\leq  s} D(\eta_1(s))\eta_2(t) ds
\,dt-\int\!\!\!\int_{0\leq t\leq  s} \eta_1(s)D(\eta_2(t))ds \,dt\\
&=&-\int\!\!\!\int_{0\leq t\leq  s} \frac{d}{ds}(\e^{is D\lambda})\eta_2(t)
ds\,dt+\int\!\!\!\int_{0\leq t\leq  s} \eta_1(s)\frac{d}{dt}(\e^{it D\mu})ds \,dt\\
\end{eqnarray*}
Now $-\int\!\!\!\int_{0\leq t\leq  s} \frac{d}{ds}(\e^{isD\lambda})\eta_2(t)ds\,dt
=\int_0^\infty\e^{it D\lambda}\eta_2(t) dt
=-i\mu\int_{0}^{\infty}\e^{it D\lambda+it D\mu} dt$
while
\begin{eqnarray*}
\int\!\!\!\int_{0\leq t\leq  s}\eta_1(s)\frac{d}{dt}(\e^{it D\mu}))ds \,dt&=&
\int_0^\infty\eta_1(s)\e^{is D\mu} ds-
\int_0^\infty\eta_1(s) ds\\
&=&-i\lambda\int_{0}^{\infty}\e^{isD\lambda+ is D\mu}
ds-\beta(\lambda).
\end{eqnarray*}

Thus we obtain on $U_1$ the wanted equality $D({\rm I}_1)=\beta(\lambda+\mu)\vert_{U_1}-\beta(\lambda)$.
The proof of the equality $-D{\rm I}_2=\beta(\lambda+\mu)\vert_{U_2}-\beta(\mu)$ is entirely similar.
\end{proof}

\bigskip

Thanks to  Lemma \ref{lem:I-produit-group}, we define the following equivariant form on $U$ :
${\rm I}_{\Phi}= \Phi_1 {\rm I}_{1} -\Phi_2 {\rm I}_{2}$. The relation
$$
\Big(1,\beta_1(\lambda)\Big)\diamond_\Phi
\Big(1,\beta_2(\mu)\Big)-
\Big(1,\beta(\lambda+\mu)\vert_U\Big)=D_{\rm rel}\Big(0,{\rm
I}_{\Phi}\Big).
$$
admits the same proof than the one of Equality (\ref{prodrel}).

\end{proof}

\bigskip

We denote $\Par^1(\lambda)$ the  image of $\prf(\lambda)$ in
$\Hcal_{C_{\lambda}^1}^{-\infty,\infty}(\kgot_1\times\kgot_2,N)$
and by $\Par^2(\mu)$ the image of $\prs(\mu)$ in
$\Hcal_{C_{\mu}^2}^{\infty,-\infty}(\kgot_1\times\kgot_2,N)$. We
denote by $\Par(\lambda+\mu)$ the  image of
$\pr(\lambda+\mu)$ in
$\Hcal^{-\infty}_{C_{\lambda+\mu}}(\kgot_1\times\kgot_2,N)$.
We use here the restriction maps
\begin{eqnarray*}
\res :\Hcal^{-\infty}_{C_{\lambda}^1\cap C_{\mu}^2}(\kgot_1\times\kgot_2,N)
&\longrightarrow& \Hcal_{\Ccal(\lambda,\mu)}^{-\infty}(\kgot_1\times\kgot_2,N)\\
\res' :\Hcal^{-\infty}_{C_{\lambda+\mu}}(\kgot_1\times\kgot_2,N)
&\longrightarrow& \Hcal_{\Ccal(\lambda,\mu)}^{-\infty}(\kgot_1\times\kgot_2,N)
\end{eqnarray*}
The fact that the map  $p_F$ is compatible with products and restrictions gives us
the following corollary.

\begin{coro} \label{coro:chg-produit-groupe}
 We have the following equality
\begin{equation}\label{eq:chg-produit-groupe}
\res\Big(\Par^1(\lambda)\wedge \Par^2(\mu)\Big)= \res'\Big(\Par(\lambda+\mu)\Big)
\end{equation}
in $\Hcal^{-\infty}_{\Ccal(\lambda,\mu)}(\kgot_1\times\kgot_2,N)$.
\end{coro}

\begin{rem}
 In later applications, it will happen that the set $\Ccal(\lambda,\mu)$
is exactly equal to $C_{\lambda+\mu}$. In this case, we obtain the
equality $$\res\left(\Par^1(\lambda)\wedge \Par^2(\mu)\right)=\Par(\lambda+\mu)$$
in $\Hcal^{-\infty}_{C_{\lambda+\mu}}(\kgot_1\times\kgot_2,N)$.
\end{rem}

%

\medskip

\subsubsection{The case of morphisms}

We now consider the case of a $K_1\times K_2$-equivariant morphism
$\sigma:\Ecal^+\to \Ecal^-$ over $N$. Let $\lambda$ a $K_1\times K_2$ equivariant one form
on $N$.  We choose an invariant super-connections $\A$,
without $0$ exterior degree terms on the $\Zbb_2$-graded vector
bundle $\Ecal$.

The relative Chern class
$\chr(\sigma,\lambda)\in \Hcal^{-\infty}(\kgot_1\times\kgot_2,N,N\setminus C_{\lambda,\sigma})$
is represented by a couple $(\ch(\A),\beta(\sigma,\lambda,\A))$ where
$\beta(\sigma,\lambda,\A)$ is a $K_1\times K_2$-equivariant  differential form on
$N\setminus C_{\lambda,\sigma}$ with generalized coefficients.

As in Section \ref{subsec:1-form} we write $C_\lambda$ as the
intersection $C^1_\lambda\cap C^2_\lambda $ where
$C^i_\lambda=\{f_\lambda^i=0\}$. We define the closed invariant
subsets
$$
C^i_{\lambda,\sigma}= \supp(\sigma)\cap C_\lambda^i.
$$
We will restrict equivariant forms on $N\setminus C_{\lambda,\sigma}$ to the open subsets
$N\setminus C^i_{\lambda,\sigma}$.

\begin{lem}
The equivariant form $\beta(\sigma,\lambda,\A)(X,Y)$, when restricted to \break
$N\setminus C^1_{\lambda,\sigma}$,  depends smoothly on $Y\in \kgot_2$.
\end{lem}
\begin{proof}
For any compactly supported function $Q\in\f(\kgot_1)$, consider
the element of  $\Acal(N, \End(\Ecal))$ defined by
$$
\mathcal{I}_Q(t,Y):=\int_{\kgot_1}\e^{\F(\sigma,\lambda,\A,t)(X,Y)}
Q(X)dX.
$$
At a point $n\in N$, $\mathcal{I}_Q(t,Y)(n)$ is equal to
$$
\e^{-it\langle f_{\lambda}^{2}(n),Y\rangle}
\int_{\kgot_1}\e^{-it\langle
f_{\lambda}^{1}(n),X\rangle}\e^{-t^2R(n)+S(n,X,Y)+T(t,n)} Q(X)dX
$$
with $R(n)= v_{\sigma}^2(n)$, $S(n,X,Y)=\mu^{\A}(n)(X,Y)$,
$T(t,n)=it[\A,v_\sigma](n) +itd\lambda(n)+\A^2(n)$.
If we use
Proposition \ref{prop-estimation-generale-J-Q} of the Appendix, we
have, for any integer $r$,  the estimate
$$
\Big|\!\Big|\mathcal{I}_Q(t,Y) \Big|\!\Big|(n)\leq \cst \
\frac{(1+t)^{\dim N}}{(1+t^2\|f_{\lambda}^{1}(n)\|^2)^r} \,
\e^{-h_\sigma(n)t^2},
$$
for all $t\geq 0$ and $(n,Y)$ in a compact subset of
$N\times\kgot_2$. The only change with respect to Proposition
\ref{estimatesgen} is that we work with the map $f_{\lambda}^{1}$
instead of $f_{\lambda}=(f_{\lambda}^{1},f_{\lambda}^{2})$. As in
Corollary \ref{coro:behaviour-gene}, we see that the
integral $\int_0^\infty\mathcal{I}_Q(t,Y)dt$ defines a smooth map
from $\kgot_2$ into $\Acal(N\setminus C^1_{\lambda,\sigma},
\End(\Ecal))$.

When we restrict the equivariant form $\beta(\sigma,\lambda,\A)(X,Y)$ to the open subset $N\setminus
C^1_{\lambda,\sigma}$, we get the relation
$$
\int_{\kgot_1}\beta(\sigma,\lambda,\A)|_{N\setminus C^1_{\lambda,\sigma}}(X,Y)Q(X)dX=
\str\left(i(v_\sigma+\lambda)\int_0^\infty\mathcal{I}_Q(t,Y)dt\right).
$$
It proves then that $Y\mapsto \beta(\sigma,\lambda,\A)|_{N\setminus C^1_{\lambda,\sigma}}(X,Y)$ is smooth.
\end{proof}

\bigskip

We can make the following definition.

\begin{defi}
We define $\chrf(\sigma,\lambda)\in
\Hcal^{-\infty,\infty}(\kgot_1\times \kgot_2,N,
 N\setminus C^1_{\lambda,\sigma})$ to be the relative
 class
$$[\ch(\A),\beta(\sigma,\lambda,\A)|_{N\setminus C^1_{\lambda,\sigma}}].$$
\end{defi}

\bigskip

Consider now  on $N$ : two $K_1\times K_2$-equivariant morphisms $\sigma,\tau$ and two invariant one forms
$\lambda,\mu$. We then consider the relative classes:
\begin{itemize}
\item $\chrf(\sigma,\lambda)\in\Hcal^{-\infty,\infty}
(\kgot_1\times\kgot_2,N,N\setminus C^1_{\lambda,\sigma})$,

\item $\chrs(\tau,\mu)\in\Hcal^{\infty,-\infty}
(\kgot_1\times\kgot_2,N,N\setminus C^2_{\mu,\tau})$,

\item $\chr(\sigma\odot\tau,\lambda+\mu)\in
\Hcal^{-\infty}(\kgot_1\times\kgot_2,N,N\setminus C_{\lambda+\mu,\sigma\odot\tau}).$
\end{itemize}

\medskip

 The element $\chrf(\sigma,\lambda)$ (resp. $\chrs(\tau,\mu)$) is
represented by an equivariant form with generalized coefficients
 which is  smooth relatively to $Y\in\kgot_2$ (resp. $X\in\kgot_1$).
 Hence we can
form the product
$\chrf(\sigma,\lambda)\diamond\chrs(\tau,\mu)$ which belongs to $\Hcal^{-\infty}_{}(\kgot_1\times\kgot_2,N,
N\setminus(\supp(\sigma\odot\tau)\cap C_\lambda^1\cap C_\mu^2))$.

We denote by $\chgf(\sigma,\lambda)$ the image of
$\chrf(\sigma,\lambda)$ in
$\Hcal^{-\infty,\infty}_{C^1_{\lambda,\sigma}}(\kgot_1\times\kgot_2,N)$. It is represented by
equivariant forms with generalized coefficients which are smooth
relatively to $Y\in\kgot_2$.

Similarly, we denote by $\chgs(\tau,\mu)$ the image of $\chrs(\tau,\mu)$ in \break
$\Hcal^{\infty,-\infty}_{ C^2_{\mu,\tau}}(\kgot_1\times\kgot_2,N)$. It is represented by
equivariant forms with generalized coefficients which are  smooth
relatively to $X\in\kgot_1$.

As in Theorem \ref{theo:ch-rel-produit-groupe}, we look at the
image of $\chrf(\sigma,\lambda)\diamond\chrs(\tau,\mu)$ and
$\chr(\sigma\odot\tau,\lambda+\mu)$ in
$\Hcal^{-\infty}_{}(\kgot_1\times\kgot_2,N,N\setminus(\supp(\sigma\odot\tau)\cap\Ccal(\lambda,\mu)))$.
We leave  the  natural restriction maps implicit.
\begin{theo}\label{prop:chg-produit-groupe-sigma}
\begin{itemize}
\item The following equality
$$
\chrf(\sigma,\lambda)\diamond
\chrs(\tau,\mu)=\chr(\sigma\odot\tau,\lambda+\mu)
$$
holds in $\Hcal^{-\infty}(\kgot_1\times\kgot_2,N,N\setminus
(\supp(\sigma\odot\tau)\cap\Ccal(\lambda,\mu)))$.

\item The following equality
\begin{equation}\label{eq:chern-good-lambda-mu}
\chgf(\sigma,\lambda)\wedge
\chgs(\tau,\mu)=\chg(\sigma\odot\tau,\lambda+\mu)
\end{equation}
holds in
$\Hcal^{-\infty}_{\supp(\sigma\odot\tau)\cap\Ccal(\lambda,\mu)}(\kgot_1\times\kgot_2,N)$.

\end{itemize}

\end{theo}

\begin{proof} As in Corollary \ref{coro:chr-lambda}, it is easy to see that
$\chrf(\sigma,\lambda)= \prf(\lambda)\diamond\chr(\sigma)$ and
 $\chrs(\tau,\mu)=\prf(\lambda)\diamond\chr(\sigma)$. Thus, using the associativity
of the product, we have
  \begin{eqnarray*}
\prf(\lambda)\diamond\chr(\sigma)\diamond
\prs(\mu)\diamond\chr(\tau) &=&
\prf(\lambda)\diamond  \prs(\mu)\diamond\chr(\sigma)\diamond\chr(\tau)\\
&=&
\pr(\lambda+\mu)\diamond\chr(\sigma\odot\tau)\\
&=& \chr(\sigma\odot\tau,\lambda+\mu).
\end{eqnarray*}

\end{proof}

\bigskip

Let us recall the meaning of Equation
(\ref{eq:chern-good-lambda-mu}). For any neighborhood $\Vcal$ of
$\supp(\sigma\odot\tau)\cap\Ccal(\lambda,\mu)$, let
$c_\Vcal(\sigma\odot\tau,\lambda+\mu)$ be the component of \break
$\chg(\sigma\odot\tau,\lambda+\mu)$ in
$\Hcal^{-\infty}_{\Vcal}(\kgot_1\times\kgot_2,N)$. Then we have
\begin{equation}\label{eq:chern-V-lambda-mu}
c_{\Vcal_1}(\sigma,\lambda)\wedge c_{\Vcal_2}(\tau,\mu)=
c_\Vcal(\sigma\odot\tau,\lambda+\mu).
\end{equation}
in $\Hcal^{-\infty}_{\Vcal}(\kgot_1\times\kgot_2,N)$. Here
$\Vcal_1$ and $\Vcal_2$ are respectively any neighborhood of
$\supp(\sigma)\cap C_\lambda^1$ and $\supp(\tau)\cap C_\mu^2$ such
that $\Vcal_1\cap \Vcal_2\subset \Vcal$. The class
$c_{\Vcal_1}(\sigma,\lambda)$ (resp. $c_{\Vcal_2}(\tau,\mu)$) is
the component of $\chgf(\sigma,\lambda)$ (resp. $\chgs(\tau,\mu)$)
in $\Hcal^{-\infty,\infty}_{\Vcal_1}(\kgot_1\times\kgot_2,N)$
(resp.
$\Hcal^{\infty,-\infty}_{\Vcal_2}(\kgot_1\times\kgot_2,N)$).

\bigskip

Let $\sigma$ and $\tau$ be two morphisms such that
$\supp(\sigma\odot\tau)\cap \Ccal(\lambda,\mu)$ is {\bf compact},
hence $\supp(\sigma\odot\tau)\cap C_{\lambda+\mu}$ is
 compact. In this case, the Chern equivariant class with
compact support $\ch_c(\sigma\odot\tau,\lambda+\mu)$ is equal to
the product
$$
c_{\Vcal_1}(\sigma,\lambda)\wedge c_{\Vcal_2}(\tau,\mu)
$$
in $\Hcal^{-\infty}_c(\kgot_1\times \kgot_2,N)$. Here $\Vcal_1$
and $\Vcal_2$ are respectively any neighborhood of
$\supp(\sigma)\cap C_\lambda^1$ and $\supp(\tau)\cap C_\mu^2$ such
that $\overline{\Vcal_1}\cap \overline{\Vcal_2}$ is compact.

In particular, we obtain the following theorem.

\begin{theo}\label{twogroups}
 Let $\sigma$ and $\tau$ be two equivariant morphisms such that

$\bullet$ $\supp(\sigma)\cap C_\lambda^{1}$ is  compact,

$\bullet$ $\supp(\tau)\cap C_\mu^{2}$ is  compact,

$\bullet$ $\supp(\sigma\odot\tau)\cap \Ccal(\lambda,\mu)$ is
compact.

Then $$\ch_c^1(\sigma,\lambda)\wedge
\ch_c^2(\tau,\mu)=\ch_c(\sigma\odot\tau,\lambda+\mu).$$
\end{theo}

\subsection{Retarded construction}\label{sec:retarded-gene}

Let $\sigma:\Ecal^+\to \Ecal^-$ be a $K$-equivariant smooth morphism and let
$\lambda$ be a $K$-invariant one-form.

Let $\A$ be a $K$-invariant super-connection without $0$ exterior
degree term and let $\F(\sigma,\lambda,\A,t)=it
D\lambda+\F(\sigma,\A,t)$ be the equivariant curvature of the
super-connection $\A^{\sigma,\lambda}(t)=\A+it
(v_\sigma+\lambda)$. For any $T\in\Rbb$, we consider the Chern
character
\begin{equation}\label{eq:ch-sigma-lambda-T}
\ch(\sigma,\lambda,\A,T):=\str( \e^{\F(\sigma,\lambda,\A,T)}).
\end{equation}
On $N\setminus C_{\lambda,\sigma}$, we have
$\ch(\sigma,\lambda,\A,T)=D(\beta(\sigma,\lambda,\A,T))$ where the generalized
equivariant odd form $\beta(\sigma,\lambda,\A,T)$ is defined on
$N\setminus C_{\lambda,\sigma}$ by the integral
\begin{equation}\label{eq:beta-sigma-lambda-T}
\beta(\sigma,\lambda,\A,T)
=\int_{T}^{\infty}\eta(\sigma,\lambda,\A,t)dt,
\end{equation}
where $\eta(\sigma,\lambda,\A,t):= -i\str\left((v_\sigma+\lambda)
\e^{\F(\sigma,\lambda,\A,t)}\right)$. It is easy to check that the
following equality
\begin{eqnarray}\label{eq-chrel-t-delta-lambda}
\lefteqn{\!\!\!\!\!\!\!\!\!\!\!\!\!\!\!\!\!\!\!\!
\Big(\ch(\A),\beta(\sigma,\lambda,\A)\Big)-
\Big(\ch(\sigma,\lambda,\A,T),\beta(\sigma,\lambda,\A,T)\Big)=} \\
& & D_{\rm rel}
\left(\int_0^T\!\!\!\eta(\sigma,\lambda,\A,t)dt,\,\nonumber 0\right)
\end{eqnarray}
holds in $\Acal^{-\infty}(\kgot,N,N\setminus  C_{\lambda,\sigma})$.
Hence we get the following
\begin{lem}\label{lem:other-representant-lambda}
 For any $T\in\Rbb$, the relative Chern character $\chr(\sigma,\lambda)$ satisfies
$$
\chr(\sigma,\lambda)=\Big[\ch(\sigma,\lambda,\A,T),\beta(\sigma,\lambda,\A,T)\Big]
$$
in $\Hcal^{-\infty}(\kgot,N,N\setminus  C_{\lambda,\sigma})$.
\end{lem}

Using Lemma \ref{lem:other-representant-lambda}, we get
\begin{lem}\label{lem:c-sigma-t-gene}
For any $T\geq 0$, the class $\chg(\sigma,\lambda)$ can be defined with the
forms $c(\sigma,\lambda,\A,\chi,T):=\chi\, \ch(\sigma,\lambda,\A,T)
+ d\chi\,\beta(\sigma,\lambda,\A,T)$.
\end{lem}
\begin{proof} It is due to the following transgression
\begin{equation}\label{eq:trangression-c-t-gene}
    c(\sigma,\lambda,\A,\chi)-c(\sigma,\lambda,\A,\chi,T)=
    D\left(\chi\int_0^T\!\!\!\eta(\sigma,\lambda,\A,t)dt\right),
\end{equation}
which follows from (\ref{eq-chrel-t-delta-lambda}).
\end{proof}\bigskip

In some situations the Chern form $\ch(\sigma,\lambda,\A,1)$ enjoys
good properties relative to the integration. So it is natural to
compare the differential form $c(\sigma,\lambda,\A,\chi)$ and
$\ch(\sigma,\lambda,\A,1)$.

\begin{lem}\label{retardord}
We have
\begin{eqnarray*}
c(\sigma,\lambda,\A,\chi)-\ch(\sigma,\lambda,\A,1)=
D\left(\chi\int_{0}^1\eta(\sigma,\lambda,\A,s)ds\right)+\\
D\Big((\chi-1) \beta(\sigma,\lambda,\A,1)\Big).
\end{eqnarray*}
\end{lem}
\begin{proof} This follows immediately from the transgressions
(\ref{eq:transgression-integral-gene}) and
(\ref{eq:trangression-c-t-gene}).

\end{proof}

\subsection{Example of Hamiltonian manifolds}

There are several natural situations where a $K$-invariant one-form
exists.

Let $(N,\Omega,\Phi)$ be a Hamiltonian $K$-manifold: here $\Omega$
is a symplectic form on $N$. The moment map $ \Phi:N\to \kgot^*$ is
a $K$-equivariant map satisfying the relation
$$d\langle \Phi,X\rangle =\iota(VX)\Omega,$$
 for every $X\in \kgot$ so that
the equivariant symplectic form
 $\Omega(X):=\langle  \Phi,X\rangle +\Omega$\, is a closed
equivariant form.

 With the help of an invariant scalar product on $\kgot^*$, we
have an identification $\kgot^*\simeq\kgot$ : the moment map $\Phi$
will be a map from $N$ on $\kgot$. We consider then the Kirwan
vector field $\K(n)=V_n(\Phi(n))$ : note that $\K$ is the
Hamiltonian vector field of the function $\frac{1}{2}\|\Phi\|^2:
N\to \Rbb$. Here we can define the invariant one-form
\begin{equation}\label{eq:kirwan-vector}
    \lambda_\K:=(\K,-)_N
\end{equation}
where $(-,-)_N$ is any $K$-invariant Riemannian metric on $N$. It
is easy to see that, for $n\in N$,
$$
f_{\lambda_\K}(n)=0\Longleftrightarrow \lambda_\K(n)=0 \Longleftrightarrow
\K(n)=0\Longleftrightarrow d(\|\Phi\|^2)(n)=0.
$$
Hence the set $C_{\lambda_\K}$ coincides with the set $\Cr(\|\Phi\|^2)$
of critical points of the function $\|\Phi\|^2$. In this situation,
the generalized equivariant form $\Par(\lambda_\K)$ have been studied
in \cite{pep1,pep2}.

We note that
\begin{equation}\label{eq:inclusion}
    \{\Phi=0\}\subset \Cr(\|\Phi\|^2).
\end{equation}
There are interesting situations where (\ref{eq:inclusion}) is an
equality.

\medskip

Suppose now that the symplectic form $\Omega$ is exact: there
exists a $K$-invariant one form $\omega$ on $N$ such that
$\Omega=d\omega$. We can choose as associated moment map $\langle
\Phi,X\rangle=-\langle\omega,VX\rangle$  and  the equivariant
symplectic form is exact:  $\Omega(X)=D\omega(X)$. We have then
two different one forms on $N$, the one form $\lambda_\K$
associated to the Kirwan vector field and the one form $\omega$.

\begin{lem}
Assume $\Omega=d\omega$ and $\langle\Phi,X\rangle
=-\langle\omega,VX\rangle$. We have then $f_\omega=-\Phi$ and
$C_{\lambda_\K}=C_\omega=\{\Phi=0\}$.
\end{lem}
\begin{proof}
The first equality is by definition of the moment map:
\begin{equation}\label{eq:moment-exact}
    \langle\Phi(n),X\rangle=-\langle\omega(n),V_nX\rangle, \quad
n\in N.
\end{equation}   If one
takes $X=\Phi(n)$ in (\ref{eq:moment-exact}) , it gives
$\|\Phi(n)\|^2=-\langle\omega(n),\K(n)\rangle$ and then
$$
C_{\lambda_\K}=\Cr(\|\Phi\|^2)=\{\Phi=0\}=\{f_\omega=0\}=C_\omega.
$$
\end{proof}\bigskip

It is natural to compare the elements $\Par(\lambda),
\Par(-\omega)\in \Hcal_{\Phi^{-1}(0)}^{-\infty}(\kgot,N)$. We
consider the following family of one-forms: $\lambda_s = s\lambda_\K
-(1-s)\omega$, $s\in [0,1]$. Since $f_{\lambda_s}=s
f_{\lambda_\K}-(1-s)f_{\omega}$, we have $C_{\lambda_\K}\subset
C_{\lambda_s}$ for any $s$. We have also: $\langle
f_{\lambda_s}(n), \Phi(n)\rangle=s\|\K\|^2 +(1-s)\|\Phi(n)\|^2$
for any $n\in N$, which shows that $C_{\lambda_s}\subset
C_{\lambda_\K}$. We have proved that $C_{\lambda_s}= C_{\lambda_\K}$ for
any $s\in[0,1]$. With the help of Theorems \ref{Prel-beta} and \ref{prop:equi-lambda},
we can conclude with the following

\begin{prop}\label{lem:exact-kirwan}
Let $N$ be a $K$-manifold, equipped with an exact symplectic two
form $\Omega=d\omega$. The moment map $\Phi: N\to \kgot^*$ is defined
by (\ref{eq:moment-exact}). We have $ \Cr(\|\Phi\|^2)=\Phi^{-1}(0)$
and
$$
\pr(\lambda_\K)=\pr(-\omega) \quad \mathrm{in}\quad
\Hcal^{-\infty}(\kgot,N,N\setminus \Phi^{-1}(0)),
$$
$$
\Par(\lambda_\K)=\Par(-\omega) \quad \mathrm{in}\quad
\Hcal_{\Phi^{-1}(0)}^{-\infty}(\kgot,N).
$$
\end{prop}

\subsubsection{The cotangent manifold}\label{example:Par-omega}

Here $N=\T^*M$, where $M$ is a $K$-manifold. Let $p:\T^*M\to M$ be
the projection. We denote by  $\omega$ the Liouville form  on $\T^* M$ :
 $-\omega_{[x,\xi]}(w)=\langle \xi, p_*w\rangle$.
 Then $\Omega:= d\omega$ is the canonical
symplectic structure on $\T^*M$. The corresponding moment map for
the Hamiltonian action of $K$ on $(\T^*M,\Omega)$ is the map
$f_\omega:\T^*M\longrightarrow \kgot^*$ defined by the relation
\begin{equation}\label{eq:f-omega}
f_\omega(x,\xi): X\mapsto \langle \xi,V_xX\rangle.
\end{equation}
Here $f_\omega^{-1}(0)$ is the subset $\T^*_K M\subset \T^* M$
formed by co-vectors orthogonal to the $K$-orbits. In this situation
we define a classes
\begin{equation}\label{eq:Prel-omega}
    \pr(\omega)\in \Hcal^{-\infty}(\kgot,\T^* M,\T^*M\setminus\T^*_K M).
\end{equation}

and
\begin{equation}\label{eq:Par-omega}
    \Par(\omega)\in \Hcal^{-\infty}_{\mathbf{T}^*_K M}(\kgot,\T^* M).
\end{equation}

This form $\Par(\omega)$ will be used extensively in a subsequent article to give
a new cohomological formula for the index of transversally
elliptic operators.

\subsubsection{Symplectic vector space}

Let $N=V$ be a real vector  space  of dimension $2n$, with a
non-degenerate skew-symmetric bilinear form $\Omega$ : we have
$\Omega=d\omega$ where $\omega=\Omega(v,dv)$ on $V$. Let $K$ be a
compact Lie group acting on $V$ by linear symplectic
transformations. Then $V$ is a $K$-Hamiltonian space with moment map
$\langle\Phi_K(v),X\rangle=\Omega(X v,v)$.

Assume for the rest of this section that the moment map  $\Phi_K: V\to \kgot^*$ is  {\em proper}.
Since $\Phi_K$ is a homogeneous map, this assumption of properness is equivalent to one of the following
conditions:
\begin{itemize}
\item $\Phi^{-1}_K(0)=0$.

\item There exists $c>0$ such that  $\|\Phi_K(v)\|\geq c\|v\|$ for all $v\in V$.
\end{itemize}

 In this case, we obtain a class $\pr(\omega)\in \Hcal^{-\infty}(\kgot,V,V\setminus \{0\})$ that we wish
to compare with the relative Thom class.

\medskip

We have shown in \cite{pep-vergne2} that $\Hcal^{\infty}(\kgot,V,V\setminus \{0\})$ is a free module
over $\f(\kgot)^K$ with basis the Thom classe $\tur(V)$. More precisely to any class
$a=[\alpha,\beta]\in\Hcal^{\infty}(\kgot,V,V\setminus \{0\})$ we consider the integral
$$
\int_V \p_c(a)(X)\in\f(\kgot)^K
$$
where $\p_c(a)(X)=\chi\alpha(X)+d\chi\beta(X)$ is a $K$-equivariant class
with compact support on $V$ defined with the help of a function $\chi\in\f(V)^K$ with compact support
and equal to $1$ in a neighborhood of $0$.

We have the
\begin{theo}\label{theo:thom-rel-V-smooth}\cite{pep-vergne2}
For any class $a\in\Hcal^{\infty}(\kgot,V,V\setminus \{0\})$, we have
the relation
\begin{equation}\label{eq:thom-rel-V}
a=(\int_V \p_c(a))\cdot \tur(V)
\end{equation}
in $\Hcal^{\infty}(\kgot,V,V\setminus \{0\})$.
\end{theo}

\medskip

The same result, with same proof,  holds if one work with equivariant forms with {\em generalized coefficients}.
For any $a\in\Hcal^{-\infty}(\kgot,V,V\setminus \{0\})$ the integral $\int_V \p_c(a)$ defines an invariant
generalized function on $\kgot$. Since $\tur(V)$ has smooth coefficients, the product
$(\int_V \p_c(a))\cdot \tur(V)$ makes sense for any $a\in\Hcal^{-\infty}(\kgot,V,V\setminus \{0\})$, and
Equality (\ref{eq:thom-rel-V}) holds in this case.

Let $dv:=\frac{\Omega^n}{n!}$ be the symplectic volume form on $V$.

\begin{prop}\label{prop:Thom-rel-P-omega}
The following relation holds in $\Hcal^{-\infty}(\kgot,V,V\setminus \{0\})$:
\begin{equation*}
\pr(\omega)=\Theta\cdot \tur(V),
\end{equation*}
where $\Theta\in\fgene(\kgot)^K$ is defined by the relation
$$
\Theta(X):=(i)^n \int_V \e^{i\langle\Phi_K(v),X\rangle}dv,\quad X\in\kgot.
$$
\end{prop}
\begin{proof}
Following Theorem \ref{theo:thom-rel-V-smooth}, we have just to compute the integral
$\Theta(X):=\int_V\p_c(\pr(\omega))(X)$. Let $f\in\f(\Rbb)$ be a compactly supported function which
is equal to $1$ in a neighborhood of $0$. We work with the invariant function
$\chi(v):=f(\|v\|^2)$ on $V$,  where $\|-\|$ is any $K$-invariant Euclidean norm on $V$.
The equivariant form with generalized coefficient
$\chi+d\chi\wedge\beta(\omega)$ which represents the cohomology class
$\p_c(\pr(\omega))\in\Hcal^{-\infty}_c(\kgot,V)$
is the limit, as $T$ goes to infinity, of the equivariant forms with compact support
\begin{eqnarray*}
\eta^T&=&\chi+d\chi\wedge(-i\omega)\wedge\int_0^T e^{itD(\omega)}dt\\
&=& \chi +D\left(\chi(-i\omega)\wedge\int_0^T e^{iD(\omega)(X)}\right)+
\chi \int_0^T \frac{d}{dt}e^{itD(\omega)}dt\\
&=& \chi e^{iT\, D(\omega)}+D\left(\chi(-i\omega)\wedge\int_0^T e^{iD(\omega)(X)}\right).
\end{eqnarray*}
Hence $\Theta(X)$ is the limit, as $T$ goes to infinity, of the integrals
$$
\int_V\eta^T(X)= \int_V\chi e^{iT\, D(\omega)(X)}.
$$
Since $D(\omega)(X)\vert_v=\Omega_v +\langle\Phi_K(v),X\rangle$ is homogeneous of
degree $2$ in the variable $v$, we have $T\, D(\omega)(X)=\delta_T^*(D(\omega)(X))$
where $\delta_T(v)=\sqrt{T}v$. Then
$\int_V\eta^T(X)= \int_V f(\frac{\|v\|^2}{T}) e^{i D(\omega)(X)}$ tends to
$$
\int_V  e^{i D(\omega)(X)}=(i)^n \int_V \e^{i\langle\Phi_K(v),X\rangle}dv.
$$
when $T$ goes to infinity.

\end{proof}

\subsection{Comparison with other constructions}

\subsubsection{Integration in mean}

As stressed in the case of ordinary cohomology, one of the main
purposes of constructing Chern character of an elliptic  morphism
$\sigma$ as a cohomology class with compact support is the fact
that such classes are integrable.

In the case of equivariant cohomology, we  introduce  appropriate
cohomology spaces for defining the integral of an equivariant
differential form. Of course, if $\alpha\in
\Hcal^{\infty}(\kgot,N)$, and the manifold $N$ is compact and
oriented, the integral of $\alpha$ is the $K$-invariant
$C^{\infty}$-function of $X\in \kgot$ defined by $\int_N
\alpha(X)$. If $N$ is non compact, we may have to define this
integral in the generalized sense.

Let $\alpha$ be an equivariant form with $C^{\infty}$ coefficients
on a vector bundle $N\to B$ over a compact basis. It may happen
that although $\alpha(X)$ is not integrable on $N$, it is
integrable in mean: by integrating $\alpha(X)$ against a smooth
compactly supported density, we obtain a differential form
$\alpha(Q)=\int_{\kgot}\alpha(X)Q(X)dX$. If this form is rapidly
decreasing over the fibers of $N\to B$, then we can integrate
$\alpha(Q)$ on $N$. In other words, if for any test function $Q$
on $\kgot$, the form $\alpha(Q)$ is rapidly decreasing over the
fibers, we can define the integral $\int_N\alpha$ in the sense of
generalized functions:

$$\int_{\kgot}(\int_N \alpha)(X)Q(X)dX=\int_N \alpha(Q).$$

We define $\Acal^{\infty}_{\mdr}(\kgot,N)$ as the space of
equivariant differential forms with $C^{\infty}$ coefficients such
that, for any test function $Q$ on $\kgot$, the form
$\alpha(Q)=\int_{\kgot}\alpha(X)Q(X)dX$ is rapidly decreasing on
$N$, as well as all its derivatives.

Similarly, we define $\Acal^{-\infty}_{\mdr}(\kgot,N)$ as the
space of equivariant differential forms with $C^{-\infty}$
coefficients such that, for any test function $Q$ on $\kgot$, the
form $\alpha(Q)=\int_{\kgot}\alpha(X)Q(X)dX$ is rapidly decreasing
on $N$, as well as all its derivatives.

 Clearly
$\Acal^{\infty}_{\mdr}(\kgot,N)$  is contained  in
$\Acal^{-\infty}_{\mdr}(\kgot,N)$.

The operator $D$ is well defined on
$\Acal^{-\infty}_{\mdr}(\kgot,N)$ and we denote the cohomology space
by $\Hcal^{-\infty}_{\mdr}(\kgot,N)$. The inclusion
$\Acal^{-\infty}_{c}(\kgot,N)\croc$ \break
$\Acal^{-\infty}_{\mdr}(\kgot,N)$ induces a map
$\Hcal^{-\infty}_{c}(\kgot,N)\to \Hcal^{-\infty}_{\mdr}(\kgot,N).$

If $\alpha$ and $\beta$ are two closed equivariant forms in
$\Acal^{-\infty}_{\mdr}(\kgot,N)$ which defines the same class in
$\Hcal^{-\infty}_{\mdr}(\kgot,N)$, then their integrals on $N$
define the same generalized function on $\kgot$.

If the basis $B$ of the fibration $\pi: N\to B$ is not compact,
the definition of $\Acal^{-\infty}_{\mdr}(\kgot,N)$ makes sense
over any relatively compact open subset of the basis $B$.  If the
bundle $N\to B$ is oriented, then the integral over the fiber
defines a map $\pi_*:\Hcal^{-\infty}_{\mdr}(\kgot,N)\to
\Hcal^{-\infty}(\kgot,B)$.

\subsubsection{Partial Gaussian look}


Assume that $N$ is a $K$-equivariant real vector bundle  over a
$K$-manifold $B$: we denote by  $\pi: N\to B$ the projection. We
denote by $(x,\xi)$ a point of $N$ with $x\in B$ and $\xi\in N_x:=\pi^{-1}(x)$.
Let $\Ecal^\pm\to B$ be two $K$-invariant Hermitian vector bundles.
We consider a $K$-invariant morphism $\sigma: \pi^*\Ecal^+\to
\pi^*\Ecal^-$. Let $\lambda$ be a $K$-invariant one-form on $N$.

 We choose a metric on the fibers of the fibration $N\to B$.
We work under the following assumption on $\sigma$ and $\lambda$.

\begin{assum}\label{assum:dec-rap}
$\bullet$ The morphism $\sigma: \pi^*\Ecal^+\to \pi^*\Ecal^-$ and
all its partial derivatives have at most a polynomial growth along
the fibers of $N\to B$.

$\bullet$ The one-form $\lambda$ and all its partial derivatives
have at most a polynomial growth along the fibers of $N\to B$.

$\bullet$  Moreover we assume that, on any compact subset
$\Kcal_1$ of $B$, there exists $R\geq 0$ and $c>0$ such that
\begin{equation}\label{eq:hyp-sigma-lambda}
    h_\sigma(x,\xi)+\|f_\lambda(x,\xi)\|^2\geq c\|\xi\|^2
\end{equation}
when $\|\xi\|\geq R$ and $x\in\Kcal_1$. Here $h_\sigma(x,\xi)\geq
0$ is the smallest eigenvalue of the positive hermitian
endomorphism $v_\sigma(x,\xi)$.
\end{assum}

\medskip

Let $U(1)$ be the circle group with Lie algebra $\ugot(1)\simeq i\Rbb$. In the following example we denote for any integer $k$ by
$\Cbb_{[k]}$ the vector space $\Cbb$ with the action of $U(1)$ given by: $t\cdot z=t^k z$.

\begin{exam}[Atiyah symbol 1]\label{example:atiyah-symbol-1} Let us consider
the case of the Atiyah symbol. We consider $B=\{pt\}$ and
$N=\T^*\Cbb_{[1]}\simeq \Cbb_{[1]}\times\Cbb_{[1]}$.
We consider the $U(1)$-equivariant symbol
\begin{eqnarray*}
\sigma:N\times \Cbb_{[0]}&\longrightarrow& N\times \Cbb_{[1]}\\
\big(\xi,v\big)&\longmapsto& \big(\xi,\sigma(\xi)v\big)
\end{eqnarray*}
defined by $\sigma(\xi)= \xi_2-i\xi_1$ for $\xi=(\xi_1,\xi_2)$.
We take for one form $\lambda$ on $\T^*\Cbb_{[1]}$ the Liouville one form :
$\lambda={\rm Re}(\xi_2 d\overline{\xi_1})$. Here we have
$f_\lambda(\xi)={\rm Im}(\xi_2 \overline{\xi_1})$
and $h_\sigma(\xi)=|\xi_2-i\xi_1|^2$ for $\xi\in\Cbb^2$.
We compute
\begin{eqnarray*}
h_\sigma(\xi)+\|f_\lambda(\xi)\|^2&=&|\xi_1|^2+|\xi_2|^2-2{\rm Im}(\xi_2 \overline{\xi_1})+
{\rm Im}(\xi_2 \overline{\xi_1})^2\\
&\geq&\frac{1}{2} \|\xi\|^2
\end{eqnarray*}
if $\|\xi\|^2=|\xi_1|^2+|\xi_2|^2\geq 2$. Hence the Atiyah symbol
satisfies Condition (\ref{eq:hyp-sigma-lambda}).
\end{exam}

\medskip

If we consider the set $ C_{\lambda,\sigma}:=\{h_\sigma=0\}\cap\{f_{\lambda}=0\}$,
Condition (\ref{eq:hyp-sigma-lambda}) implies that, for any
compact subset $\Kcal_1$ of $B$, the intersection $
\pi^{-1}(\Kcal_1)\cap  C_{\lambda,\sigma}$ is a
\emph{compact} subset of $N$. Hence we have a natural map
$$
\Hcal^{-\infty}_{ C_{\lambda,\sigma}}(\kgot,N)\longrightarrow \Hcal^{-\infty}_{\mdr}(\kgot,N).
$$

Our purpose in this section is to give a representative of
$\chg(\sigma,\lambda)$ in $\Hcal^{-\infty}_{\mdr}(\kgot,N)$ with
``partial Gaussian look". We will use the results of Section \ref{sec:retarded-gene}.

Let $\nabla=\nabla^+\oplus \nabla^-$ be a connection on $\Ecal\to
B$, let $\A=\pi^*\nabla$ and consider the invariant super-connection $\A^
{\sigma,\lambda}(t)=\A+i t(v_\sigma+ \lambda)$. Let
$\ch(\sigma,\lambda,\A,1)$ and $\beta(\sigma,\lambda,\A,1)$ be the
equivariant forms defined in (\ref{eq:ch-sigma-lambda-T}) and (\ref{eq:beta-sigma-lambda-T}).

\begin{lem}\label{lem:ch=dec-rapid-gene}
    The differential forms
    $\ch(\sigma,\lambda,\A,1)$ and $\beta(\sigma,\lambda,\A,1)$
    belong respectively to $\Acal^{\infty}_{\mdr}(\kgot, N)$
    and $\Acal^{-\infty}_{\mdr}(\kgot,N\setminus C_{\lambda,\sigma})$.
\end{lem}

\medskip

Before going into the proof, let us look at the example

\begin{exam}[Atiyah's symbol 2]\label{example:atiyah-symbol-2} In the case of the Atiyah symbol,
we have $C_{\lambda,\sigma}:=\{|\xi_2-i\xi_1|^2=0\}\cap\{{\rm Im}(\xi_2 \overline{\xi_1})=0\}=\{(0,0)\}$.
We work on $N\simeq\Cbb^2$ with the coordinates $z_1= \xi_2-i\xi_1$ and $z_2= \xi_2+i\xi_1$.

We take on the vector bundle $N\times(\Cbb_{[0]}\oplus\Cbb_{[1]})$ the connection $\nabla=d$.
The equivariant curvature of the invariant super-connections $\A_t^{\sigma}:=d +it v_{\sigma}$ is
$$
\F_t(i\theta)=
\left(\begin{array}{cc} -t^2 |z_1|^2& itd\overline{z_1}\\ itdz_1 &-t^2 |z_1|^2+
it\theta\end{array}\right)
$$
for $i\theta\in \ugot(1)$. The Volterra expension formula gives
$$
\e^{\F_t(i\theta)}=\e^{-t^2 |z_1|^2}
\left(\begin{array}{cc} 1+(g'(i\theta)-g(i\theta)) t^2 dz_1d\overline{z_1}&
it g(i\theta) d\overline{z}\\
it g(i\theta) dz & \e^{i\theta} +g'(i\theta)t^2 dz_1d\overline{z_1}\end{array}\right)
$$
where $g(z)=\frac{\e^z -1}{z}$. In the coordinates $z=(z_1,z_2)$, we have
$$
D\lambda(i\theta)= \frac{1}{4}\left(dz_1d\overline{z_1}+
dz_2d\overline{z_2}-\theta(|z_1|^2-|z_2|^2)\right)
$$

Hence
\begin{eqnarray*}
\ch(\sigma,\lambda,\nabla,1)(\theta)&=&\e^{iD\lambda(i\theta)}\str(\e^{\F_1(i\theta)})\\
&=&\alpha(\theta,z)\e^{- |z_1|^2}
\e^{\frac{i\theta}{4}(|z_2|^2-|z_1|^2)}
\end{eqnarray*}
where $\alpha(\theta,z)$ depends polynomialy of $z$. For any  test function
$Q$ on $\ugot(1)$, we see that the differential form
$\int_{\ugot(1)}\ch(\sigma,\lambda,\nabla,1)(\theta)Q(\theta)d\theta$ on $\Cbb^2$
decomposes in forms of the type $\eta(z)\e^{- |z_1|^2}h(|z_2|^2-|z_1|^2)$ where
$\eta$ depends polynomialy of $z$, and $h$ is a rapidly decreasing function on $\Rbb$ : hence
$\ch(\sigma,\lambda,\nabla,1)\in \Acal^{\infty}_{\mdr}(\ugot(1), \Cbb^2)$.

Now we consider the equivariant forms
\begin{eqnarray*}
 \eta(\sigma,\lambda,\nabla,t)(\theta) &=& -i \e^{it\, D\lambda(i\theta)}\str\left(
\left(\begin{array}{cc}0 & \overline{z_1}\\z_1 & 0\end{array}\right) \e^{\F_t(\phi)}\right) \\
   &=&g(i\theta)(z_1d\overline{z_1}-\overline{z_1}dz_1)\, t \, \e^{-t^2|z_1|^2}\e^{it\,D\lambda(i\theta)}\\
   &=& \gamma(\theta,t,z)\e^{-t^2|z_1|^2}\e^{\frac{it\theta}{4}(|z_2|^2-|z_1|^2)}
\end{eqnarray*}
where $\gamma(\theta,t,z)$ depends polynomialy of $(t,z)$. Now the integral
$$
\beta(\sigma,\lambda,\nabla,1)(\theta)=\int_1^\infty\eta(\sigma,\lambda,\nabla,t)(\theta)dt
$$
defines an $U(1)$-equivariant form on $\Cbb^2\setminus \{(0,0)\}$ with generalized coefficients : it decomposes
in sum of generalized equivariant form of the type
$$
\alpha(\theta,z)\int_1^\infty t^k\e^{-t^2|z_1|^2}\e^{\frac{it\theta}{4}(|z_2|^2-|z_1|^2)}dt
$$
where $\alpha(\theta,z)$ is an equivariant form which depends polynomialy of $z$.

For any  test function $Q$ on $\ugot(1)$, we see that the differential form \break
$\int_{\ugot(1)}\beta(\sigma,\lambda,\nabla,1)(\theta)Q(\theta)d\theta$ on $\Cbb^2\setminus \{(0,0)\}$
decomposes in forms of the type
$$
\gamma(z)\int_1^\infty t^k\e^{-t^2|z_1|^2}h(t(|z_2|^2-|z_1|^2))dt
$$
where $\gamma$ depends polynomialy of $z$, and $h$ is a rapidly decreasing function on $\Rbb$: hence
$\beta(\sigma,\lambda,\nabla,1)\in \Acal^{-\infty}_{\mdr}(\ugot(1), \Cbb^2\setminus \{(0,0)\})$.
\end{exam}

\bigskip


{\bf Proof of Lemma \ref{lem:ch=dec-rapid-gene}.}
We consider the equivariant curvature $\F(t):=\F(\sigma,\lambda,\A,t)$
of the invariant super-connection $\A^{\sigma,\lambda}(t)$. We
 have
$$
\F(t)(X)=  -t^2 v_\sigma^2-it \langle f_\lambda,X\rangle+
\pi^*\F(X)+it[\pi^*\nabla,v_\sigma]+it d\lambda,
$$
where $\F(X)\in\Acal(B,\End(\Ecal))$ is the equivariant curvature
of $\nabla$, and the terms
$[\pi^*\nabla,v_\sigma]\in\Acal^1(N,\End(\pi^*\Ecal)) ,
d\lambda\in\Acal^2(N)$, have at most a polynomial growth along the
fibers of $N\to B$.

Let $Q$ be a test function on $\kgot$ with support in a compact subset
$\Kcal''$ of $\kgot$.
 We need to estimate
the behavior  on the fiber of the differential form $\int_{\kgot}
\e^{\F(t)(X)}Q(X)dX$ over $\pi^{-1}(\Kcal_1)$ where $\Kcal_1$ is a
compact subset of $B$. More explicitly, at a point $n=(x,\xi)\in
N$, we have
$$
 \left( \int_{\kgot}\e^{\F(t)(X)}Q(X)dX\right)(n) = \int_{\kgot} \e^{-it\langle f_\lambda,X\rangle}
 \e^{-t^2R(n)+S(n,X) + T(t,n)}Q(X)dX,
$$
with $R(n)=v_{\sigma}(n)^2,$ $S(n,X)=\pi^*\mu^{\A}(X)(n)$ and
$T(t,n)=it d\lambda(n) +\pi^*\nabla^2(n)+$ $it [\pi^*\nabla,
v_\sigma](n)$. The assumptions of Section \ref{subsec:appendix3}
of the Appendix are satisfied: the map $R(n)$ and $T(t,n)$ are
slowly increasing along the fiber  and the map $S(n,X)$ does not
depend of the variable $\xi \in N_x$.

The form $\e^{it d\lambda}$ is a finite sum of powers of $t
\,d\lambda$, so that, over $\pi^*(\Kcal_1)$, it is bounded in norm by a fixed polynomial
$P(t,\|\xi\|)$ (it is due to our assumption on $\lambda$).

If we use the estimate (\ref{eq:estimate:D-fourier-slowly}) of the
appendix, we have, for every integer $r$, the estimate
\begin{eqnarray*}
\lefteqn{\Big|\!\Big|\int_{\kgot}
\e^{\F(t)(X)}Q(X)dX\Big|\!\Big|(x,\xi)\leq \cst \, \|Q\|_{\Kcal''\!,2r}\, (1+t)^{\dim N}\times}\\
& &
P(t,\|\xi\|)\frac{(1+\|\xi\|)^\mu}{(1+\|t f_\lambda(x,\xi)\|^2)^r} \,  \e^{-t^2 h_{\sigma}(x,\xi)}
\end{eqnarray*}
for $(x,\xi)\in\pi^{-1}(\Kcal_1)$ and $t\geq 0$. Here $\mu$ does not depend of the choice of
$p$.

Consider the subset $S=\{(x,\xi);\|\xi\|\geq R,x\in \Kcal_1\}$ of
$\pi^{-1}(\Kcal_1)$. Thus, on $S$, the estimate
$h_\sigma(x,\xi)+\|f_\lambda(x,\xi)\|^2\geq c\|\xi\|^2$ holds.
Since  for any positive real  $a,b$, we have $(1+
a)^{-r}\e^{-b}\leq (1+ a + b/r)^{-r}$, we get the following
estimate
$$
\Big|\!\Big|\int_{\kgot} \e^{\F(t)(X)}Q(X)dX\Big|\!\Big|(x,\xi)\leq
\cst \, \|Q\|_{\Kcal''\!,2r}\, (1+t)^{\dim N}
P(t,\|\xi\|)\frac{(1+\|\xi\|)^\mu}{(1+ t^2
\frac{c}{r}\|\xi\|^2)^r}
$$
for $(x,\xi)\in S$ and $t\geq 0$. Combining this estimate with the fact that $r$ can be chosen large
enough, we see that, for any integer $q$, we can find a constant
$\cst(q)$ such that, on $S$, and for any $t\geq 1$,
\begin{equation}\label{petit}
\Big|\!\Big|\int_{\kgot} \e^{\F(t)(X)}Q(X)dX\,\Big|\!\Big|
(x,\xi)\leq \frac{\cst(q)}
 {\left(1+t^2\|\xi\|^2\right)^q}.
 \end{equation}
This  implies that $\ch(\sigma,\lambda,\A,1)=\str\left(\e^{\F(1)}\right)$ is
 rapidly decreasing in mean along the fibers.

Consider now $\beta(\sigma,\lambda,\A,1)= -i\int_{1}^\infty
\str\left(v_\sigma \e^{\F(t)}\right)dt$ which is defined, at least,
for $\|\xi\|\geq R+1$. The estimate (\ref{petit}) shows also that
$\beta(\sigma,\lambda,\A,1)$ is rapidly decreasing  in mean along
the fibers. With the help of Proposition \ref{prop-estimation-generale-slowly}, we can prove in
the same way that all partial derivatives of $\ch(\sigma,\lambda,\A,1)$ and
$\beta(\sigma,\lambda,\A,1)$ are rapidly decreasing in mean along
the fibers: hence
$\ch(\sigma,\lambda,\A,1)$ and $\beta(\sigma,\lambda,\A,1)$ belong respectively
to $\Acal^{\infty}_{\mdr}(\kgot,N)$ and to
$\Acal^{-\infty}_{\mdr}(\kgot, N\setminus C_{\lambda,\sigma})$.

\bigskip

Combining Lemmas \ref{retardord} and \ref{lem:ch=dec-rapid-gene},
we obtain the following proposition.
\begin{prop}\label{prop:ch=dec-rapid-mean}
    The equivariant form
    $\ch(\sigma,\lambda,\A,1)\in  \Acal^{\infty}_{\mdr}(\kgot, N)$
    represents the image of $\chg(\sigma,\lambda)$ in
    $\Hcal^{-\infty}_{\mdr}(\kgot, N)$.
\end{prop}

\bigskip

When the fibers of $\pi:N\to B$ are oriented, we have an integration
morphism $\pi_*: \Hcal^{-\infty}_{\mdr}(\kgot,N)\to
\Hcal^{-\infty}(\kgot,B)$.

\begin{coro}
We have
$\pi_*(\ch(\sigma,\lambda,\A,1))=\pi_*(\chg(\sigma,\lambda))$ in
$\Hcal^{-\infty}(\kgot,B)$.
\end{coro}

\section{The transversally elliptic case}\label{sec:transversally}

Here $N=\T^*M$, where $M$ is a $K$-manifold (not necessarily
compact). We denote by   $\omega$ the Liouville form  on $\T^* M$.
The moment map for the action of $K$ on $(\T^*M,d\omega)$ is the map
$f_\omega:\T^*M\to \kgot^*$ defined by (\ref{eq:f-omega}). We denote
by   $\T^*_K M\subset N$ the set of zeroes of $f_\omega$. In other
words, an element $(x,\xi)$ is in $\T^*_K M$, if $\xi$ vanishes on
all the tangent vectors at $x$ to the orbit $K\cdot x$.

In Example \ref{example:Par-omega}, we have defined in this situation the
generalized equivariant class
$$
\Par(\omega)\in \Hcal^{-\infty}_{\mathbf{T}^*_K M}(\kgot,\T^* M).
$$

Let $\Ecal^\pm\to M$ be Hermitian $K$-vector bundles. Let $p:
\T^* M\to M$ be the projection. Let $\sigma:p^*\Ecal^+\to p^*
\Ecal^-$ be a $K$-equivariant morphism. We suppose that $\sigma$ is
$K$-transversally elliptic: the subset
$$
C_{\omega,\sigma}=\supp(\sigma)\cap \T^*_K M
$$
 is {\em compact}.

Choose an invariant super-connection $\A$ on $p^*\Ecal$, without $0$ exterior
degree term. We consider, as in Subsection \ref{sec:ch-good-lambda}, the
family of invariant super-connections $\A^{\sigma,\omega}(t)=\A+i
t\,(\omega+v_\sigma), t\in \Rbb,$ on $\Ecal$ with equivariant
curvature $\F(\sigma,\omega,\A,t)$. Recall the equivariant forms :
\begin{eqnarray*}
   \eta(\sigma,\omega,\A,t)&=& -i
   \str\left((v_\sigma+\omega)
\,\e^{\F(\sigma,\omega,\A,t)}\right),\\
\beta(\sigma,\omega,\A)&=&\int_{0}^{\infty}\eta(\sigma,\omega,\A,t)dt.
  \end{eqnarray*}

The Chern character $\chg(\sigma,\omega)$ can be constructed as
a class in \break $\Hcal^{-\infty}_{C_{\omega,\sigma}}(\kgot,\T^*M)$ (see Section
\ref{sec:ch-good-lambda}). Since $C_{\omega,\sigma}=\supp(\sigma)\cap \T^*_K M$ is
compact, we have a natural map
$\Hcal^{-\infty}_{C_{\omega,\sigma}}(\kgot,\T^*M)\to\Hcal^{-\infty}_c(\kgot,\T^*M)$,
and we define $\ch_c(\sigma,\omega)$ as the image of
$\chg(\sigma,\omega)$ in $\Hcal^{-\infty}_{c}(\kgot,\T^*M)$ (see
Definition \ref{def:ch-c-sigma-lambda}).
 We gave in Theorem \ref{theo:produit-chg-lambda} another way to
represent the class $\ch_c(\sigma,\omega)$ as the product of
$\Par(\omega)\in \Hcal^{-\infty}_{\mathbf{T}^*_K M}(\kgot,\T^*M)$, with
$\chg(\sigma)\in \Hcal^{-\infty}_{\supp(\sigma)}(\kgot,\T^*M)$.

We summarize our results in
the following proposition.

\begin{prop}\label{prop:ch-c-omega}
$\bullet$ Let $\chi\in\f(\T^* M)$ be a $K$-invariant function, with
\emph{compact support} and equal to $1$ in a neighborhood of
$\supp(\sigma)\cap \T^*_K M$. The following generalized equivariant
form on $\T^*M$
\begin{equation}\label{ch-c-chi}
  c(\sigma,\omega,\A,\chi)=\chi \ch(\A)(X) +
   d\chi\,\beta(\sigma,\omega,\A)(X)
\end{equation}
is closed, with \emph{compact support}, and its cohomology class
$\ch_{c}(\sigma,\omega)$ in \break $\Hcal^{-\infty}_{c}(\kgot,\T^*M)$ does
not depend of the data $(\A,\chi)$. Furthermore this class depends
only of the restriction of $\sigma$ to $\T^*_KM$.

$\bullet$ Let $\chi_1,\chi_2\in\f(\T^* M)$ be $K$-invariant
functions such that : $\chi_1$ is equal to $1$ in a neighborhood of $\T^*_KM$,
$\chi_2$ is equal to $1$ in a neighborhood of $\supp(\sigma)$ and
the product $\chi_1\chi_2$ is compactly supported. Then the product
$$
\Big(\chi_1 + d\chi_1\beta(\omega)(X)\Big)\wedge \Big(\chi_2 \ch(\A)(X) +
   d\chi_2\,\beta(\sigma,\A)(X)\Big)
$$
is a closed equivariant form with generalized coefficients and with \emph{compact support} on $\T^*M$.
Its cohomology class coincides with $\ch_{c}(\sigma,\omega)$ in $\Hcal^{-\infty}_{c}(\kgot,\T^*M)$.

\end{prop}

\subsection{Free action}\label{sec:free-action}

Let $G,K$ be two compact Lie groups. Let $P$ be a compact manifold
provided with an action of $G\times K$. We assume that the action of $K$ is free.
Then the manifold $M:= P/K$ is provided with an action of $G$ and the quotient
map $q: P\to M$ is $G$-equivariant.

We consider the canonical bundle map $\V:P\times \kgot\to\T P$ defined by the $K$-action :
$\V(x,X)=V_x X$.
Let $\theta$ be an invariant connection one form on $P$: it is a $K\times G$-equivariant bundle map
$$
\theta :  \T P\longrightarrow P\times \kgot
$$
such that $\theta\circ \V$ is the identity on $P\times \kgot$.  We may also think at
$\theta$ as an invariant one-form on $P$ with values in $\kgot$.

Let $j: \T^*P\to P\times \kgot^*$ and $\theta^* :  P\times
\kgot^*\longrightarrow \T^* P$ be the bundle maps which are
respectively dual to the bundle maps $\V$ and $\theta$. The kernel
of $j$ is equal to  $\T^*_K P$. We obtain the direct sum
decomposition
$$\T^*P=P\times \kgot^*\oplus\T^*_KP$$
and the dual direct sum decomposition
$$\T P=P\times \kgot \oplus {\rm Hor}.$$

Here  $P\times \kgot$ is isomorphic to the vertical tangent bundle
and the bundle ${\rm Hor}$ is the bundle of horizontal tangent
vectors. The projection $\T^* P\to \T^*_K P$ is defined by
$\eta\mapsto \eta- \theta^*\circ j(\eta)$.

 Note that, for each $x\in P$, we have a canonical
isomorphism $\T^*_K P\vert_x\simeq \T^* M\vert_{q(x)}$ defined by
$\eta\mapsto \eta\circ \T q\vert_{x}$.

\begin{defi}\label{def:Q}
 The smooth map $Q :\T^* P\to \T^* M$ is defined as follows.
For $\eta\in \T^* P\vert_x$, we have  $Q(x,\eta)=(q(x),\eta')$ where $\eta'\in \T^* M\vert_{q(x)}$
is the image of $\eta$ through the projection $\T^* P\vert_x\to \T^*_K P\vert_x$ composed
with the isomorphism $\T^*_K P\vert_x\simeq \T^* M\vert_{q(x)}$.
\end{defi}

Let $\omega_P$ and $\omega_M$ be the Liouville $1$-forms on
$\T^*P$ and $\T^*M$ respectively, and let $f^{^{K\times
G}}_{\omega_P}:\T^*P\to\kgot^*\times\ggot^*$ and
$f^{^G}_{\omega_M}:\T^*M\to\ggot^*$ be the corresponding
equivariant maps.

 We consider the $K\times G$ invariant one form $\nu$ on $P\times \kgot^*$ which is defined by
\begin{equation}\label{eq:forme-nu}
\nu(x,\xi):=\langle \theta(x),\xi\rangle.
\end{equation}
The corresponding map $(f^{^K}_\nu, f^{^G}_\nu) : P\times
\kgot^*\to \kgot^*\times \ggot^*$ satisfies
$f^{^K}_\nu(x,\xi)=\xi$, and $f^{^G}_\nu(x,\xi)=-\xi\circ\mu(x)$.
Here $\mu:P\to\hom(\ggot,\kgot)$ is the {\em moment} of the
connection $1$-form $\theta$ : $\mu(x)(Y)=-\langle \theta(x), V_x
Y\rangle$.

\begin{lem}
$\bullet$ We have $\omega_P= Q^*(\omega_M) + j^*(\nu)$.

$\bullet$ We have
\begin{equation}\label{eq:formule-f-omega-P}
f^{^{K\times G}}_{\omega_P}=\Big(f^{^K}_{\nu}\circ j\, , \,
f^{^G}_{\omega_M}\circ Q  + f^{^G}_{\nu}\circ j\Big).
\end{equation}
\end{lem}

\begin{proof} We write $\theta=\sum_i \theta_i \otimes E^i$ where $(E^i)$ is a base of $\kgot$. We denote
$\langle -, E^i_P\rangle$ the smooth function on $\T^* P$ defined
by $(x,\eta)\mapsto \langle \eta, E^i_P(x)\rangle$. First we have
$j^*(\nu)=\sum_i p^*(\theta_i)\langle -, E^i_P\rangle$  where
$p:\T^* P\to P$ is the projection. Next we have, for
$(x,\eta)\in\T^*P$ and $v\in \T(\T^* P)\vert_{(x,\eta)}$, the
relations :
\begin{eqnarray*}
\langle Q^*(\omega_M)(x,\eta), v\rangle  &=&
\langle \omega(Q(x,\eta)), \T Q\vert_{(x,\eta)}\xi\rangle\\
&=&\langle \eta', \T p_1\vert_{Q(x,\eta)}\circ \T Q\vert_{(x,\eta)} v\rangle\\
&=&\langle \eta',\T (q\circ p_2)\vert_{(x,\eta)} v\rangle\\
&=&\langle \eta-\sum_i \theta_i(x)\langle \eta, E^i_P(x)\rangle,\T p_2(v)\rangle\\
&=&\langle \omega_P(x,\eta), v\rangle +\langle j^*(\nu)(x,\eta),
v\rangle.
\end{eqnarray*}
Here $Q(x,\eta)=(q(x),\eta')$ and we use the relation $p'\circ
Q=q\circ p$, where $p':\T^*M\to M$ is the projection. The last
point is a consequence of the first one.
\end{proof}

\medskip

Following Section \ref{subsec:trivial}, we associate to the invariant $1$-forms $\omega_P$ and $\omega_M$ the
relative equivariant classes :
\begin{enumerate}
\item[$\bullet$] $\pr(\omega_P)\in \Hcal^{-\infty}(\kgot\times
\ggot,\T^* P,\T^* P\setminus\T^*_{K\times G}P)$, \item[$\bullet$]
$\pr(\omega_M)\in \Hcal^{-\infty}(\ggot,\T^* M,\T^*
M\setminus\T^*_{G}M)$.
\end{enumerate}

We are in the setting of  Section
\ref{sec:product-transv-elliptic}. We consider the manifold
$N:=\T^*P$ equipped with the actions of the group $K_1:= K$ and
$K_2:= G$, and the invariant one forms $\mu=Q^*(\omega_M)$,
$\lambda=j^*(\nu)$.

 We first consider the
$K\times G$ invariant form $\nu$ on $P\times \kgot^*$, and the map
$f_\nu=(f_\nu^K, f_\nu^G)$ from $P\times \kgot^*$ to
$\kgot^*\times\ggot^*$.
\begin{lem}\label{dabor}
We have
$$C_{\nu}=C_{\nu}^{K}=P\times \{0\}.$$
\end{lem}
\begin{proof}
 As $f^{^K}_\nu(x,\xi)=\xi$, and
 $f^{^G}_\nu(x,\xi)=-\xi\circ\mu(x)$,
 the relation $f^{^K}_\nu(x,\xi)=0$ implies that $f_\nu^G=0$, so
 $f_\nu=0$.
\end{proof}

\bigskip

We consider the class $\prf(\nu)\in
\Hcal^{-\infty,\infty}(\kgot\times \ggot,P\times \kgot^*, P\times
(\kgot^*\setminus \{0\}))$.

The pull-backs $Q^*\left(\pr(\omega_M)\right)$ and $
j^*\left(\prf(\nu)\right)$  belong respectively to \break
$\Hcal^{-\infty}(\ggot,\T^* P,\T^* P\setminus\T^*_{G}P)$,
$\Hcal^{-\infty,\infty}(\kgot\times \ggot,\T^* P,\T^*
P\setminus\T^*_{K}P)$: the relative class
$Q^*\left(\pr(\omega_M)\right)(X,Y)$ does not depend of
$X\in\kgot$ and the relative class \break
$Q^*\left(\prf(\nu)\right)(X,Y)$ is smooth relatively to
$Y\in\ggot$. We can then take the product
$$
Q^*\left(\pr(\omega_M)\right)(Y)\diamond
j^*\left(\prf(\nu)\right)(X,Y)
$$
which belongs to $\Hcal^{-\infty}(\kgot\times \ggot,\T^* P,\T^*
P\setminus\T^*_{K\times G}P)$.

The main point of this section is the following
\begin{theo}\label{th:par-action-libre}
The following equality
$$
\pr(\omega_P)=Q^*\left(\pr(\omega_M)\right)\diamond
j^*\left(\prf(\nu)\right)
$$
holds in $\Hcal^{-\infty}(\kgot\times \ggot,\T^* P,\T^*
P\setminus\T^*_{K\times G}P)$.
\end{theo}

\begin{proof}
This  theorem follows from  Theorem
\ref{theo:ch-rel-produit-groupe} and of the following description
of the  sets where we need to work. Indeed, let us see that we
have
$$C_{Q^*\omega_M}=C_{Q^*\omega_M}^G=Q^*\T^*_GM,\hspace{0.5cm}
C_{j^*\nu}=C_{j^*\nu}^{K}=\T^*_KP$$ and
$$\Ccal(Q^*\omega_M,j^*\nu)=C_{\omega_P}=T^*_{G\times K}P.$$
 As the component of $f_{\omega_M}$ on $\kgot^*$ is equal to $0$,
the first equality is clear. The second equality follows from
Lemma \ref{dabor}.

To compute $\Ccal(Q^*\omega_M,j^*\nu)$, we take some invariant
metrics on $\ggot^*$ and $\kgot^*$.
 The set $\Ccal(Q^*\omega_M,j^*\nu)$ is the set
$ \left\{\|Q^*f_{\omega_M}^{G} \|\leq \|j^*f_\nu^{G}\|\right\}
\bigcap \left\{\|j^*f_\nu^{K}\|\leq  0 \right\}.$

Thus, on $\Ccal(Q^*\omega_M,j^*\nu)$, we have $j^*f_\nu^{K}=0$. As
shown by Lemma \ref{dabor}, this implies $j^*f_\nu^G=0$, so that
all maps $j^*f_\nu^{K},j^*f_\nu^{G}, Q^*f_{\omega_M}^{G}$ are zero
on $\Ccal(Q^*\omega_M,j^*\nu)$. We obtain the last equality.
\end{proof}

\bigskip

We denote by $\Par^1(\nu)\in \Hcal_P^{-\infty,\infty}(\kgot \times
\ggot,P\times \kgot^*)$ the image of the class $\prf(\nu)$. Then
$\Par^1(\nu)(X,Y)$ depends smoothly on $Y$. As $P$ is compact, it
defines a class still denoted by $\Par^1(\nu)$ in
$\Hcal_c^{-\infty,\infty}(\kgot \times \ggot,P\times \kgot^*)$

Let $\sigma_P: p^*\Ecal^+\to p^*\Ecal^-$ be a $K\times
G$-transversally elliptic morphism on $\T^*P$. Let
$\overline{\Ecal}^\pm\to M$ be the vector bundles equal to the
quotient $\Ecal^\pm/K$. We define the morphism $\sigma_M :
p^*\overline{\Ecal}^+\to p^*\overline{\Ecal}^-$ on $\T^*M$ by the
relation
$$
Q^*\sigma_M(x,\eta)=\sigma_P([x,\eta]_{{\rm T}^*_K P}),\quad
(x,\eta)\in\T^*P.
$$
Here $(x,\eta)\to [x,\eta]_{{\rm T}^*_K M}$ denotes the projection $\T^*P\to \T^*_K P$.

It is immediate to see that $\sigma_M$ is $G$-transversally
elliptic, and that $Q^*\sigma_M$ defines the same class than
$\sigma_P$ in $\KK_{K\times G}(\T_{K\times G}^*P)$.

\begin{prop}
We have the following equality
$$
\ch_c(\sigma_P,\omega_P)=
Q^*\Big(\ch_c(\sigma_M,\omega_M)\Big)\wedge
j^*\Big(\Par^1(\nu)\Big)
$$
in $\Hcal^{-\infty}_c(\kgot\times\ggot,\T^*P)$.
\end{prop}
\begin{proof}
We use here the results of Section
\ref{sec:product-transv-elliptic}. We work on $N:=\T^*P$ with the
symbol $\sigma= Q^*\sigma_M$ and the symbol $\tau=[0]$. The Chern
character with compact support
$\ch_c(\sigma_P,\omega_P)=\ch_c(Q^*(\sigma_M),Q^*(\omega_M)+j^*(\nu))$
is equal to the product
$$
c_{\Vcal_1}(Q^*(\sigma_M),Q^*(\omega_M))\wedge c_{\Vcal_2}([0],j^*(\nu))
$$
in $\Hcal^{-\infty}_c(\kgot_1\times \kgot_2,N)$. Here $\Vcal_1$ is
any neighborhood of
$$
\supp(Q^*(\sigma_M))\cap
C_{Q^*(\omega_M)}=Q^{-1}\left(\supp(\sigma_M)\cap \T^*_G M\right)
$$
and $\Vcal_2$ is any neighborhood of
$$
\supp([0])\cap C_{j^*(\nu)}= j^{-1}(P\times\{0\})= \T^*_K P
$$
with the condition  that $\overline{\Vcal_1}\cap \overline{\Vcal_2}$ is compact.

Here we take $\Vcal_1$ of the form $Q^{-1}(\Ucal_1)$ where
$\Ucal_1$ is a neighborhood of \break
$\supp(\sigma_M)\cap\T^*_{K}M$ in $\T^*M$ with
$\overline{\Ucal_1}$ compact. We take $\Vcal_2$ of the form
$j^{-1}(\Ucal_2)$ where $\Ucal_2=\{(p,\xi)\in P\times \kgot^*\
\vert\ \|\xi\|\leq \epsilon\}$ is defined for $\epsilon$ small
enough.

Then the class $c_{\Vcal_1}(Q^*(\sigma_M),Q^*(\omega_M))$ and
$c_{\Vcal_2}([0],j^*(\nu))$ are respectively equal to
$Q_1^*(c_{\Ucal_1}(\sigma_M,\omega_M))$ and to $j^*(
c_{\Ucal_2}([0],\nu))$. The class $c_{\Ucal_1}(\sigma_M,\omega_M)$
is equal to $\ch_c(\sigma_M,\omega_M)$ in
$\Hcal^{-\infty}_c(\ggot,\T^*M)$, and the class
$c_{\Ucal_2}([0],\nu)$ defines $\Par^1(\nu)$ in
 $\Hcal^{-\infty,\infty}_{c}(\kgot\times\ggot,P\times\kgot^*)$.

\end{proof}\bigskip

\subsection{Exterior product}\label{sec:product-transv-elliptic-bis}

To define products of symbols, we will need to use ``almost
homogeneous symbols''.

\begin{defi} A morphism $\sigma: p^*\Ecal^+\to p^*\Ecal^-$ over $\T^* M$ is said
to be {\rm almost homogeneous} of order $m$ if $\sigma([x,t\xi])=
t^m\sigma([x,\xi])$, for every $t\geq 1$ and for $\xi$ large
enough\footnote{It means that $\|\xi\|\geq c$ for some Riemannian
metric $\|\cdot\|$ and some constant $c>0$.}.
\end{defi}

\begin{lem} A $K$-transversally elliptic morphism $\sigma$ is homotopic to a
$K$- transversally elliptic morphism which is furthermore {\rm
almost homogeneous} of order $0$.
\end{lem}

\begin{proof} Let $c>0$ such that $C_{\omega,\sigma}=\supp(\sigma)\cap\T^*_{K}M\subset$ \break
$\{ (x,\xi)\in\T^*M\,\vert\, \|\xi\|\leq c\}$. We consider a
smooth function $\phi:\Rbb\to \Rbb^{\geq 0}$ satisfying: $\phi=1$
on $[0,c]$, $\phi\geq 1$ on $[c,2c]$, and  $\phi(y)= \frac{2c}{y}$
for $y\geq 2c$.

We define now, for $s\in[0,1]$, the morphism
$\sigma^s(x,\xi):=\sigma(x,\phi(s\|\xi\|)\xi)$. We see easily that
$C_{\omega,\sigma}=C_{\omega,\sigma^s}$ for all $s\in[0,1]$. Hence
$\sigma=\sigma^0$ is homotopic to $\sigma^1$ which is almost
homogeneous of order $0$.
\end{proof}

\bigskip

Let $K_1,K_2$ be two compact Lie groups. We work with
the following data:
\begin{enumerate}
\item[$\bullet$] $M_1$ is a $K_1\times K_2$-manifold not
\emph{necessarily compact} , \item[$\bullet$] $M_2$ is a
\emph{compact} $K_2$-manifold , \item[$\bullet$] $\Ecal_1$ is a
$K_1\times K_2$-equivariant complex super-vector bundle on $M_1$,
\item[$\bullet$] $\Ecal_2$ is a $K_2$-equivariant complex
super-vector bundle on $M_2$, \item[$\bullet$]
$\sigma_1:p_1^*\Ecal_1^+\to p_1^*\Ecal_1^-$ is a $K_1\times
K_2$-equivariant morphism on $\T^*M_1$ which is
$K_1$-transversally elliptic \item[$\bullet$]
$\sigma_2: p_2^*\Ecal_2^+\to p_2^*\Ecal_2^-$ is a
$K_2$-equivariant morphism on $\T^*M_2$ which is
$K_2$-transversally elliptic.
\end{enumerate}

We consider now the exterior product $\sigma:=\sigma_1\odot_{\rm
ext}\sigma_2$ which is an equivariant morphism on $M:=M_1\times
M_2$ with support equal to $\supp(\sigma_1)\times\supp(\sigma_2)$.
Since $\T^*_{K_1\times K_2}(M_1\times M_2)\neq \T^*_{K_1}M_1\times
\T^*_{K_2} M_2$, the morphism $\sigma$ is not necessarily
$K_1\times K_2$-transversally elliptic. However, we will see that
it is so when  the morphism $\sigma_2$ is taken {\bf almost
homogeneous} of order $0$.

\medskip

For $k=1,2$, let  $p_k: \T^*(M_1\times M_2)\to \T^*M_k$ be the
projection. The Liouville one form $\omega$ on $\T^*(M_1\times
M_2)$ is equal to $p_1^*\omega_1+p_2^*\omega_2$, where $\omega_k$
is the Liouville one form
 on $\T^*M_k$.

\begin{lem} Assume that the morphism $\sigma_2$ is taken {\bf almost
homogeneous} of order $0$. Then the morphism
$\sigma:=\sigma_1\odot_{\rm ext}\sigma_2$ on $M:=M_1\times M_2$ is
$K_1\times K_2$-transversally elliptic.
\end{lem}

\begin{proof} Let $f_{\omega_2}^{2}:\T^*M_2\to \kgot_2^*$ and
$(f_{\omega_1}^{1}, f_{\omega_1}^{2}): \T^*M_1\to
\kgot_1^*\times \kgot_2^*$ be the moment maps associated to the
actions of $K_2$ on $M_2$ and $K_1\times K_2$ on $M_1$. An element
$(n_1,n_2)\in \T^*M_1\times \T^*M_2$ belongs to $\T^*_{K_1\times
K_2}(M_1\times M_2)$ if and only if  $f_{\omega_1}^{1}(n_1)=0$ and
$f_{\omega_1}^{2}(n_1)+ f_{\omega_2}^{2}(n_2)=0$.

Let $f_{2}$ be the  restriction of the map
$f_{\omega_2}^{2}$ to the subset $\supp(\sigma_2)$. Then
$$
\supp(\sigma_1\odot_{\rm ext}\sigma_2)\cap \T^*_{K_1\times K_2}(M_1\times M_2)
\subset C^1_{\sigma_1,\omega_1}\times f_{2}^{-1}(\Kcal)
$$
where $C^1_{\sigma_1,\omega_1}=\supp(\sigma_1)\cap \T^*_{K_1}(M_1)$ and
$\Kcal:=-f_{\omega_1}^{K_2}(C^1_{\sigma_1,\omega_1})$
are compacts. The proof follows from the

\begin{lem} \label{lem:f-2-propre}
The map $f_{2}:\supp(\sigma_2)\to\kgot_2^*$ is proper.
\end{lem}

\begin{proof}  Since $\sigma_2$ is $K_2$-transversally
elliptic, we have
\begin{equation}\label{eq:f-2}
\supp(\sigma_2)\cap \T^*_{K_2}M_2\subset \{\|\xi_2\| < c\}.
\end{equation}

Thus the function $\|f_2\|$ is strictly positive on the compact subset \break
$\supp(\sigma_2)\cap \{\|\xi_2\|=c\}$. We choose $u>0$ such
that $\|f_2\|>u$ on \break $\supp(\sigma_2)\cap \{\|\xi_2\|=c\}$.

As $\sigma_2$ is almost homogeneous of order $0$, we can choose
$c$ sufficiently large such that $\sigma_2([x,t\xi])=
\sigma_2([x,\xi])$ for every $t\geq 1$ and for $\|\xi\|\geq c$, so
that $S'_2=\supp(\sigma_2)\cap \{\|\xi_2\|\geq c\}$ is stable by
multiplication by $t\geq 1$. It is sufficient to prove that the
restriction of $f_2$ to $S'_2$ is proper.
 By homogeneity, $f_2([x,\xi])>uc\|\xi\|$
 on $S'_2$. It follows that  the restriction of ${\rm f_2}$ to
$S'_2$ is proper.

\end{proof}
\end{proof}

\bigskip

Consider first $N_1=\T^*M_1$ with the $K_1\times K_2$ invariant
form $\omega_1$. We are in the situation of Section
\ref{sec:product-transv-elliptic}. We write the map
$f_{\omega_1}:N_1\to \kgot_1^*\times \kgot_2^*$ as
$(f_{\omega_1}^{1},f_{\omega_1}^{2})$. We see that the set
$C^1_{\omega_1}$ is just $\T^*_{K_1}M_1$. Let
$$
C^1_{\sigma_1,\omega_1}=\supp(\sigma_1)\cap \T^*_{K_1}M_1.
$$
By our assumption, this is a compact subset of $\T^*M_1$.
 The relative class $\chrf(\sigma_1,\omega_1)$ belongs to
$\Hcal^{-\infty,\infty}(\kgot_1^*\times \kgot_2^*,N_1\setminus
C^1_{\sigma_1,\omega_1})$. We  consider its image
$\ch_c^1(\sigma_1,\omega_1)$ in
$\Hcal_c^{-\infty,\infty}(\kgot_1\times \kgot_2, \T^*M_1)$. A
representant of $\ch_c^1(\sigma_1,\omega_1)$ is given by
$c(\sigma_1,\omega_1,\A_1,\chi_1)$ (see (\ref{ch-c-chi})), where
$\chi_1$ is an invariant compactly supported function on $\T^*
M_1$ which is  equal to $1$ in a neighborhood of
$\supp(\sigma_1)\cap \T^*_{K_1}M_1$.

From now on, we assume that the morphism $\sigma_2$ is taken {\bf
almost homogeneous} of order $0$.

We consider the Chern classes with compact support
 associated to the transversally elliptic morphisms
$\sigma_1$,  $\sigma_2$ and $\sigma:=\sigma_1\odot_{\rm
ext}\sigma_2$ :

$\bullet$ $\ch_c^1(\sigma_1,\omega_1)\in
\Hcal^{-\infty,\infty}_{c} (\kgot_1\times \kgot_2,\T^*M_1)$,

$\bullet$ $\ch_c(\sigma_2,\omega_2)\in
\Hcal^{-\infty}_{c}(\kgot_2,\T^*M_2)$,

$\bullet$ $\ch_c(\sigma,\omega)\in
\Hcal^{-\infty}_{c}(\kgot_1\times \kgot_2,\T^*M)$.

 We may then form the product of the generalized equivariant
forms $p_1^*\ch_c^1(\sigma_1,\omega_1)$ and
$p_2^*\ch_c(\sigma_2,\omega_2)$. The main result of this Section
is the

\begin{theo}\label{theo:produit-exterieur-gene}
The following equality
$$
p_1^*\ch_c^1(\sigma_1,\omega_1)(X,Y)\wedge
p_2^*\ch_c(\sigma_2,\omega_2)(Y)=\ch_c(\sigma,\omega)(X,Y)
$$
holds in $\Hcal^{-\infty}_{c}(\kgot_1\times \kgot_2,\T^*M)$.
\end{theo}

\begin{proof}
This theorem follows from the results proved in Section
\ref{sec:product-transv-elliptic}.  We consider the manifold
$N:=\T^*(M_1\times M_2)$ equipped with the actions of the groups
$K_1$, $K_2$ and the invariant one forms $\lambda=p_1^*(\omega_1)$
and $\mu:=p_2^*(\omega_2)$ : $\lambda+\mu=\omega$.  The morphism
$\sigma$ is equal to the product $p_1^*(\sigma_1)\odot
p_2^*(\sigma_2)$. As the component of $f_\mu$ on $\kgot_1^*$ is
equal to $0$, the closed subset $\Ccal:=\Ccal(\lambda,\mu)$ of $N$
is equal to
\begin{eqnarray*}
\Ccal&:=&
 \left\{\|f_{\lambda}^{1} \| = 0 \right\}\bigcap \left\{\|f_\mu\|\leq \|f_\lambda^2\|  \right\}\\
&=&\T^*_{K_1}M\bigcap \left\{(n_1,n_2)\in \T^*M_1\times \T^*M_2\ \  \vert \ \
\|f_{\omega_2}^{2}(n_2)\|\leq \|f_{\omega_1}^{2}(n_1)\|  \right\}.
\end{eqnarray*}
Let us check that $\supp(\sigma)\cap\Ccal$ is compact. Since
$\supp(\sigma)=\supp(\sigma_1)\times \supp(\sigma_2)$, we have
$$
\supp(\sigma)\cap\Ccal=
\left\{(n_1,n_2)\in C^1_{\omega_1,\sigma_1}\times \supp(\sigma_2)\ \  \vert \ \
\|f_{\omega_2}^{2}(n_2)\|\leq \|f_{\omega_1}^{2}(n_1)\|  \right\}
$$
where $ C^1_{\omega_1,\sigma_1}=\supp(\sigma_1)\cap \T^*_{K_1}M_1$ is compact.
 Let $c>0$ such that
$\|f_{\omega_1}^{2}(n_1)\|\leq c$ for all $n_1\in  C^1_{\omega_1,\sigma_1}$. We have then
$$
\supp(\sigma)\cap\Ccal\subset
 C^1_{\omega_1,\sigma_1}\times \Big(\supp(\sigma_2)\cap \{\|f_{\omega_2}^{2}\|\leq c\}\Big)
$$
We know from Lemma \ref{lem:f-2-propre} that the map
$f_{\omega_2}^{2}$ is proper on $\supp(\sigma_2)$: the set
$\supp(\sigma_2)\cap \{\|f_{\omega_2}^{2}\|\leq c\}$ is compact
and then $\supp(\sigma)\cap\Ccal$ is also compact.

We are exactly in the situation of Theorem \ref{twogroups}. The
 equivariant Chern character with
compact support $\ch_c(\sigma,\omega)=\ch_c(p_1^*(\sigma_1)\odot
p_2^*(\sigma_2),\lambda+\mu)$ is equal to the product
$$
c_{\Vcal_1}(p_1^*(\sigma_1),\lambda)\wedge
c_{\Vcal_2}(p_2^*(\sigma_2),\mu)
$$
in $\Hcal^{-\infty}_c(\kgot_1\times \kgot_2,\T^*M)$. Here
$\Vcal_1$ and $\Vcal_2$ are respectively any neighborhood of
$\supp(p_1^*(\sigma_1))\cap C_\lambda^1
=p_1^*(\supp(\sigma_1)\cap\T^*_{K_1}M_1)$ and
$\supp(p_2^*(\sigma_2))\cap C_\mu=$\break
$p_2^*(\supp(\sigma_2)\cap\T^*_{K_2}M_2) $ such that
$\overline{\Vcal_1}\cap \overline{\Vcal_2}$ is compact. Here we
take $\Vcal_k$ of the form $p_k^{-1}(\Ucal_k)$ where $\Ucal_k$ is
a neighborhood of $\supp(\sigma_k)\cap\T^*_{K_k}M_k$ in $\T^*M_k$
with $\overline{\Ucal_k}$ compact. Then the class
$c_{\Vcal_1}(p_1^*(\sigma_1),\lambda)$ and
$c_{\Vcal_2}(p_1^*(\sigma_1),\mu)$ are respectively equal to
$p_1^*(c_{\Ucal_1}(\sigma_1,\omega_1))$ and to $p_2^*(
c_{\Ucal_2}(\sigma_2,\omega_2))$. The proof is completed since
each $c_{\Ucal_1}(\sigma_1,\omega_1)$ is equal to
$\ch_c^1(\sigma_1,\omega_1)$ in
 $\Hcal^{-\infty,\infty}_c(\kgot_1\times \kgot_2,\T^*M_1)$,
 while $c_{\Ucal_2}(\sigma_2,\omega_2)$ is equal to
$\ch_c(\sigma_2,\omega_2)$ in
 $\Hcal^{-\infty}_c( \kgot_2,\T^*M_2)$.

\end{proof}

\subsection{The Berline-Vergne Chern character}\label{sec:chBV}

In this section, we compare $\ch_{c}(\sigma,\omega)$  with the class
defined by Berline-Vergne \cite{B-V.inventiones.96.1}, with the help
of transversally good symbols. We suppose in this Section that the manifold
$M$ is compact.

In Berline-Vergne \cite{B-V.inventiones.96.1}, we associated to  a
``transversally good elliptic symbol" $\sigma$ a class
$\chbv(\sigma)$ which was an equivariant differential form on
$\T^*M$ with smooth coefficients, rapidly decreasing in mean on
$\T^*M$. If $\sigma$ is any transversally elliptic symbol, the
Chern character $\ch_c(\sigma,\omega)$ is compactly supported on
$\T^*M$, so defines an element of
$\Hcal^{-\infty}_{\mdr}(\kgot,\T^*M)$. Our aim is to prove that
the classes $\chbv(\sigma)$ and $\ch_c(\sigma,\omega)$ coincide in
$\Hcal^{-\infty}_{\mdr}(\kgot,\T^*M)$.

\medskip

We recall the definition of $\chbv(\sigma)$.
 A $K$-transversally
elliptic symbol $\sigma: p^*\Ecal^+\to  p^*\Ecal^-$ is "good"  if
it satisfies the following conditions:

$\bullet$ $\sigma$ and all its derivatives are slowly increasing
along the fibers,

$\bullet$ the endomorphism $v_\sigma^2$ is "good" with respect to
the moment map $f_\omega$. That is, there exists $r>0$, $c>0$ and
$a>0$ such that\footnote{$h_\sigma(x,\xi)\geq 0$ is the smallest
eigenvalue of the positive hermitian endomorphism
$\sigma(x,\xi)^2$.} for every $(x,\xi)$ :
\begin{equation}\label{eq-sigma-good}
    \|f_\omega(x,\xi)\|\leq a\|\xi\|\  \mathrm{and}\
\|\xi\|\geq r \Longrightarrow h_\sigma(x,\xi)\geq c\|\xi\|^2.
\end{equation}

\medskip

Let $\A= p^*\nabla$, where $\nabla=\nabla^+\oplus \nabla^-$ is a
sum of connections on the bundles $\Ecal^{\pm}\to B$. Consider the
invariant super-connection $\A_1=\A+iv_\sigma+i\omega$ with equivariant
curvature
$$
\F(\sigma,\A,1)(X)= -v_\sigma^2+i \langle f_\omega,X\rangle +i
\Omega+i[\A, v_\sigma]+\A^2+\mu^{\A}(X).
$$

If $\sigma$ is a $K$-transversally good symbol, Berline and Vergne
have shown that the smooth equivariant form
$\ch(\sigma,\omega,\A,1)=\e^{iD\omega}\str(\e^{F(\sigma,\A,1)})$ is
rapidly decreasing in  mean: for any test function $Q$ on $\kgot$,
$\int_{\kgot}\ch(\sigma,\omega,\A,1)(X)Q(X)dX$ is a differential
form on $\T^*M$ which is rapidly decreasing along the fibers of the
projection $\T^*M\to M$. It thus defines a class
$$
\chbv(\sigma)\in \Hcal^{\infty}_{\mdr}(\kgot, \T^*M).
$$

\begin{theo}\label{theo:ch-PV-BV}
If $\sigma$ is a transversally good symbol, then
$\chbv(\sigma)=\ch_c(\sigma,\omega)$ in $\Hcal^{-\infty}_{\rm
mean-dec-rap}(\kgot, \T^*M)$. In particular, the integrals on the
fibers of $\chbv(\sigma)$ and of $\ch_c(\sigma,\omega)$ defines
the same element in $\Hcal^{-\infty}(\kgot,M)$.

\end{theo}

\begin{proof}
 The condition (\ref{eq-sigma-good}) implies that
 $$
 h_\sigma(x,\xi)+\|f_\omega(x,\xi)\|^2\geq c'\|\xi\|^2\quad \mathrm{when}
 \quad \|\xi\|\geq r
 $$
 where $c'=\min(a^2,c)$.  So we can exploit the result of
 Proposition \ref{prop:ch=dec-rapid-mean} : we have
    $\ch(\sigma,\omega,\A,1)=\chg(\sigma,\omega)$ in
    $\Hcal^{-\infty}_{\mdr}(\kgot, \T^*M)$. The proof is
    finished, since $\chg(\sigma,\omega)=\ch_c(\sigma,\omega)$ in
    $\Hcal^{-\infty}_{\mdr}(\kgot, \T^*M)$.

 \end{proof}\bigskip

\section{Appendix}

We give  proofs of the estimates used in this article. They are all
based on  Volterra's expansion formula: if $R$ and $S$ are elements
in a finite dimensional associative algebra, then
\begin{equation}\label{volterra}
\e^{(R+S)}=\e^{R}+\sum_{k=1}^{\infty} \int_{\Delta_k}\e^{s_1 R} S
\e^{s_2 R} S \cdots S \e^{s_{k}R} S \e^{s_{k+1}R}ds_1\cdots ds_{k}
\end{equation}
where $\Delta_{k}$ is the simplex $\{s_i\geq 0;
s_1+s_2+\cdots+s_{k}+s_{k+1}=1\}.$ We recall that the volume of
$\Delta_{k}$  for the measure $ds_1\cdots ds_{k} $ is
$\frac{1}{k!}$.

Now, let $\Acal=\oplus_{i=0}^q \Acal_i$  be a finite dimensional
graded commutative algebra with a norm $\|\cdot\|$ such that
$\|ab\|\leq \|a\|\|b\|$. We assume $\Acal_0=\Cbb$ and we denote by
$\Acal_+=\oplus_{i=1}^q \Acal_i$. Thus $\omega^{q+1}=0$ for any
$\omega \in \Acal_+.$ Let $V$ be a finite dimensional  Hermitian
vector space. Then $\End(V)\otimes \Acal$ is an algebra with a norm
still denoted by $\|\cdot\|$. If $S\in \End(V)$, we denote also by
$S$ the element $S\otimes 1$ in $\End(V)\otimes \Acal$.

\begin{rem} In the rest of this section we will denote $\cst(a,b,\cdots)$ some positive constant
which depends on the parameter $a,b,\cdots$.
\end{rem}

\subsection{First estimates}

We denote $\herm(V)\subset\End(V)$ the subspace formed by the
Hermitian endomorphisms. When $R\in\herm(V)$, we denote
$\sm(R)\in\Rbb$ the smallest eigenvalue of $R$ : we have
$$
\Big|\!\Big| \e^{-R}\Big|\!\Big|=\e^{-\sm(R)}.
$$

\begin{lem}\label{suffit}
Let $\Pcal(t)=\sum_{k=0}^q \frac{t^k}{k!}$. Then, for any $S\in
\End(V)\otimes \Acal$, $T\in \End(V)\otimes \Acal_+$, and
$R\in\herm(V)$, we have
$$
\|\e^{-R+S+T}\|\leq \e^{-\sm(R)}\e^{\|S\|}\Pcal(\|T\|).
$$
\end{lem}

\begin{proof} Let $c=\sm(R)$.
Then $\|\e^{-u R}\|= \e^{-uc}$ for all $u\geq 0$. Using Volterra's
expansion for the couple  $sR,s S$, we obtain $\|\e^{s(-R+S)}\|\leq
\e^{-sc}\e^{s\|S\|}.$ Indeed,
$\e^{s(-R+S)}=\e^{-sR}+\sum_{k=1}^{\infty} I_k$ with
$$
I_k= s^k\int_{\Delta_k}\e^{-s_1s R}S\cdots S\e^{-s_ks R}
 S \e^{-s_{k+1}sR}ds_1\cdots ds_k.
$$
The term $I_k$ is bounded in norm by $\frac{s^k}{k!}\|S\|^k\e^{-sc}$.
Summing in $k$, we obtain $\|\e^{-s(R+S)}\|\leq \e^{-sc}\e^{s\|S\|}$
for $s\geq 0$. We reapply Volterra's expansion to compute
$\e^{(-R+S)+T}$ as the sum
$$\e^{-R+S}+\sum_{k\geq 1}^q
\int_{\Delta_k}\e^{s_1(-R+S)}T\cdots T\e^{s_k(-R+S)}
T\e^{s_{k+1}(-R+S)}ds_1\cdots ds_k.
$$

Here the sum in $k$ is finite and stops at $k=q$.  The norm of the
$k^{th}$ term is bounded by $\frac{1}{k!}\e^{-c}\e^{\|S\|}\|T\|^k$.
Summing up in $k$, we  obtain our estimate.
\end{proof}\bigskip

\bigskip

For proving  Proposition \ref{estimatesgen},   we need to consider
the following situation. Let $E$ be a (finite dimensional) vector
space. We consider the following smooth maps
\begin{itemize}
  \item $x\mapsto S(x)$ from $E$ to $\End(V)\otimes\Acal$.
  \item $(t,x)\mapsto t^2 R(x)$ from $\Rbb\times E$ to $\herm(V)$.
  \item $(t,x)\mapsto T(t,x)=T_0(x)+t T_1(x)$ from $\Rbb\times E$ to
  $\End(V)\otimes\Acal_+$.
\end{itemize}


\begin{prop}\label{prop-estimation-generale}
Let $D(\partial)$ be a constant coefficient differential operator in
$x\in E$ of degree $r$. Let $\Kcal$ be a compact subset of $E$.
There exists a constant $\cst>0$ (depending on $\Kcal,R(x)$, $S(x),T_0(x),T_1(x)$
 and $D(\partial)$)  such that\footnote{$q$
is highest degree of the graded algebra $\Acal$.}
\begin{equation}\label{eq:maj-D-exp}
    \Big|\!\Big|D(\partial)\cdot \e^{-t^2R(x)+S(x)+T(t,x)}\Big|\!\Big|
\leq \cst\, (1+t)^{2r+q}\,  \e^{-t^2\sm(R(x))},
\end{equation}
for all $(x,t)\in\Kcal\times \Rbb^{\geq 0}$.
\end{prop}

\begin{coro}\label{coro:integration-smooth}

Let $\Ucal$ be an open subset of $E$ such that $R(x)$ is positive
definite for any $x\in\Ucal$, that is $\sm(R(x))>0$ for all $x\in
\Ucal$. Then the integral
$$
\int_0^{\infty}\e^{-t^2R(x)+S(x)+T(t,x)}dt
$$
defines a smooth map from $\Ucal$ into $\End(V)\otimes\Acal$.
\end{coro}

\begin{proof} We fix a basis $v_1,\ldots,v_p$ of $E$. Let us denote
$\partial_i$ the partial derivative along the vector $v_i$. For
any sequence $I:=[i_1,\ldots,i_n]$ of integers
$i_k\in\{1,\ldots,p\}$, we denote $\partial_I$ the differential
operator of order $n=|I|$ defined by the product $\prod_{k=1}^n
\partial_{i_k}$.

For any smooth function $g:E\to\End(V)\otimes\Acal$ we define the
functions
$$
\Big|\!\Big| g \Big|\!\Big|_{n}(x):=\sup_{|I|\leq
n}\ \|\partial_I\cdot g(x)\|
$$
and the semi-norms $\| g\|_{\Kcal,n}:=\sup_{x\in\Kcal} \| g\|_{n}(x)$ attached to a compact subset
$\Kcal$ of $E$. We will use the trivial fact that $\| g\|_{n}(x)\leq \| g\|_{m}(x)$ when $n\leq m$.
Since any constant differential operator $D(\partial)$ is a finite
sum $\sum_I a_I\partial_I$, it is enough to proves
(\ref{eq:maj-D-exp}) for the $\partial_I$.

First, we analyze $\partial_I\cdot\left(\e^{-t^2R(x)}\right)$.
The Volterra expansion formula gives
\begin{equation}\label{eq-derive-1}
\partial_i\cdot\left(\e^{-t^2R(x)}\right)= -t^2
\int_{\Delta_1}\e^{-s_1 t^2 R(x)} \partial_i\cdot R(x)\, \e^{-s_2 t^2
R(x)}ds_1 ,
\end{equation}
and then $\|\partial_i\cdot \e^{-t^2R(x)}\| \leq
\| R\|_{1}(x)\, (1+t)^2\, \e^{-t^2\sm(R(x))}$ for
$(x,t)\in E\times \Rbb^{\geq 0}$.

With (\ref{eq-derive-1}), one can easily prove by induction on the
degree of $\partial_I$ that: if $| I  |=n$  then
\begin{equation}\label{eq:maj-D-exp-R}
    \Big|\!\Big|\partial_I\cdot \e^{-t^2R(x)}\Big|\!\Big|
\leq \cst(n)\, \Big(1+\|R\|_{n}(x)\Big)^n\, (1+t)^{2n}\,  \e^{-t^2\sm(R(x))}
\end{equation}
for $(x,t)\in E\times \Rbb^{\geq 0}$. Note that (\ref{eq:maj-D-exp-R}) is still true
when $I=\emptyset$ with $\cst(0)=1$.

Now we look at  $\partial_I\cdot\left(\e^{-t^2R(x)+S(x)}\right)$ for
 $| I  |=n$. The
Volterra expansion formula gives $\e^{-t^2R(x)+S(x)}= \e^{-t^2R(x)}+
\sum_{k=1}^{\infty} \mathcal{Z}_k(x)$ with
$$
\mathcal{Z}_k(x)=\int_{\Delta_k}\e^{-s_1(t^2R(x))}S(x) \e^{-s_2(t^2R(x))}S(x)
\cdots S(x) \e^{-s_{k+1}(t^2 R(x))}ds_1\cdots ds_k.
$$
The term $\partial_I\cdot \mathcal{Z}_k(x)$ is  equal to the sum,
indexed by the partitions\footnote{We allow some of the $I_j$ to
be empty.} $\Pcal:=\{I_1,I_2,\ldots, I_{2k+1}\}$ of $I$, of the
terms
\begin{equation}\label{eq-Z-I}
\mathcal{Z}_k(\Pcal)(x):=
\end{equation}
$$
\int_{\Delta_k}\left(\partial_{I_1}\!\cdot\! \e^{-s_1(t^2R(x))}\right)
\left(\partial_{I_2}\!\cdot\! S(x) \right)\cdots
\left(\partial_{I_{2k}}\!\cdot \!S(x) \right)
\left(\partial_{I_{2k+1}}\!\cdot \!\e^{-s_{k+1}(t^2 R(x))}\right)
ds_1\cdots ds_k
$$
which are, thanks to (\ref{eq:maj-D-exp-R}), smaller in norm than
\begin{equation}\label{eq:maj-Z-k-P}
\cst(\Pcal)\,\Big(1+\|R\|_{n^+_\Pcal}(x)\Big)^{n^+_\Pcal}\,
\frac{\left(\|S\|_{n^-_\Pcal}(x)\right)^k}{k!}\,
(1+t)^{2n^+_\Pcal}\,  \e^{-t^2\sm(R(x))}.
\end{equation}
The integer $n^+_\Pcal,n_\Pcal^-$ are respectively  equal to the
sums $| I_1 | + | I_3| + \cdots +| I_{2k+1} |$, $| I_2| + | I_4| +
\cdots +| I_{2k} |$, and then $n^+_\Pcal +n_\Pcal^-=n$. The
constant $\cst(\Pcal)$ is equal to the products $\cst(| I_1
|)\cst(| I_3|)\cdots\cst(| I_{2k+1}|)$.  Since the sum $\sum_\Pcal
\cst(\Pcal)$ is bounded by a constant $\cst'(n)$, we find that
\begin{equation}\label{eq:maj-D-exp-R-S}
    \Big|\!\Big|\partial_I\cdot \e^{-t^2R(x)+S(x)}\Big|\!\Big|
\leq \cst'(n)\,\Big(1+\|R\|_{n}(x)\Big)^n \e^{\|S\|_{n}(x)} (1+t)^{2n}\,  \e^{-t^2\sm(R(x))}
\end{equation}
for $(x,t)\in E\times \Rbb^{\geq 0}$. Note that (\ref{eq:maj-D-exp-R-S}) is still true
when $I=\emptyset$ with $\cst'(0)=1$.

Finally we look at  $\partial_I\cdot\left(\e^{-t^2R(x)+S(x)+T(t,x)}\right)$ for
 $| I  |=n$. The
Volterra expansion formula gives $\e^{-t^2R(x)+S(x)+T(t,x)}=
\e^{-t^2R(x)+S(x)}+\sum_{k=1}^{q} \mathcal{W}_k(x)$ with
$$
\mathcal{W}_k(x)=\int_{\Delta_k}\e^{s_1(-t^2R(x)+S(x))}T(t,x)
\cdots T(t,x) \e^{s_{k+1}(-t^2 R(x)+S(x))}ds_1\cdots ds_k.
$$
Note that the term $ \mathcal{W}_k(x)$ vanishes for $k>q$. If we use
(\ref{eq:maj-D-exp-R-S}), we get for $(x,t)\in E\times \Rbb^{\geq 0}$ :
\begin{eqnarray*}
\lefteqn{\|\partial_I \cdot  \mathcal{W}_k(x)\|\leq
\cst''(n)\Big(\|T_0\|_{n}(x)+\|T_1\|_{n}(x)\Big)^k \times} \\
& &\Big(1+\|R\|_{n}(x)\Big)^n
\frac{(1+t)^{2n+k}}{k!}\,\e^{\|S\|_{n}(x)}\,  \e^{-t^2\sm(R(x))}.\nonumber
\end{eqnarray*}
Finally we get  for $(x,t)\in E\times \Rbb^{\geq 0}$ :
\begin{eqnarray}\label{eq:maj-D-exp-R-S-T}
\lefteqn{\quad \Big|\!\Big|\partial_I\cdot \e^{-t^2R(x)+S(x)+ T(t,x)}\Big|\!\Big|
\leq \cst''(n) \Big(1+\|R\|_{n}(x)\Big)^n\times }\\
& &\, {\rm P}\Big(\|T_0\|_{n}(x)+\|T_1\|_{n}(x)\Big)
\,\e^{\|S\|_{n}(x)}\,  (1+t)^{2n+q} \, \e^{-t^2\sm(R(x))}\nonumber
\end{eqnarray}
where ${\rm P}$ is the polynomial ${\rm P}(z)=\sum_{k=0}^q
\frac{z^k}{k!}$.

So (\ref{eq:maj-D-exp}) is proved with
$$
\cst= \cst''(n) \sup_{x\in \Kcal}
\left\{\Big(1+\|R\|_{n}(x)\Big)^n{\rm P}\Big(\|T_0\|_{n}(x)+\|T_1\|_{n}(x)\Big)
\,\e^{\|S\|_{n}(x)}\right\}.
$$

\end{proof}

\subsection{Second estimates}\label{subsec:appendix2}

Consider now the case where $E=W\times \kgot$ : the variable $x\in E$ will be replaced
by $(y,X)\in W\times \kgot$. {\bf We suppose that the maps $R$ and $T$ are constant
relatively to the parameter $X\in \kgot$.}

Let $\Kcal=\Kcal'\times\Kcal''$ be a compact
subset of $W\times\kgot$. Let $D(\partial)$ be a constant coefficient differential operator in
$(y,X)\in W\times \kgot$ of degree $r$ : let $r_W$ be its degree relatively to the variable $y\in W$.

\begin{prop}\label{prop-estimation-generale-X}
There exists a constant $\cst>0$, depending on $\Kcal,R(y),S(y,X)$,\break
$T_0(y),T_1(y)$ and $D(\partial)$,  such that\footnote{$q$
is the highest degree of the graded algebra $\Acal$.}
\begin{equation}\label{eq:maj-D-exp-bis}
    \Big|\!\Big|D(\partial)\cdot \e^{-t^2R(y)+S(y,X)+T(t,y)}\Big|\!\Big|
\leq \cst\, (1+t)^{2r_W+q}\,  \e^{-t^2\sm(R(y))},
\end{equation}
for all $(y,X,t)\in\Kcal'\times\Kcal''\times \Rbb^{\geq 0}$.
\end{prop}

\begin{proof} We follow the proof of Proposition \ref{prop-estimation-generale}. We have just to explain
why we can replace in (\ref{eq:maj-D-exp}) the factor $(1+t)^{2 r}$ by $(1+t)^{2 r_W}$.

We choose some basis $v_1,\ldots,v_{p_1}$ of $W$ and
 $X_1,\ldots,X_{p_2}$ of $\kgot$.  Let us denote
$\partial_i^1,\partial_j^2$ the partial derivatives along the vector $v_i$ and
$X_j$.  For any sequence
$$
I:=\underbrace{\{i_1,\ldots,i_{n_1}\}}_{I(1)}\cup \underbrace{\{
j_1,\ldots,j_{n_2}\}}_{I(2)}
$$
of integers where $i_k\in\{1,\ldots,p_1\}$ and
$j_k\in\{1,\ldots,p_2\}$, we denote $\partial_I$ the differential
operator of order $|I|=n_1+n_2$ defined by the product
$\prod_{k=1}^n \partial_{i_k}^1\prod_{l=1}^m \partial^2_{j_k}$.

We first notice that $\partial_I\cdot \e^{-t^2R(y)}=0$ if
$I(2)\neq\emptyset$. Now we look at  $\partial_I\cdot\left(\e^{-t^2R(y)+S(y,X)}\right)$ for
 $I=I(1)\cup I(2)$.  The term
 $\mathcal{Z}_k(\Pcal)$  of (\ref{eq-Z-I}) vanishes when there exists
 a subsequence $I_{2l+1}$ with $I_{2l+1}(2)\neq\emptyset$.
 In the other cases, the integer $n^+_\Pcal=| I_1 | + | I_3| + \cdots +| I_{2k+1}|$ appearing in
 (\ref{eq:maj-Z-k-P}) is smaller than
 $| I(1)  |=n_1$.  So the inequalities (\ref{eq:maj-D-exp-R-S}) and (\ref{eq:maj-D-exp-R-S-T})
hold with the factor
 $(1+t)^{2n}$ replaced by  $(1+t)^{2n_1}$.

\end{proof}

\bigskip

In order to prove Proposition \ref{estimatesgen}, we need to consider for every
compactly supported function $Q\in\f(\kgot)$ the integral
$$
J_Q(\xi,y,t):=\int_{\kgot}\e^{i\langle \xi,X\rangle
}\e^{-t^2R(y)+S(y,X)+T(t,y)}Q(X)dX.
$$

\begin{prop}\label{prop-estimation-generale-J-Q}

Let $\Kcal'\times \Kcal''$ be a compact subset of $W\times\kgot$, and let $p$ be any positive integer.

$\bullet$ There exists a constant $\cst>0$, such that: for any
function $Q\in\f(\kgot)$ with support on $\Kcal''$, we have
$$
\Big|\!\Big|J_Q(\xi,y,t)\Big|\!\Big|\leq
\cst \, \|Q\|_{\Kcal'',2p}\, \frac{(1+t)^q}{(1+\|\xi\|^2)^p} \e^{-t^2\sm(R(y))}.
$$
for all $(\xi,y,t)\in \kgot^*\times\Kcal'\times \Rbb^{\geq 0}$.

$\bullet$ Let $D(\partial_y)$ be a constant differential operator
on $W$ of order $r$. There exists a constant $\cst>0$, such that:
for any function $Q\in\f(\kgot)$ with support on $\Kcal''$, we
have
\begin{equation}\label{eq:estimate:D-fourier}
\Big|\!\Big|D(\partial_y)\cdot J_Q(\xi,y,t)\Big|\!\Big|\leq
\cst \, \|Q\|_{\Kcal'',2p}\,
\frac{(1+t)^{q+2r}}{(1+\|\xi\|^2)^p} \e^{-t^2\sm(R(y))}
\end{equation}
for all $(\xi,y,t)\in\kgot^*\times\Kcal'\times \Rbb^{\geq 0}$.
\end{prop}

\begin{proof} Let us concentrate on the first point. We have
\begin{eqnarray*}
(1+\|\xi\|^2)^p J_Q(\xi,y,t) &=&
\int_{\kgot}\left(D_{2p}(\partial_X) \cdot \e^{i\langle
\xi,X\rangle}\right) \e^{-t^2R(y)+S(y,X)+T(t,y)}Q(X)dX \\
&=& \int_{\kgot}\e^{i\langle \xi,X\rangle}
D_{2p}(\partial_X)\cdot\left(\e^{-t^2R(y)+S(y,X)+T(t,y)}Q(X)\right)dX
\end{eqnarray*}
where $D_{2p}(\partial_X)=(1-\sum_a(\partial_{X_a})^2)^p$ is a constant
coefficients differential operator in $X$ with order equal to $2p$.
Now $D_{2p}(\partial_X)\cdot\left(\e^{-t^2R(y)+S(y,X)+T(t,y)} Q(X)\right)$ is a
finite sum of terms $\partial_X^{\alpha}\cdot\left(\e^{-t^2R(y)+S(y,X)+T(t,y)}
\right)(\partial_X^{\beta}\cdot Q(X))$ with $|\alpha|$ and $|\beta|$
less or equal than $2p$. All the
derivatives $\partial_X^{\beta}\cdot Q(X)$ are bounded by $\|Q\|_{\Kcal_2,p}$.
 We now employ the estimate of Proposition \ref{prop-estimation-generale-X} for
$\partial_X^{\alpha}\cdot\left(\e^{-t^2R(y)+S(y,X)+T(t,y)}\right)$, where
$\|\alpha\|\leq 2p$, and $(y,X)\in\Kcal'\times\Kcal''$, and we obtain our estimate.

The second point works similarly. We need to estimate
$\partial_X^{\alpha}D(\partial_y)\cdot$\break
 $\left(\e^{-t^2R(y)+S(y,X)+T(t,y)}\right)$.
We use also the estimate of Proposition \ref{prop-estimation-generale-X}.
\end{proof}

\bigskip

 \begin{rem}\label{rem:estimate-Q-general}
The estimate of Proposition \ref{prop-estimation-generale-J-Q}
still holds when $Q$ is a smooth map (with compact support) from
$\kgot$ with values in $\End(V)\otimes\Acal$.
\end{rem}

\bigskip

In fact, we need still a slightly more general situation. The
proof is identical to the preceding proof.

Let us denote $\F(y,X,t):= -t^2 R(y)+S(y,X)+T(t,y)$. Let
$U_1(y), U_2(y)$ be two smooth maps with values in $\End(V)\otimes\Acal$.
For any smooth function $Q$ on $\kgot$ with compact support, we consider
the integral
$$
\mathcal{I}_Q(u,t,y,\xi):=\int_{\kgot}\e^{i\langle \xi,X\rangle}
U_1(y)\e^{u\F(y,X,t)} U_2(y) \e^{(1-u)\F(y,X,t)} Q(X)dX.
$$


\begin{prop}\label{suffitgendouble}
Let $p$ be any positive integer. Let $\Kcal'\times\Kcal''$ be a
compact subset of $W\times\kgot$. Let $D(\partial_{y})$ be a
constant differential operator on $W$ of order $r$. There exists
$\cst>0$ such that: for any function $Q\in\f(\kgot)$ with support
in the compact $\Kcal''$ we have
$$
\Big|\!\Big|D(\partial_{y})\cdot \mathcal{I}_Q(u,t,y,\xi)\Big|\!\Big|
\leq
 \cst \, \|Q\|_{\Kcal_2,2p}\,
\frac{(1+t)^{q+2r}}{(1+\|\xi\|^2)^p} \e^{-t^2\sm(R(y))}
$$
for all $(t,y,\xi)\in\Rbb^{\geq 0}\times\Kcal'\times\kgot^* $, and $u\in [0,1]$.
\end{prop}

\subsection{Third estimates}\label{subsec:appendix3}

In order to prove Theorem \ref{theo:ch-PV-BV},  we need to
consider the following setting:

$\bullet$ The vector space $W$ decomposes as $W=W_1\times W_2$,

$\bullet$ The maps $R$ and $T$ do not depend of the variable $X\in\kgot$,

$\bullet$ The map $S$ does not depend of the variable $y_2\in W_2$,

$\bullet$ The maps $R$, $T_0$ and $T_1$ are \emph{slowly
increasing} relatively to the variable $y_2\in W_2$. Let us recall
the definition. For any integer $n$, and any compact subset
$\Kcal_1$ of $W_1$, there exist some positive constants $\cst,\mu$
such that each function $\| R\|_n(y_1,y_2)$, $\| T_0\|_n(y_1,y_2)$
and $\| T_1\|_n(y_1,y_2)$ is bounded by $\cst (1+\|y_2\|)^\mu$
on $\Kcal_1\times W_2$.

\medskip

Here we  consider for every
compactly supported function $Q\in\f(\kgot)$ the integral
$$
J_Q(\xi,y_1,y_2,t):=\int_{\kgot}\e^{i\langle \xi,X\rangle
}\e^{-t^2R(y_1,y_2)+S(y_1,X)+T(t,y_1,y_2)}Q(X)dX.
$$

\begin{prop}\label{prop-estimation-generale-slowly}
Let $\Kcal_1\times \Kcal''$ be a compact subset of
$W_1\times\kgot$. Let $D(\partial_y)$ be a constant differential
operator on $W$ of order $r_W$.

There exists a constant $\mu\geq 0$, such that for any positive
integer $p$, there exists $\cst>0$ for which the following
estimate holds for any function $Q\in\f(\kgot)$ with support on
$\Kcal''$ :
\begin{eqnarray}\label{eq:estimate:D-fourier-slowly}
\lefteqn{\Big|\!\Big|D(\partial_y)\cdot J_Q(\xi,y_1,y_2,t)\Big|\!\Big|\leq}\\
& & \cst \, \|Q\|_{\Kcal''\!,2p}\,
\frac{(1+\|y_2\|)^\mu}{(1+\|\xi\|^2)^p} \, (1+t)^{q+2r_W}\,
\e^{-t^2\sm(R(y_1,y_2))}\nonumber
\end{eqnarray}
for $(\xi,y_1,y_2,t)\in\kgot^*\times\Kcal'\times W_2\times \Rbb^{\geq 0}$.
\end{prop}

\begin{rem} In the estimate (\ref{eq:estimate:D-fourier-slowly}), the crucial point is that
the constant $\mu$ is the same for all integer $p$.
\end{rem}

\begin{proof} In the following the parameter $(y_1,X)$ belongs to the compact
$\Kcal:=\Kcal'_1\times \Kcal''$ and the parameter $(y_2,t)$ belongs to $W_2\times \Rbb^{\geq 0}$.

As in the proof of Proposition
\ref{prop-estimation-generale-J-Q}, we get the estimates
(\ref{eq:estimate:D-fourier-slowly}) if we show that for any
differential operator $D(\partial_X)$ we have the estimate:
\begin{eqnarray}\label{eq:estimate-slowly}
\lefteqn{\Big|\!\Big| D(\partial_y)\circ D(\partial_X)\cdot
\e^{-t^2R(y_1,y_2)+S(y_1,X)+T(t,y_1,y_2)}\Big|\!\Big|\leq}\\
& &\cst \, (1+\|y_2\|)^\mu \, (1+t)^{q+2r_W}\, \e^{-t^2\sm(R(y_1,y_2))},\nonumber
\end{eqnarray}
where the parameter $\mu$ in (\ref{eq:estimate-slowly}) does not depend of the choice
of  $D(\partial_X)$.

First, we consider the term $D(\partial_y)\circ D(\partial_X)\cdot
\e^{-t^2R(y_1,y_2)}$ : it vanishes if the order of $D(\partial_X)$
is not zero. In the other case we exploit (\ref{eq:maj-D-exp-R})
with the \emph{slowly increasing} behavior of $R$ to get
$$
\Big|\!\Big| D(\partial_y)\cdot
\e^{-t^2R(y_1,y_2)}\Big|\!\Big|\leq \cst \, (1+\|y_2\|)^\alpha \, (1+t)^{2r_W}\, \e^{-t^2\sm(R(y_1,y_2))},
$$
where $\alpha$ depends of the order $D(\partial_y)$.

Now we consider the term $D(\partial_y)\circ D(\partial_X)\cdot
\e^{-t^2R(y_1,y_2)+S(y_1,X)}$. The estimate (\ref{eq:maj-Z-k-P}) gives, modulo the changes explained
in the proof of Proposition \ref{prop-estimation-generale-X}, the following
\begin{eqnarray}\label{eq:estimate-slowly-2}
\lefteqn{\Big|\!\Big| D(\partial_y)\circ D(\partial_X)\cdot
\e^{-t^2R(y_1,y_2)+S(y_1,X)}\Big|\!\Big|\leq}\\
& &\cst \, (1+\|y_2\|)^\delta \, (1+t)^{2r_W}\, \e^{-t^2\sm(R(y_1,y_2))}\nonumber
\end{eqnarray}
where $\cst$ take into account the term\footnote{$n$ is the order of $D(\partial_y)\circ D(\partial_X)$}
$\e^{\| S\|_{\Kcal,n}}$. Here the term $(1+\|y_2\|)^\delta$ comes from the term
$(1+\|R\|_{n^+_\Pcal}(x))^{n^+_\Pcal}$ of (\ref{eq:maj-Z-k-P}). The factor $n^+_\Pcal$ is bounded by
the order of $D(\partial_y)$,  hence it explains why $\delta$ does not depend of the order of $D(\partial_X)$.

Using  Volterra expansion formula, it is now an easy matter to
derive (\ref{eq:estimate-slowly}) from
(\ref{eq:estimate-slowly-2}).

\end{proof}

\bigskip

The preceding estimates hold if we work in the algebra
$\End(\Ecal)\otimes \Acal$, where $\Ecal$ is a super-vector space
and $\Acal$ a super-commutative algebra.

\end{document}